\documentclass[12pt, a4]{amsart}

\usepackage{latexsym,amsmath,amssymb,amsthm,mathrsfs,color}
\usepackage[bbgreekl]{mathbbol}
\usepackage{bbm}
\usepackage{tikz-cd}
\usepackage[all]{xy}
\usepackage{stmaryrd}
\usepackage{hyperref} 
\hypersetup{colorlinks,%
    citecolor=brown,%
    filecolor=black,%
   linkcolor=blue,%
   urlcolor=black}

\usepackage[shortalphabetic]{amsrefs}
 \makeatletter
\def\append@label@year@{%
    \safe@set\@tempcnta\bib@year
    \edef\bib@citeyear{\the\@tempcnta}%
    \ifnum\bib@citeyear>9
      \append@to@stem{%
          \ifx\bib@year\@empty
          \else
            \@xp\year@short \bib@citeyear \@nil
          \fi
      }%
    \fi    
}
\makeatother

\let\oldtocsection=\tocsection
\renewcommand{\tocsection}[2]{\hspace{0em}\oldtocsection{#1}{#2}}

\usepackage[a4paper]{geometry}

\def\upddots{\mathinner{\mkern 1mu\raise 1pt \hbox{.}\mkern 2mu
\mkern 2mu \raise 4pt\hbox{.}\mkern 1mu \raise 7pt\vbox {\kern 7
pt\hbox{.}}} }

\begin{document}
   
\setlength{\unitlength}{2.5cm}

\newtheorem{thm}{Theorem}[section]
\newtheorem{lm}[thm]{Lemma}
\newtheorem{prop}[thm]{Proposition}
\newtheorem{cor}[thm]{Corollary}
\newtheorem{conj}[thm]{Conjecture}

\theoremstyle{definition}
\newtheorem{dfn}[thm]{Definition}
\newtheorem{eg}[thm]{Example}
\newtheorem{rmk}[thm]{Remark}

\newcommand{\F}{\mathbf{F}}
\newcommand{\N}{\mathbf{N}}
\newcommand{\R}{\mathbf{R}}
\newcommand{\C}{\mathbf{C}}
\newcommand{\Z}{\mathbf{Z}}
\newcommand{\Q}{\mathbf{Q}}

\newcommand{\Mp}{{\rm Mp}}
\newcommand{\Sp}{{\rm Sp}}
\newcommand{\GL}{{\rm GL}}
\newcommand{\PGL}{{\rm PGL}}
\newcommand{\SL}{{\rm SL}}
\newcommand{\SO}{{\rm SO}}
\newcommand{\Spin}{\text{Spin}}
\newcommand{\Ind}{{\rm Ind}}
\newcommand{\Res}{{\rm Res}}
\newcommand{\Hom}{{\rm Hom}}
\newcommand{\msc}[1]{\mathscr{#1}}
\newcommand{\mfr}[1]{\mathfrak{#1}}
\newcommand{\mca}[1]{\mathcal{#1}}
\newcommand{\mbf}[1]{{\bf #1}}
\newcommand{\mbm}[1]{\mathbbm{#1}}

\newcommand{\into}{\hookrightarrow}
\newcommand{\onto}{\twoheadrightarrow}

\newcommand{\s}{\mathbf{s}}
\newcommand{\cc}{\mathbf{c}}
\newcommand{\bfa}{\mathbf{a}}
\newcommand{\id}{{\rm id}}
\newcommand{\g}{\mathbf{g}_{\psi^{-1}}}
\newcommand{\w}{\mathbbm{w}}
\newcommand{\Ftn}{{\rm Ftn}}
\newcommand{\p}{\mathbf{p}}
\newcommand{\bq}{\mathbf{q}}
\newcommand{\WD}{\text{WD}}
\newcommand{\W}{\text{W}}
\newcommand{\Wh}{{\rm Wh}_\psi}
\newcommand{\ggma}{\pmb{\gamma}}
\newcommand{\sct}{\text{\rm sc}}
\newcommand{\OF}{\mca{O}^\digamma}
\newcommand{\vep}{\xi}
\newcommand{\gk}{c_{\sf gk}}
\newcommand{\Ima}{{\rm Im}}
\newcommand{\Ker}{{\rm Ker}}
\newcommand{\lrho}{ {}^L\rho }

\newcommand{\fwchi}{{\chi^\flat}}

\newcommand{\cu}[1]{\textsc{\underline{#1}}}
\newcommand{\set}[1]{\left\{#1\right\}}
\newcommand{\ul}[1]{\underline{#1}}
\newcommand{\wt}[1]{\overline{#1}}
\newcommand{\wtsf}[1]{\wt{\sf #1}}
\newcommand{\angb}[2]{\left\langle #1, #2 \right\rangle}
\newcommand{\wm}[1]{\wt{\mbf{#1}}}
\newcommand{\elt}[1]{\pmb{\big[} #1\pmb{\big]} }
\newcommand{\ceil}[1]{\left\lceil #1 \right\rceil}
\newcommand{\val}[1]{\left| #1 \right|}
\newcommand{\twchi}{ \wt{\chi} }
\newcommand{\JH}{ {\rm JH} }


\title[Kazhdan-Lusztig representations and Whittaker space]{Kazhdan-Lusztig representations and Whittaker space of some genuine representations}

\author{Fan Gao}
\address{School of Mathematical Sciences, Yuquan Campus, Zhejiang University, 38 Zheda Road, Hangzhou, China 310027}
\email{gaofan.math@gmail.com}

\maketitle


\begin{abstract} 
We prove a formula for the dimension of Whittaker functionals of irreducible constituents of a regular unramified genuine principal series for covering groups. The formula explicitly relates such dimension to the Kazhdan-Lusztig representations associated with certain right cells of the Weyl group. We also state a refined version of the formula, which is proved under some natural assumption. The refined formula is also verified unconditionally in several important cases.
\end{abstract}
\tableofcontents

\section{Introduction}
In their seminal paper \cite{KL1}, Kazhdan and Lusztig constructed a new basis $\set{C_w}$ for the Hecke algebra of a Coxeter group $W$, which in particular includes the case of a finite Weyl group or an affine Weyl group. Such a basis has some remarkable properties and carries much significance for the representation theory of the Hecke algebra. What is embodied in the description of $\set{C_w}$ is the family of Kazhdan-Lusztig polynomials, which have also been studied extensively in literature.

Moreover, Kazhdan and Lusztig introduced the notion of right, left and two-sided cells of $W$  and showed that, by using the basis $\set{C_w}$, there is a natural representation $\sigma_\mfr{C}$ of $W$ associated to a right cell $\mfr{C}$ (or left cell). In this paper, $W$ is the Weyl group associated to a coroot system, and we call the representation $\sigma_\mfr{C}$ attached to a right cell  $\mfr{C} \subset W$ a Kazhdan-Lusztig representation. In general, the representation $\sigma_\mfr{C}$ may not be irreducible. However, the regular right representation $\C[W]$ of $W$ decomposes into the $\sigma_\mfr{C}$'s, which reflects the decomposition of $W$ into disjoint union of right cells.

In fact, the concept of cells was introduced before the work of Kazhdan and Lusztig. Motivated from combinatorics, Robinson and Schensted (see \cite{Rob1,Rob2,Rob3,Schen})  showed that there is a bijection between the symmetric group $S_r$ (the Weyl group for $\GL_r$) and the set of pairs $(P,Q)$ of Young tableaux of the same shape and size. A left cell is thus just the set $\set{(P,Q)}$ with $P$ fixed. For a general Weyl group $W$, Joseph \cite{Jos2} introduced a notion of  left cells in terms of primitive ideals in the enveloping algebra of a complex semisimple Lie algebra. His definition of a left cell and the representation of $W$ that the cell bears is motivated from understanding the decomposition of the Verma modules. The notion of cells of Weyl groups by Joseph, which is a priori different from that of Kazhdan-Lusztig \cite{KL1}, is actually shown to agree. This  has a striking consequence of connecting the two theories. 

 These geometric structures (left, right or double cells) carry intriguing information for the representation of the Weyl group. For instance, it is conjectured \cite[\S 1.7]{KL1} that for every right cell $\mfr{C}$ the representation $\sigma_\mfr{C}$ has a unique ``special" irreducible component, which can be characterised by the equality of the leading monomial of its fake-degree and generic-degree polynomials (see \cite[Chapter 5]{Lus4} or \cite[Chapter 11.3]{Car}). Moreover, $\sigma_\mfr{C}$ and $\sigma_{\mfr{C}'}$ contain  the same special representation if and only if $\mfr{C}$ and $\mfr{C}'$ lie in a common two-sided cell.  This conjecture, which is a theorem by the work of Barbasch and Vogan (see \cite{BV1, BV2}), provides deep link between the geometry of cells and the theory of primitive ideals, the latter of which governs important aspects of the representation of reductive group.  The notion of ``fake degree" mentioned above is also linked to the representation of $W$ on its coinvariant algebra, and thus to the more classical Coxeter geometry along the line of the work of Bernstein, Gelfand and Gelfand \cite{BGG} (see also \cite{Hil}). 
 

In general, the Kazhdan-Lusztig representations and polynomials manifest with much significance in the representation theory of algebraic groups over either finite field or local field. 
Indeed, Hecke algebra arises naturally in the study of decomposing a principal series representation for an algebraic group over a finite field. This is one of the motivations for Lusztig in his relevant work \cite{Lus4}. Moreover, when the Coxeter group is an affine Weyl group, the associated Iwahori-Hecke algebra captures the representations of an algebraic group over a non-archimedean field with Iwahori-fixed vector (see \cite{Bor76}). The connection of this with the Kazhdan-Lusztig theory is exploited in \cite{Lus2, Rog, Ree1}. For instance, by using the Kazhdan-Lusztig representations, Lusztig \cite{Lus2} constructed certain square integrable representations of a simple $p$-adic group.  On the other hand, the work of Rogawski  \cite{Rog} provides an interesting way of modelling intertwining operators on Iwahori-Hecke algebra, where the Kazhdan-Lusztig theory plays a pivotal role.

Kazhdan-Lusztig polynomials and their analogues also appear intriguingly in the recent work of Bump and Nakasuji \cite{BN1, BN2}. Here, the problem arises from relating two generating basis for the vector space of Iwahori-fixed vectors in a principal series representation. On the one hand, one has the natural basis from considering functions with support in the Schubert cells; on the other hand, there is the Casselman's basis arising from considering intertwining operators. It is an important problem of finding the transition matrix between such two basis. In loc. cit., the Kazhdan-Lusztig polynomials alike provide a realization of such a transition matrix. Recently, the conjectures posed by Bump and Nakasuji have been solved by geometric methods in \cite{AMSS}.

\subsection{The main conjecture and results}
With the above quick and selective review, it is perhaps superfluous to restate the intimacy of the Kazhdan-Lusztig theory with representation of algebraic groups. Nevertheless, the goal of our paper is to point out a link between the Kazhdan-Lusztig representations with the Whittaker space of genuine representations of finite degree central covering of a linear algebraic group. 

It is a well-known result by Gelfand-Kazhdan, Shalika and Rodier (see \cite{GK, Shal, Rod1}) that for a linear algebraic group $G$ over a local field, the space $\Wh(\pi)$ of Whittaker functionals of any irreducible representation $\pi \in \text{Irr}(G)$ has dimension always bounded above by one, which is often referred to as the multiplicity-one property. On the other hand, uniqueness of Whittaker functionals fails for genuine representations of degree $n$ central covering group $\wt{G}$ of $G$:
$$\begin{tikzcd}
\bbmu_n \ar[r, hook]  & \wt{G}  \ar[r, two heads]  & G,
\end{tikzcd}$$
where $\bbmu_n$ acts via a fixed embedding $\epsilon: \bbmu_n \into \C^\times$. Such high-multiplicity has both hindered and enriched the development of the representation theory of a covering group, see the historical review \cite{GGW} and the introduction of \cite{GSS1, GSS2}. At the moment, it seems to be an insurmountable task to determine completely the space $\Wh(\pi)$. In fact, it is already a difficult problem to compute $\dim \Wh(\pi)$ for a general genuine representation $\pi \in \text{Irr}_\epsilon(G)$, or equivalently, to describe the group homomorphism
$$\begin{tikzcd}
\dim \Wh(-): \  \msc{R}(\wt{G})  \ar[r]  & \Z,
\end{tikzcd}$$
where $\msc{R}(\wt{G})$ denotes the Grothendieck group of the $\epsilon$-genuine irreducible representations $\text{Irr}_\epsilon(\wt{G})$ of $\wt{G}$. In literature, the most widely studied family consists of the so-called theta representation $\Theta(\wt{G}, \chi)$. For Kazhdan-Patterson coverings $\wt{\GL}_r$, the dimension of $\Wh(\Theta(\wt{\GL}_r, \chi))$ was first studied in \cite{KP}; their results are generalized to Brylinski-Deligne covering group of a linear reductive group in \cite{Ga2}. Even for a theta representation, the dimension is not given by an elementary formula: the underlying group involved, the degree of covering and the defining data for a theta representation all play sensitive roles. 

In this paper, we consider irreducible constituent of a regular unramified genuine principal series represntation of $\wt{G}$ and propose a formula for the dimension of its Whittaker space in terms of the associated Kazhdan-Lusztig representation. The formula is stated in Conjecture \ref{C:SW} below. The conjecture applies in particular to a class of covering groups called persistent, see Definition \ref{D:per}. Such a constraint  is expected, as for example the degree-$n$ cover of $\wt{\SL}_2^{(n)}$ is persistent if and only if $n\ne 2k$ with $k$ odd; otherwise, $\dim \Wh( \Theta(\wt{\SL}_2^{(n)}, \chi))$ for a theta representation of $\wt{\SL}_2^{(n)}$ depends sensitively on $\chi$.

Now we give a brief outline of the paper and state our main results.
\vskip 10pt

In section \S \ref{S:cg}, we introduce covering groups following the Brylinski-Deligne framework \cite{BD}. In particular, we recall some notations and notions from \cite{We6, GG}. A summary of some structural facts of $\wt{G}$ is also given. 

In section \S \ref{S:red-ps}, a regular unramified genuine principal series $I(\chi)$ of $\wt{G}$ is considered and we show how to adapt the argument of Rodier in \cite{Rod4} to give a classification of the irreducible components of $I(\chi)$. In particular, the reducibility of $I(\chi)$ is controlled by a certain subset $\Phi(\chi) \subset \Phi$ of the roots $\Phi$ of $G$. The main result is Theorem \ref{T:Rodier}. As the proof largely follows closely that of Rodier, we will only highlight the key difference and ingredients used in the covering setting.

In section \S \ref{S:WhGa}, we carry out some preliminary study of the Whittaker space $\Wh(\pi_\Gamma)$ of an irreducible constituent $\pi_\Gamma$ of $I(\chi)$. For an unramified genuine character $\chi$ of $Z(\wt{T})$, one has
$$\dim \Wh(I(\chi))= \val{\msc{X}_{Q,n}},$$
where $\msc{X}_{Q,n}:=Y/Y_{Q,n}$ is the ``moduli space"  of Whittaker functionals for the unramified principal series $I(\chi)$. Here $Y$ is the cocharacter lattice of $G$ and $Y_{Q,n} \subset Y$ a sublattice. The set $\msc{X}_{Q,n}$ carries a natural twisted Weyl group action denoted by $w[y]$. In particular, one has a natural permutation representation
$$\begin{tikzcd}
\sigma_\msc{X}: \ W  \ar[r]  & \text{Perm}(\msc{X}_{Q,n})
\end{tikzcd}$$
given by $\sigma_\msc{X}(w)(y)= w[y]$. For every Weyl-orbit $\mca{O}_y \subset \msc{X}_{Q,n}$, there is also the permutation representation 
$$\sigma_{\msc{X}}^y: W \longrightarrow \text{Perm}(\mca{O}_y),$$
as a constituent of $\sigma_\msc{X}$. We study the scattering matrix for intertwining operators between principal series; as a consequence, we have a decomposition of 
$\Wh(\pi_\Gamma)$ into various subspaces $\Wh(\pi_\Gamma)_{\mca{O}_y}$ associated to Weyl-orbits $\mca{O}_y \subset \msc{X}_{Q,n}$. One expects an $\mca{O}_y$-version of the exactness of the function $\pi \mapsto \dim \Wh(\pi)_{\mca{O}_y}$, which is stated as Conjecture \ref{C:O-ex}. The analysis in \S \ref{S:WhGa} culminates to a coarse formula for $\dim \Wh(\pi_\Gamma)_{\mca{O}_y}$, see Proposition \ref{P:c-dim}. Moreover, Conjecture \ref{C:O-ex} is verified if $\val{\Phi(\chi)} \le 2$.

In section \S \ref{S:KLcon}, we first introduce the Kazhdan-Lusztig theory for a Weyl group. We will not give any extensive exposition but be content with giving the minimally necessary presentation for our purpose. If $\Phi(\chi) \subseteq \Delta$, then to every $\pi_\Gamma$ there is a naturally associated representation $\sigma_\Gamma$ of $W$, as a sum of certain Kazhdan-Lusztig representations. The conjectural formula directly equates the dimension $\Wh(\pi_\Gamma)_{\mca{O}_y}$ (resp. $\dim \Wh(\pi_\Gamma)$ ) to the pairing of $\sigma_\msc{X}^y$ (resp. $\sigma_{\msc{X}}$) against $\sigma_\Gamma$:

\begin{conj}[Conjecture \ref{C:S}]  \label{C:SW}
Let $\chi$ be a regular unramified genuine character of $Z(\wt{T})$ such that $\Phi(\chi) \subset \Delta$. Then for every irreducible constituent $\pi_\Gamma$ of $I(\chi)$ and every persistent $W$-orbit $\mca{O}_y$ in $\msc{X}_{Q,n}$ (see Definition \ref{D:per}), one has
\begin{equation} \label{E01}
 \dim \Wh(\pi_\Gamma)_{\mca{O}_y}= \angb{ \sigma_\msc{X}^y }{ \sigma_\Gamma },
 \end{equation}
where the pairing denotes the inner product of the two representations of $W$.  In particular, if $\wt{G}$ is persistent (see Definition \ref{D:per}), then the above equality holds for every orbit $\mca{O}_y \subseteq \msc{X}_{Q,n}$, and consequently,
\begin{equation} \label{E02}
\dim \Wh(\pi_\Gamma)= \angb{ \sigma_\msc{X} }{ \sigma_\Gamma }.
\end{equation}
\end{conj} 
The remaining of \S \ref{S:KLcon} and \S \ref{S:ThSt}--\S \ref{S:Gamma+} are devoted to proving various cases of Conjecture \ref{C:SW}.

First, immediately after Conjecture \ref{C:S}, we consider the case  when $\mca{O}_y=\set{y} \subset \msc{X}_{Q,n}$ is a singleton, i.e., $y\in (\msc{X}_{Q,n})^W$. We have

\begin{thm}[{Theorem \ref{T:dis}}]
If $\mca{O}_y=\set{y} \subset (\msc{X}_{Q,n})^W$ is persistent, then Conjecture \ref{C:SW} holds, i.e.,
$$\dim \Wh(\pi_\Gamma)_{\mca{O}_y}= \angb{ \sigma_\msc{X}^y }{ \sigma_\Gamma }$$
for every $\pi_\Gamma$.
\end{thm}
\noindent This recovers Rodier's result in \cite{Rod4} when $\wt{G}=G$ is a linear group, which is always saturated and thus persistent (see Definition \ref{D:sat} and Lemma \ref{Satu}).

Second, the main result amalgamated from \S \ref{S:ThSt}--\S \ref{S:Gamma+} (especially Theorem \ref{T:G-pm} and Theorem \ref{T:St-O})  is as follows:

\begin{thm}  \label{T:summa}
Retain the same notation and assumption on $\chi$ as in Conjecture \ref{C:SW}. Then we have
\begin{enumerate}
\item[(i)] equality \eqref{E01} holds for $\pi_{\Gamma^\pm}$ and every persistent orbit $\mca{O}_y$;
\item[(ii)] if Conjecture \ref{C:O-ex} holds, then \eqref{E01} holds for every $\pi_\Gamma \in \JH(I(\chi))$ and every persistent orbit $\mca{O}_y$;
\item[(iii)] if $\wt{G}$ is persistent, then \eqref{E02} holds for every $\pi_\Gamma \in \JH(I(\chi))$.
\end{enumerate}
\end{thm}
Here Theorem \ref{T:summa} (i) and (iii) are proven unconditionally (i.e., independent of Conjecture \ref{C:O-ex}) and thus constitute the main results in our paper. In particular, Theorem \ref{T:summa} (iii) is a much wider generalization of the main results in \cite{Ga2}, which only deals with $\dim \Wh(\pi_{\Gamma-})$ for $\Phi(\chi) =\Delta$.

The proof of Theorem \ref{T:summa} (i) is achieved in two steps:
\begin{enumerate}
\item[(S1)] Prove the equality $\dim \Wh(\pi_{\Gamma^\pm})_{\mca{O}_y} = \angb{ \sigma_{\msc{X}}^y }{ \sigma_{\Gamma^\pm} }$ under the assumption $\Phi(\chi) =\Delta$. The argument relies on our earlier work \cite{Ga2}, see Proposition \ref{T:C-} and Proposition \ref{T:C+}.  Note that in this case, $\pi_{\Gamma^-}$ is the unique Langlands quotient of $I(\chi)$, the so-called theta representation. On the other hand, $\pi_{\Gamma^-} \subset I(\chi)$ is the unique subrepresentation, and is the covering analogue of the Steinberg representation. One has $\sigma_{\Gamma^-} = \varepsilon_W$ and $\sigma_{\Gamma^+} = \mbm{1}_W$.
\item[(S2)] For the general case $\Phi(\chi) \subset \Delta$, we apply a reduction to the two sides of the desired equality
$$\dim \Wh(\pi_{\Gamma^\pm}) = \angb{ \sigma_{\msc{X}}^y }{ \sigma_{\Gamma^\pm} },$$
and invoke results in (S1). More precisely, what applies to the left hand side is a form of Rodier's heredity, and to the right hand side is the induction property of the Kazhdan-Lusztig representation $\sigma_{\Gamma^\pm}$. For details, see the argument in Theorem \ref{T:G-pm}.
\end{enumerate}

The proof of Theorem \ref{T:summa} (ii) starts from the inclusion-exclusion principle (see Lemma \ref{L:key}) of relating $\pi_\Gamma$ to an alternating sum of parabolically-induced representations from the theta representations on certain Levi subgroups. The exactness of the function $\dim \Wh(-)_{\mca{O}_y}$, if we assume Conjecture \ref{C:O-ex}, coupled with similar results in (i) prove (ii). The argument for (iii) is similar to that of (ii), and is unconditional as we have the exactness of $\dim \Wh(-)$.

In the last section \S \ref{S:3eg}, we provide numerical illustration for Theorem \ref{T:summa} by considering covers of $\SL_3, \Sp_4$ and the exceptional group ${\rm G}_2$. 

\subsection{Consequence and some remarks}
There are several immediate observations or remarks from the results above.
\begin{enumerate}
\item[(R1)] 
It is desirable to have a natural parametrization of the space $\Wh(\pi_\Gamma)$ for every constituent $\pi_\Gamma$, beyond merely determining the dimension. For $\Gamma^\pm$, this is carried out explicitly in the constructive proof of Theorem \ref{T:summa} (i). More precisely, if $\wt{G}$ is persistent with $\Phi(\chi) \subset \Delta$, then $\Wh(\pi_{\Gamma^-})$ is essentially parametrized by the free $W(\Phi(\chi))$-orbits in $\msc{X}_{Q,n}$, while $\Wh(\pi_{\Gamma^+})$ by the $W(\Phi(\chi))$-orbits in $\msc{X}_{Q,n}$. However, for general $\pi_\Gamma$, we do not know a similar simple recipe of describing $\Wh(\pi_\Gamma)$. 
\item[(R2)] 
Several results on linear algebraic groups are generalized in the covering setting. For example, for linear algebraic $G$, the standard module conjecture \cite{CasSha}, which is proved in \cite{HeMu, HeOp}, asserts that the Langlands quotient of a standard module is the least generic representation among the irreducible constituents. We show in Corollary \ref{C:lb} that one has an analogue of the standard module conjecture in the very restricted setting in our paper,  
namely, $\pi_{\Gamma^-}$ is the least generic constituent among all $\pi_\Gamma \in \JH(I(\chi))$. In fact, we believe that the theta representation $\pi_{\mca{C}^-}$ is the least generic representation among all irreducible constituents of an unramified principal series of $\wt{G}$. See Conjecture \ref{C:theta-mini}.
\item[(R3)] Theorem \ref{T:summa} also describes new phenomenon which exists only for covering groups, especially when the degree of covering is large enough. For example, as discussed in Proposition \ref{P:asymF} and Remark \ref{R:ub}, the irreducible subrepresentation of a standard module might not be the most generic constituent. In particular, if $\Phi(\chi)=\Delta$, then the covering Steinberg representation $\pi_{\Gamma^+}$ is always generic; however, it is possible that $\dim \Wh(\pi_{\Gamma^+}) < \dim \Wh(\pi_\Gamma)$ for some other $\pi_{\Gamma}$. This is in contrast with the linear algebraic case.
\item[(R4)] As a byproduct of the proof of Theorem \ref{T:summa}, we also obtain a refinement of Ginzburg's conjecture \cite[page 448]{Gin4} (in the special case of regular principal series with $\Phi(\chi) \subset \Delta$) on non-generic unramified representation of a covering group.  See 
Remark \ref{Gin-ref}.
\end{enumerate}

We hope that results in our paper also provide a preliminary step towards understanding the arithmetic arising from the endomorphism (see \cite[\S 3.2]{GSS2})
$$\mca{T}(w, \pi)^*:  \Wh(I_{\wt{P}}^{\wt{G}} \ \pi)  \to  \Wh(I_{\wt{P}}^{\wt{G}} \ \pi),$$
where $\pi$ is an irreducible genuine representation of the Levi subgroup $\wt{M}$ of $\wt{P}$. In the case $\wt{P} =\wt{B}$, the two invariants trace and determinant of $\mca{T}(w, \pi)^*$ are investigated in \cite{GSS2}. For general parabolic subgroup, one needs to understand the dimension $\Wh(I_{\wt{P}}^{\wt{G}} \ \pi)$ first, and our paper answers exactly this question for $\pi\in \JH(I(\chi))$ with $\Phi(\chi) \subseteq \Delta$, and thus in principle enables one to carry out an explicit computation, if a parametrization of $\Wh (\pi_\Gamma)$ (and therefore  $\Wh(I_{\wt{P}}^{\wt{G}} \ \pi_\Gamma)$) is possible, as discussed in (R1) above.

The above consideration also has application in determining the global Whittaker-Fourier coefficients for the induced representation $I_{\wt{P}}^{\wt{G}} \pi$, especially when $\pi$ is a global theta representation with unique Whittaker model at all local places. For work pertaining to this topic, see \cite{Suz2, Suz3, BBL, Ga5}.

Lastly, we remark that in this paper we actually do not intertwine with the deeper aspect of the Kazhdan-Lusztig theory, since the representation $\sigma_\Gamma$ has the simple interpretation as an alternating sum (see Corollary  \ref{C:IE-KL}) in the Grothendieck group of ${\rm Irr}(W)$. Indeed, the crucial point invoked is the inductive property of right cells and Kazhdan-Lusztig representations proved by Barbasch and Vogan \cite{BV2} (see Proposition \ref{P:ind}). Nevertheless, we hope our paper serves as a small impetus to unravelling many of the mysteries of the function $\dim \Wh(-)_{\mca{O}_y}$. In fact, in a companion paper \cite{Ga7} to this, we will investigate unitary unramified principal series $I(\chi)$, and propose an analogous formula for $\dim \Wh(\pi)_{\mca{O}_y}, \pi \in \JH(I(\chi))$, where  characters of the R-group of $I(\chi)$ take place of the Kazhdan-Lusztig representations in this paper.

\subsection{Acknowledgement} I would like to thank Caihua Luo for several discussions on the content of \S \ref{S:red-ps}. Thanks are also due to the referee for his or her careful reading and insightful comments.

\section{Covering group} \label{S:cg}
Let $F$ be a finite extension of $\Q_p$. Denote by $O:=O_F \subseteq F$ the ring of integers of $F$ and $\varpi \in O$ a fixed uniformizer. 

\subsection{Covering group} \label{Sec:SF}
Let $\mbf{G}$ be a split connected linear algebraic group over $F$ with a maximal split torus $\mbf{T}$. Let
$$\set{ X,\ \Delta,  \ \Phi; \ Y, \ \Delta^\vee, \ \Phi^\vee  }$$
be the based root datum of $\mbf{G}$. Here $X$ (resp. $Y$) is the character lattice (resp. cocharacter lattice) for $(\mbf{G}, \mbf{T})$. Choose a set $\Delta\subseteq \Phi$ of simple roots from the set of roots $\Phi$, and let $\Delta^\vee$ be the corresponding simple coroots. This gives us a choice of positive roots $\Phi_+$ and positive coroots $\Phi_+^\vee$. Denote $\Phi_-:= - \Phi_+$ and $\Phi_-^\vee:= - \Phi_+^\vee$. Write $Y^{\sct}\subseteq Y$ for the sublattice generated by $\Phi^\vee$. Let $\mbf{B} =\mbf{T} \mbf{U}$ be the Borel subgroup associated with $\Delta$. Denote by $\mbf{U}^- \subset \mbf{G}$ the unipotent subgroup opposite $\mbf{U}$.

Fix a Chevalley-Steinberg system of pinnings for $(\mbf{G}, \mbf{T})$. That is, we fix a set of compatible isomorphisms
$$\set{e_\alpha: \mbf{G}_\text{a} \to \mbf{U}_\alpha}_{\alpha\in \Phi},$$ 
where $\mbf{U}_\alpha \subseteq \mbf{G}$ is the root subgroup  associated with $\alpha$. In particular, for each $\alpha\in \Phi$, there is a unique morphism $\varphi_\alpha: \SL_2 \to \mbf{G}$ which restricts to $e_{\pm \alpha}$ on the upper and lower triangular subgroup of unipotent matrices of $\SL_2$.

Denote by $W$ the Weyl group of $(\mbf{G}, \mbf{T})$, which we identify with the Weyl group of the coroot system. In particular, $W$ is generated by simple reflections $\set{w_\alpha: \alpha^\vee \in \Delta^\vee}$ for $Y \otimes \Q$. Let $l: W \to \N$ be the length function. Let $w_G$ be the longest element in $W$.

Consider the algebro-geometric $\mbf{K}_2$-extension $\wm{G}$ of $\mbf{G}$, which is categorically equivalent to the pairs $\set{(D, \eta)}$ (see \cite[\S 2.6]{GG}). Here 
$$\eta: Y^{\sct} \to F^\times$$
 is a homomorphism. On the other hand, 
$$D: Y \times Y \to \Z$$ 
is a (not necessarily symmetric) bilinear form on $Y$ such that 
$$Q(y):=D(y, y)$$
is a Weyl-invariant integer valued quadratic form on $Y$. We call $D$ a bisector following \cite[\S 2.1]{We3}. Let $B_Q$ be the Weyl-invariant bilinear form associated to $Q$ given by
$$B_Q(y_1, y_2)=Q(y_1+y_2)-Q(y_1) -Q(y_2).$$
Clearly, $D(y_1, y_2) + D(y_2, y_1)=B_Q(y_1, y_2)$.  Every $\wm{G}$ is, up to isomorphism, incarnated by (i.e. categorically associated to) a pair $(D,\eta)$ for a bisector $D$ and $\eta$.

The couple $(D, \eta)$ plays the following role for the structure of $\wm{G}$.
\begin{enumerate}
\item[(i)] The group $\wm{G}$ splits canonically and uniquely over any unipotent subgroup of $\mbf{G}$. For $\alpha\in \Phi$ and $a\in \mbf{G}_a$, denote by $\wt{e}_\alpha(a) \in \wm{G}$  the canonical lifting of $e_\alpha(a) \in \mbf{G}$. For $\alpha\in \Phi$ and $a\in \mbf{G}_m$, define 
\begin{equation*}  \label{F:w}
w_\alpha(a):=e_{\alpha}(a) \cdot e_{-\alpha}(-a^{-1}) \cdot e_{\alpha}(a) \text{ and } \wt{w}_\alpha(a):=\wt{e}_{\alpha}(a) \cdot \wt{e}_{-\alpha}(-a^{-1}) \cdot \wt{e}_{\alpha}(a).
\end{equation*}
This gives natural representatives $w_\alpha(1) \in \mbf{G}$, and also $\wt{w}_\alpha(1) \in \wm{G}$ of the Weyl element $w_\alpha\in W$. By abuse of notation, we also write $w_\alpha$ for $w_\alpha(1)$ and denote $\wt{w}_\alpha:= \wt{w}_\alpha(1)$. Moreover, for any $h_\alpha(a):=\alpha^\vee(a)\in \mbf{T}$, there is a natural lifting 
\begin{equation*}  \label{h-alpha}
\wt{h}_\alpha(a):=\wt{w}_\alpha(a)\cdot \wt{w}_\alpha(-1) \in \wm{T},
\end{equation*}
which depends only on the pinnings and the canonical unipotent splitting.
\item[(ii)] There is a section $\s$ of $\wm{T}$ over $\mbf{T}$ such that 
\begin{equation} \label{F:s}
\s(y_1(a)) \cdot \s(y_2(b)) = \set{a, b}^{D(y_1, y_2)} \cdot \s(y_1(a)\cdot y_2(b))
\end{equation}
for any $a, b\in \mbf{G}_m$, where $\set{a, b} \in \mbf{K}_2$ as in \cite[\S 0.N.5]{BD}. Moreover, for $\alpha\in \Delta$ and the natural lifting $\wt{h}_\alpha(a)$ of $h_\alpha(a)$ above, one has
\begin{equation*} 
\wt{h}_\alpha(a)=\set{\eta(\alpha^\vee), a} \cdot \s(h_\alpha(a)) \in \wm{T}.
\end{equation*}
\item[(iii)] Let $\wt{w}_\alpha \in \wm{G}$ be the above natural representative of $w_\alpha\in W$ with $\alpha\in \Delta$. For any $\wt{y(a)} \in \wm{T}$ with $y\in Y$ and $a\in \mbf{G}_m$, one has
\begin{equation} \label{F:W-act}
\wt{w}_\alpha \cdot \wt{y(a)} \cdot \wt{w}_\alpha^{-1} = \wt{y(a)} \cdot \wt{h}_\alpha(a^{-\angb{y}{\alpha}}),
\end{equation}
where $\angb{-}{-}$ is the paring between $Y$ and $X$.
\end{enumerate}

We remark that if the derived group of $\mbf{G}$ is simply-connected, then the isomorphism class of $\wm{G}$ is determined by the Weyl-invariant quadratic form $Q$. In particular, for such $\mbf{G}$, every extension $\wm{G}$ is incarnated by $(D, \eta=\mbf{1})$ for some bisector $D$, up to isomorphism.
In this paper, we assume that the composite
$$\eta_n: Y^{sc} \to F^\times \onto F^\times/(F^\times)^n$$ 
of $\eta$ with the obvious quotient is trivial.

Let $n\ge 1$. We assume that $F$ contains the full group of $n$-th roots of unity, denoted by $\bbmu_n$. An $n$-fold cover of $\mbf{G}$, in the sense of \cite[Definition 1.2]{We6}, is just the pair $(n, \wm{G})$.  The $\mbf{K}_2$-extension $\wm{G}$ gives rise to an $n$-fold covering $\wt{G}$ as follows.

Let
$$(-,-)_n:  F^\times \times F^\times \to \bbmu_n$$
be the local $n$-th Hilbert symbol. The local extension $\wt{G}$ arises from the central extension
 $$\begin{tikzcd}
\mbf{K}_2(F) \ar[r, hook] & \wm{G}(F) \ar[r, two heads, "\phi"] & \mbf{G}(F)
\end{tikzcd}$$
by push-out via the natural map $\mbf{K}_2(F) \to \bbmu_n$ given by $\set{a, b} \mapsto (a, b)_n$. This gives
$$\begin{tikzcd}
\bbmu_n \ar[r, hook] & \wt{G} \ar[r, two heads, "\phi"] & G.
\end{tikzcd}$$
We may write $\wt{G}^{(n)}$ to emphasize the degree of covering. A representation of $\wt{G}$ is called $\epsilon$-genuine (or simply genuine) if $\bbmu_n$ acts by a fixed embedding $\epsilon: \bbmu_n \into \C^\times$. We consider only genuine representations of a covering group in this paper.

For any subset $H \subset G$, denote $\wt{H}:=\phi^{-1}(H)$. The relations on generators of $\wm{G}$ described above give rise to the corresponding relations for $\wt{G}$.  For example, inherited from \eqref{F:s} is the following relation for the covering torus $\wt{T}$:
\begin{equation*}
\s(y_1(a)) \cdot \s(y_2(b)) = (a, b)_n^{D(y_1, y_2)} \cdot \s(y_1(a)\cdot y_2(b)),
\end{equation*}
where $y_i\in Y$ and $a, b\in F^\times$. The commutator $[\wt{t}_1, \wt{t}_2]:=\wt{t}_1 \wt{t}_2 \wt{t}_1^{-1} \wt{t}_2^{-1}$ on $\wt{T}$, which descends to a map $[-,-]: T \times T \to \bbmu_n$, is thus given by
$$[y_1(a), y_2(b)]=(a, b)_n^{B_Q(y_1, y_2)}.$$

Let $W' \subset \wt{G}$ be the group generated by $\wt{w}_\alpha$ for all $\alpha$. Then the map $\wt{w}_\alpha \mapsto w_\alpha$ gives a surjective morphism
$$W' \onto W$$
with kernel being a finite group (see \cite[\S 6.1]{Ga1}). For any $w=w_{\alpha_k} ... w_{\alpha_2} w_{\alpha_1} \in W$ in a minimal decomposition, we let
$$\wt{w}:=\wt{w}_{\alpha_k} ... \wt{w}_{\alpha_2} \wt{w}_{\alpha_1} \in W'$$
be its representative, which is independent of the minimal decomposition (see \cite[Lemma 83 (b)]{Ste16}). In particular, we denote by $\wt{w}_G \in \wt{G}$ the above representative of the longest Weyl element $w_G$ of $W$.

\subsection{Dual group and $L$-group}
For a cover $(n, \wm{G})$ associated to $(D, \eta)$, with $Q$ and $B_Q$ arising from $D$, we define
\begin{equation} \label{YQn}
Y_{Q,n}:= Y\cap nY^*,
\end{equation}
where $Y^* \subset Y\otimes \Q$ is the dual lattice of $Y$ with respect to $B_Q$; more explicitly,
$$Y_{Q,n}= \set{y\in Y: B_Q(y, y')\in n\Z \text{ for all } y'\in Y} \subset Y.$$
For every $\alpha^\vee\in \Phi^\vee$, denote
$$n_\alpha:= \frac{n}{\text{gcd}(n, Q(\alpha^\vee))}$$
and
$$ \alpha_{Q,n}^\vee=n_\alpha \alpha^\vee, \quad \alpha_{Q,n}=\frac{\alpha}{n_\alpha} .$$
Let 
$$Y_{Q,n}^{sc} \subset Y_{Q,n}$$ 
be the sublattice generated by $\Phi_{Q,n}^\vee=\{\alpha_{Q,n}^\vee: \alpha^\vee \in \Phi^\vee \}$.  Denote $X_{Q,n}=\text{Hom}_\Z(Y_{Q,n}, \Z)$ and $\Phi_{Q,n}=\set{\alpha_{Q,n}: \alpha \in \Phi }$. We also write 
$$\Delta_{Q,n}^\vee=\{ \alpha_{Q,n}^\vee: \alpha^\vee \in \Delta^\vee \} \text{ and } \Delta_{Q,n}=\set{\alpha_{Q,n}: \alpha\in \Delta}.$$
Then
$$\big( Y_{Q,n}, \ \Phi_{Q,n}^\vee, \ \Delta_{Q,n}^\vee;\  X_{Q,n},\  \Phi_{Q,n}^\vee, \Delta_{Q,n} \big)$$
forms a root datum with a given choice of simple roots $\Delta_{Q,n}$. It gives a unique (up to unique isomorphism) pinned reductive group $\wm{G}_{Q,n}^\vee$ over $\Z$, called the dual group of $(n, \wm{G})$. In particular, $Y_{Q,n}$ is the character lattice for $\wt{G}_{Q,n}^\vee$ and $\Delta_{Q,n}^\vee$ the set of simple roots. Let 
$$\wt{G}_{Q,n}^\vee:=\wm{G}_{Q,n}^\vee(\C)$$
be the associated complex dual group.

\begin{dfn} \label{D:sat}
A covering group $\wt{G}$ is called saturated if $Y_{Q,n} \cap Y^{sc} = Y_{Q,n}^{sc}$.
\end{dfn}

\begin{eg} \label{ss-sat}
Let $\mbf{M} \subseteq \mbf{G}$ be a Levi subgroup and $\wt{M} \subseteq \wt{G}$ the arising Levi covering subgroup. If $\wt{G}$ is saturated, then $\wt{M}$ is also saturated. Moreover, if $n=1$, then a linear algebraic group $\wt{G}=G$ is always saturated. If $\wt{G}$ is a degree $n$ cover of a simply-connected group $G$, then $\wt{G}$ is saturated if and only if $Y_{Q,n}=Y_{Q,n}^{sc}$; equivalently, the complex dual group $\wt{G}_{Q,n}^\vee$ is of adjoint type. In particular, covers of the exceptional group of type ${\rm E}_8, {\rm F}_4$ and ${\rm G}_2$ are always saturated, since the complex dual group of such covers always has trivial center. See \cite[\S 2]{We6} for more concrete examples.
\end{eg}

Let $WD_F = W_F \times \SL_2(\C)$ be the Weil-Deligne group of $F$. In \cite{We3, We6}, Weissman constructed the local $L$-group extension
 $$\begin{tikzcd}
\wt{G}^\vee_{Q,n} \ar[r, hook] & {}^L\wt{G} \ar[r, two heads] & WD_{F},
\end{tikzcd}$$
which is compatible with the global $L$-group. His construction of $L$-group is functorial, and in particular it behaves well with respect to the restriction of $\wm{G}$ to parabolic subgroups. More precisely, let $\mbf{M} \subset \mbf{G}$ be a Levi subgroup. By restriction, one has the $n$-cover $\wt{M}$ of $M$. Then the $L$-groups ${}^L\wt{M}$ and ${}^L\wt{G}$ are compatible, i.e., there are natural morphisms of extensions:
\begin{equation} \label{L-comp}
\begin{tikzcd}
\wt{G}^\vee_{Q,n} \ar[r, hook] & {}^L\wt{G} \ar[r, two heads] & WD_{F} \\
\wt{M}^\vee_{Q,n} \ar[r, hook] \ar[u, hook] & {}^L\wt{M} \ar[u, hook] \ar[r, two heads] & WD_{F}  \ar[u, equal] .
\end{tikzcd}
\end{equation}
This applies in particular to the case when $M=T$ is a torus.

In general, the extension ${}^L\wt{G}$ does not split over $WD_F$. However, if $\wt{G}_{Q,n}^\vee$ is of adjoint type, then we have a canonical isomorphism
\begin{equation*} \label{dual-ad}
{}^L\wt{G} \simeq \wt{G}_{Q,n}^\vee \times WD_F.
\end{equation*}
For general $\wt{G}$, under our assumption that $\eta_n=\mbm{1}$, there exists a so-called distinguished genuine character $\chi_\psi: Z(\wt{T}) \to \C^\times$ (see \cite[\S 6.4]{GG}), depending on a nontrivial additive character $\psi$ of $F$, such that $\chi_\psi$ gives rise to a splitting of ${}^L\wt{G}$ over $WD_F$, with respect to which one has an isomorphism 
\begin{equation} \label{L-iso}
{}^L\wt{G} \simeq_{\chi_\psi} \wt{G}_{Q,n}^\vee \times WD_F.
\end{equation}
For details on the construction and properties of the $L$-group, we refer the reader to \cite{We3, We6, GG}.

\subsection{Twisted Weyl action} \label{SS:twW}
Denote by $w(y)$ the natural Weyl group action on $Y$ or $Y\otimes \Q$, which is generated by the reflections $w_\alpha$. The two lattices $Y_{Q,n}$ and $Y_{Q,n}^{sc}$ are both $W$-stable under the usual action $w(y)$, since $Q$ is Weyl-invariant.  Let 
$$\rho:= \frac{1}{2} \sum_{\alpha^\vee >0} \alpha^\vee$$
be the half sum of all positive coroots of $\mbf{G}$.  We consider the twisted Weyl-action
$$w[y]:=w(y-\rho)+ \rho.$$
It induces a well-defined twisted action of $W$ on 
$$\msc{X}_{Q,n}:=Y/Y_{Q,n}$$ given by
$w[y + Y_{Q,n}]:=w[y] + Y_{Q,n}$, since $W(Y_{Q,n}) = Y_{Q,n}$ as mentioned. Let
$$\sigma_\msc{X}: W \longrightarrow  \text{Perm}(\msc{X}_{Q,n})$$
be the permutation representation given by $\sigma_\msc{X}(w)(y):= w[y]$. Similarly, the twisted action of $W$ on 
$$\msc{X}_{Q,n}^{sc}:=Y/Y_{Q,n}^{sc}$$
 is also well-defined.

Throughout the paper, we denote
$$y_\rho:=y-\rho \in Y\otimes \Q$$ 
for $y\in Y$. Clearly, $w[y]-y=w(y_\rho) - y_\rho$. Henceforth, by Weyl action or Weyl orbits in $Y$ or $Y\otimes \Q$,  we always refer to the ones with respect to the twisted action $w[y]$, unless specified otherwise. Clearly, the quotient
$$f: \msc{X}_{Q,n}^{sc} \onto \msc{X}_{Q,n}$$
is equivariant with respect to the Weyl action on the two sides.

Denote by $\hat{y} \in \msc{X}_{Q,n}^{sc}$ the class of $y\in Y$. Let ${\rm Stab}_W(\hat{y}; \msc{X}_{Q,n}^{sc}) \subseteq W$ be the stabiliser of $\hat{y}$ with respect to the action of $W$ on $\msc{X}_{Q,n}^{sc}$; similarly we have ${\rm Stab}_W(f(\hat{y}); \msc{X}_{Q,n})$.

\begin{dfn} \label{D:per}
An orbit $\mca{O}_y \subseteq Y$ is called persistent if  
\begin{equation} \label{E:stab}
{\rm Stab}_{W}(\hat{y}; \msc{X}_{Q,n}^{sc})= {\rm Stab}_{W}(f(\hat{y}); \msc{X}_{Q,n}).
\end{equation}
A $W$-orbit in $\msc{X}_{Q,n}$ is called persistent if it is the image of a persistent $W$-orbit in $Y$. 
A group $\wt{G}$ is called persistent if every $W$-orbit in $\msc{X}_{Q,n}$ is persistent.
\end{dfn}

It is not clear from the definition that if the image of $\mca{O}_y \subset Y$ in $\msc{X}_{Q,n}$ is a persistent orbit, then $\mca{O}_y$ is a persistent orbit. However, we show that this is indeed the case. Consider $\mca{O}_{f(\hat{z})} = \mca{O}_{f(\hat{y})} \subset \msc{X}_{Q,n}$. Then there exists $w_o \in W$ such that
$f(\hat{z}) = w_o[f(\hat{y}) ] \in \msc{X}_{Q,n}$, or equivalently, 
$$z_\rho = w_o(y_\rho) + z' \text{ for some }  z' \in Y_{Q,n}.$$
One has an isomorphism of finite groups
$${\rm Int}_{w_o}:  {\rm Stab}_W(f(\hat{y}); \msc{X}_{Q,n}) \to {\rm Stab}_W(f(\hat{z}); \msc{X}_{Q,n})$$
given by
$$ {\rm Int}_{w_o}(w) := w_o w w_o^{-1}.$$
Consider the diagram:
\begin{equation*} \label{PersO}
\begin{tikzcd}
{\rm Stab}_W(y; Y)  \ar[d, hook]  &    {\rm Stab}_W(z; Y) \ar[d, hook] \\
{\rm Stab}_W(\hat{y}; \msc{X}_{Q,n}^{sc}) \ar[d, hook] \ar[r, dashed, "{ {\rm Int}_{w_o} }"]   &  {\rm Stab}_W(\hat{z}; \msc{X}_{Q,n}^{sc}) \ar[d, hook]  \\
{\rm Stab}_W(f(\hat{y}); \msc{X}_{Q,n})   \ar[r, "{ {\rm Int}_{w_o} }"] &  {\rm Stab}_W(f(\hat{z}); \msc{X}_{Q,n})  \\
\end{tikzcd} 
\end{equation*}
where the vertical arrows are canonical injections.

\begin{prop} \label{P:ladC}
The restriction of ${\rm Int}_{w_o}$ to ${\rm Stab}_W(\hat{y}; \msc{X}_{Q,n}^{sc})$ gives a well-defined isomorphism into ${\rm Stab}_W(\hat{z}; \msc{X}_{Q,n}^{sc})$.
\end{prop}
\begin{proof}
First, we show that the restriction of ${\rm Int}_{w_o}$ to ${\rm Stab}_W(\hat{y}; \msc{X}_{Q,n}^{sc})$ is well-defined. For this, it suffices to show that if $w(y_\rho) = y_\rho + y'$ for some $y'\in Y_{Q,n}^{sc}$, then 
$$w_o w w_o^{-1}(z_\rho) - z_\rho \in Y_{Q,n}^{sc}.$$
Recall that $z_\rho = w_o(y_\rho) + z'$  for some   $z' \in Y_{Q,n}$. A simple computation gives that
$$\begin{aligned}
w_o w w_o^{-1}(z_\rho) - z_\rho  = w_o(y') +  w_o w w_o^{-1}(z') - z'. 
\end{aligned}$$
We have $w_o(y') \in Y_{Q,n}^{sc}$. On the other hand, for every $w\in W$, by using induction on the length of $w$, we show that $w(z') - z' \in Y_{Q,n}^{sc}$. Indeed, for this purpose, it suffices to consider a simple reflection $w = w_\alpha$ and thus
$$w_\alpha(z') - z' = \angb{z'}{\alpha} \cdot \alpha^\vee.$$
Since $B(z', \alpha^\vee) = Q(\alpha^\vee) \cdot \angb{z'}{\alpha} \in n\Z$ as $z'\in Y_{Q,n}$, one has $n_\alpha| \angb{z'}{\alpha}$, i.e., $\angb{z'}{\alpha} \cdot \alpha^\vee \in Y_{Q,n}^{sc}$. This completes the proof that  $w_o w w_o^{-1}(z_\rho) - z_\rho \in Y_{Q,n}^{sc}$ and thus the restriction of ${\rm Int}_{w_o}$ to 
${\rm Stab}_W(\hat{y}; \msc{X}_{Q,n}^{sc})$ is well-defined.

Second, one obtains similarly an injective homomorphism ${\rm Int}_{w_o^{-1}}: {\rm Stab}_W(\hat{z}; \msc{X}_{Q,n}^{sc}) \to {\rm Stab}_W(\hat{y}; \msc{X}_{Q,n}^{sc})$, which is clearly the inverse of ${\rm Int}_{w_o}$. Thus ${\rm Int}_{w_o}$ is an isomorphism from ${\rm Stab}_W(\hat{y}; \msc{X}_{Q,n}^{sc})$ to ${\rm Stab}_W(\hat{z}; \msc{X}_{Q,n}^{sc})$.
\end{proof}

This immediately gives:
\begin{cor} \label{C:indO}
If $\mca{O}_{f(\hat{y})} = \mca{O}_{f(\hat{z})} \subset \msc{X}_{Q,n}$, then $\mca{O}_y \subset Y$ is persistent if and only if $\mca{O}_z \subset Y$ is persistent. Equivalently, a $W$-orbit $\mca{O}_y \subset Y$ is persistent if and only if $\mca{O}_{f(\hat{y})} \subset \msc{X}_{Q,n}$ is persistent.
\end{cor}

Let $S\subset \Delta$ and let $W(S) \subset W$ be the parabolic subgroup generated by $S$. By restriction of the action of $W$ to $W(S)$, we obtain a decomposition of the $W$-orbit
\begin{equation} \label{dc-O1}
\mca{O}_y = \bigsqcup_{i\in I}  \mca{O}^{w_S}_{y_i},
\end{equation}
where each $\mca{O}_{y_i}^{w_S}$ is a $W(S)$-orbit in $Y$.

\begin{lm} \label{pers-para}
If $\mca{O}_y$ is a persistent $W$-orbit in $Y$, then every $\mca{O}_{y_i}^{w_S}$ is a persistent $W(S)$-orbit. Therefore, if $\wt{G}$ is persistent, then every standard Levi subgroup $\wt{M}$ is also persistent.
\end{lm}
\begin{proof}
It suffices to prove the first assertion, as the second follows from the first, Corollary \ref{C:indO}, and the fact that every $W(S)$-orbit in $\msc{X}_{Q,n}$ lies in some $W$-orbit. Let $M \subset G$ be the Levi subgroup associated with $S \subset \Delta$. To differentiate the notation, we add subscripts $G$ or $M$ as in $Y_{Q,n,G}^{sc}$ versus $Y_{Q,n, M}^{sc}$. We want to show that the inclusion
$${\rm Stab}_{W(S)}(\hat{y}'; Y/Y_{Q,n,M}^{sc} ) \into {\rm Stab}_{W(S)}(f(\hat{y}); Y/Y_{Q,n})$$
is an equality, where $\hat{y}'$ denotes the image of $y$ in $Y/Y_{Q,n,M}^{sc}$. For this purpose, let $w\in {\rm Stab}_{W(S)}(f(\hat{y}); Y/Y_{Q,n})$. Since
$${\rm Stab}_{W(S)}(f(\hat{y}); Y/Y_{Q,n}) \subset {\rm Stab}_W(f(\hat{y}); Y/Y_{Q,n}) = {\rm Stab}_W(\hat{y}; Y/Y_{Q,n, G}^{sc}),$$
where the equality follows from our assumption, we have
$$w[y] - y \in Y_{Q,n, G}^{sc} \cap Y_M^{sc} = Y_{Q,n, M}^{sc}.$$
That is, $w\in {\rm Stab}_{W(S)}(\hat{y}'; Y/Y_{Q,n,M}^{sc} )$. This completes the proof.
\end{proof}

As a first example, we have:
\begin{lm} \label{Satu}
Let $\wt{G}$ be a saturated covering group. Then every orbit $\mca{O}_y \subseteq Y$ is persistent and thus $\wt{G}$ is persistent. In particular, $f$ preserves free $W$-orbits; that is, $\mca{O}_{\hat{y}} \subseteq \msc{X}_{Q,n}^{sc}$ is a $W$-free orbit if and only if $\mca{O}_{f(\hat{y})} \subseteq \msc{X}_{Q,n}$ is $W$-free.
\end{lm}
\begin{proof}
The inclusion $\subset$ in \eqref{E:stab} always holds for every $\mca{O}_y \subset Y$. 

We show the other direction for saturated $\wt{G}$. Let $w \in {\rm Stab}_W(f(\hat{y}); \msc{X}_{Q,n})$. That is, $w[y] \equiv y \text{ mod } Y_{Q,n}$. However, as $\wt{G}$ is saturated, we have
$$w[y] - y \in Y^{sc} \cap Y_{Q,n} = Y_{Q,n}^{sc}.$$
Thus, $w[\hat{y}] = \hat{y}$ in $\msc{X}_{Q,n}^{sc}$, i.e., $w \in {\rm Stab}_W(\hat{y}; \msc{X}_{Q,n}^{sc}) $ as well.
\end{proof}

\begin{eg}
Consider the cover $\wt{\SL}_2^{(n)}, n\in \N$ associated to the quadratic form $Q$ with $Q(\alpha^\vee)=1$. Then $\wt{\SL}_2^{(n)}$ is
\begin{enumerate}
\item[$\bullet$] saturated (and thus persistent), if $n$ is odd;
\item[$\bullet$] persistent but not saturated, if $n=4m$;
\item[$\bullet$] not persistent, if $n=2k$ with $k$ odd.
\end{enumerate}
For example, if $n=4m$, then
$$Y_{Q,n}=\Z \cdot (2m\alpha^\vee) \text{ and } Y_{Q,n}^{sc} = \Z\cdot (4m \alpha^\vee).$$
In this case, every orbit in $\msc{X}_{Q,n}$ is $W$-free, and thus \eqref{E:stab} always holds. However, $\wt{\SL}_2^{(4m)}$ is not saturated as $Y_{Q,n} \ne Y_{Q,n}^{sc}$. 

On the other hand, every Brylinski-Deligne covering $\wt{\GL}_r$ of $\GL_r$ is saturated, as we always have $Y_{Q,n} \cap Y^{sc} = Y_{Q,n}^{sc}$, see \cite[Example 3.16]{Ga2}. This contrast between $\wt{\SL}_r$ and $\wt{\GL}_r$ (even for $r=2$) has some interesting consequences on the representations of the two covering groups, see \cite[Example 4.20, 4.21]{GSS2}.
\end{eg}

For simplicity, we will abuse notation and denote by $y$ (instead of $\hat{y}$ or $f(\hat{y})$) for an element in $\msc{X}_{Q,n}^{sc}$ or $\msc{X}_{Q,n}$. Denote by $\mca{O}_{\msc{X}_{Q,n}}$ (resp. $\OF_{\msc{X}_{Q,n}}$) the set of $W$-orbits (resp. free $W$-orbits) in $\msc{X}_{Q,n}$.

\section{Regular unramified principal series} \label{S:red-ps}
We assume that $|n|_F=1$ henceforth. Let $K \subset G$ be the hyperspecial maximal compact subgroup generated by $\mbf{T}(O)$ and $e_\alpha(O)$ for all root $\alpha$. With our assumption that $\eta_n$ is trivial, the group $\wt{G}$ splits over $K$ (see \cite[Theorem 4.2]{GG}) and we fix such a splitting $s_K$. If no confusion arises, we will omit $s_K$ and write $K \subset \wt{G}$ instead. Call $\wt{G}$ an unramified covering group in this setting. 

A genuine representation $(\pi, V_\pi)$ called unramified if $\dim V_\pi^{K}\ne 0$. With our assumption $\val{n}_F=1$ made, $\wt{G}$ splits canonically and uniquely over the unipotent subgroup $e_\alpha(O)$, and thus we see that $\wt{h}_\alpha(u) \in s_K(K) \subset \wt{G}$ for every $u \in O^\times$.

\subsection{Principal series representation}
Recall that $U$ is the unipotent subgroup of $B=TU$. As $U$ splits canonically in $\wt{G}$, we have $\wt{B}=\wt{T}U$. The covering torus $\wt{T}$ is a Heisenberg-type group. The center $Z(\wt{T})$ of the covering torus $\wt{T}$ is equal to $\phi^{-1}(\text{Im}(i_{Q,n}))$, where 
$$i_{Q,n}: Y_{Q,n}\otimes F^\times \to T$$
is the isogeny induced from the embedding $Y_{Q,n} \subset Y$, see \cite[Proposition 4.1]{We1}.

Let $\chi \in \Hom_\epsilon(Z(\wt{T}), \C^\times)$ be a genuine character of $Z(\wt{T})$, write 
$$i(\chi):=\text{Ind}_{\wt{A}}^{\wt{T}} \ \chi'$$
 for the induced representation of $\wt{T}$, where $\wt{A}$ is a maximal abelian subgroup of $\wt{T}$, and $\chi'$ is any extension of $\chi$. By the Stone-von Neumann theorem (see \cite[Theorem 3.1]{We1}), the construction 
 $$\chi \mapsto i(\chi)$$
 gives a bijection between isomorphism classes of genuine representations of $Z(\wt{T})$ and $\wt{T}$. Since we consider unramified covering group $\wt{G}$ in this paper, we take 
 $$\wt{A}:=Z(\wt{T})\cdot (K\cap T).$$
 
The left action of $w$ on $i(\chi)$ is given by ${}^w i(\chi) (\wt{t})= i(\chi)(w^{-1} \wt{t} w)$. The group $W$ does not act on $i(\chi)$, but only on its isomorphism class. On the other hand, we have a well-defined action of $W$ on $\chi$:
$$({}^w \chi)(\wt{t}):= \chi(w^{-1} \wt{t} w).$$

View $i(\chi)$ as a genuine representation of $\wt{B}$ by inflation from the quotient map $\wt{B} \to \wt{T}$. Denote by 
$$I(i(\chi)):=\text{Ind}_{\wt{B}}^{\wt{G}}\ i(\chi)$$ the normalized induced principal series representation of $\wt{G}$. For simplicity, we may also write $I(\chi)$ for $I(i(\chi))$. The representation $I(\chi)$ is unramified (i.e. $I(\chi)^K\ne 0$) if and only if $\chi$ is unramified, i.e., $\chi$ is trivial on $Z(\wt{T})\cap K$.  Moreover, by the Satake isomorphism for Brylinski-Deligne covers \cite[\S 7]{We6}, a genuine representation is unramified if and only if it is a subquotient of an unramified principal series.

Denoting $\wt{Y}_{Q,n}:=Z(\wt{T})/(Z(\wt{T})\cap K)$, one has a natural abelian extension
\begin{equation} \label{Ext1}
\begin{tikzcd}
\bbmu_n \ar[r, hook] & \wt{Y}_{Q,n} \ar[r, two heads, "\varphi"] & Y_{Q,n}
\end{tikzcd}
\end{equation}
such that unramified genuine characters of $\chi$ of $Z(\wt{T})$ correspond to genuine characters of $\wt{Y}_{Q,n}$. Since $\wt{A}/(T\cap K)\simeq \wt{Y}_{Q,n}$ as well, there is a canonical extension (also denoted by $\chi$) of an unramified character $\chi$ of $Z(\wt{T})$ to $\wt{A}$, by composing $\chi$ with $\wt{A} \twoheadrightarrow \wt{Y}_{Q,n}$. Therefore, we will identify $i(\chi)$ as $\text{Ind}_{\wt{A}}^{\wt{T}}\ \chi$ for this  canonical extension $\chi$.

For $w\in W$, the intertwining operator $T(w, \chi): I(\chi) \to I({}^{w}\chi)$ is defined by
$$T(w, \chi) f)(\wt{g})=\int_{U_{w}} r_w^{\rm un} (f(\wt{w}^{-1} u \wt{g}) )du$$
whenever it is absolutely convergent. Here $r_w^{\rm un}: {}^w i(\chi) \to i({}^w \chi)$ is the canonical isomorphism given by $r_w^{\rm un}(f)(\wt{t}) = f(w^{-1} \wt{t} w)$ for every $f\in {}^w i(\chi)$. The operator $T(w, \chi)$ can be meromorphically continued for all $\chi$ (see \cite[\S 7]{Mc1}), and satisfies the cocycle condition as in the linear case.

Let $\Phi_w:=\set{\alpha\in \Phi^+: w(\alpha) \in \Phi^-}$. For unramified $I(\chi)$, let $f_0\in I(\chi)$ and $f_0' \in I({}^{w} \chi)$ be the normalized unramified vectors. We have
$$T(w, \chi) (f_0) = c(w, \chi) \cdot f_0'$$
where
\begin{equation} \label{GK}
c(w, \chi)=\prod_{\alpha \in  \Phi_w} \frac{1- q^{-1} \chi\big( \wt{h}_\alpha(\varpi^{n_\alpha}) \big) }{1- \chi\big( \wt{h}_\alpha(\varpi^{n_\alpha}) \big) }.
\end{equation}
For general covering groups, the coefficient $c(w, \chi)$ was computed in \cite[Theorem 12.1]{Mc2} and  later refined in \cite[Corollary 7.4]{Ga1}. We use the latter formalism which is more adapted to the Brylinski-Deligne framework.

\begin{rmk}
For simplicity of notation, we have used $T(w, \chi)$ to mean $T(w, \chi; r_w^{\rm un})$ in the notation of \cite[\S 3.6]{GSS2}. The importance and subtlety of the involvement of $r_w^{\rm un}$ is discussed in \cite[\S 3.6]{GSS2}.
\end{rmk}

\subsection{Reducible principal series}
A genuine character $\chi$ of $Z(\wt{T})$ is called regular if ${}^w \chi \ne \chi$ for all $w \in W -\set{ \text{id} }$. In this paper, we consider only regular character $\chi$, and call the associated $I(\chi)$ a regular principal series.

\begin{lm} \label{L:reg}
Let $\chi$ be an unramified regular character of $Z(\wt{T})$. Then, for every $\beta \in \Phi$,
$$\chi(\wt{h}_\beta(\varpi^{n_\beta}))  \ne 1.$$
Moreover, the set $\Phi \cap \Z[\Phi(\chi)]$ forms a root system with $\Phi(\chi)$ a set of simple roots.
\end{lm}
\begin{proof} It is shown in \cite[Page 112]{Ga1} that for all $w \in W$,
$$w^{-1}\cdot \wt{h}_\beta(\varpi^{n_\beta})\cdot w =  \wt{h}_{w(\beta)}(\varpi^{n_\beta}),$$
where $n_\beta = n_{w(\beta)}$, since $Q$ is Weyl-invariant. Now suppose there exists $\beta\in \Phi$ such that $\chi(\wt{h}_\beta(\varpi^{n_\beta}))  =1$. Then it follows that there exists $\alpha=w(\beta) \in \Delta$ for some $w$ such that $\chi(\wt{h}_\alpha(\varpi^{n_\alpha}))  =1$. Since $\chi$ is unramified, we have
$$\chi(\wt{h}_\alpha(a^{n_\alpha}))  =1$$
for all $a\in F^\times$. We claim that ${}^{w_\alpha} \chi = \chi$, which will yield a contradiction. For this purpose, it suffices to evaluate at any $\wt{y(a)}\in Z(T)$ with $y\in Y_{Q,n}$. 
Note first that $\wt{h}_\alpha(a)= \wt{w}_\alpha(a) \wt{w}_\alpha(-1)= \wt{w}_\alpha(a) \wt{w}_\alpha^{-1} $ and thus
$$\wt{w}_\alpha^{-1}= \wt{h}_\alpha(-1) \cdot \wt{w}_\alpha.$$ 
Therefore, by \eqref{F:W-act}
$$\begin{aligned}
w_\alpha^{-1} \cdot \wt{y(a)} \cdot w_\alpha & = \wt{h}_\alpha(-1) \cdot w_\alpha \cdot \wt{y(a)} \cdot w_\alpha^{-1} \cdot \wt{h}_\alpha(-1)^{-1} \\
& = \wt{h}_\alpha(-1) \cdot \wt{y(a)} \cdot \wt{h}_\alpha(a^{-\angb{y}{\alpha}}) \cdot \wt{h}_\alpha(-1)^{-1} \\
\end{aligned} $$
Since $y\in Y_{Q,n}$, we have $B_Q(y,\alpha^\vee)= Q(\alpha^\vee) \cdot \angb{y}{\alpha} \in n\Z$; thus, $n_\alpha |\angb{y}{\alpha}$ for all $y\in Y_{Q,n}$. In particular, $\wt{h}_\alpha(a^{-\angb{y}{\alpha}})  \in Z(\wt{T})$. This shows that
$$w_\alpha^{-1} \cdot \wt{y(a)} \cdot w_\alpha= (-1, a)_n^{B_Q(\alpha^\vee, y)} \cdot \wt{y(a)} \cdot \wt{h}_\alpha(a^{-\angb{y}{\alpha}}) =\wt{y(a)} \cdot \wt{h}_\alpha(a^{-\angb{y}{\alpha}}).$$
It follows
$${}^{w_\alpha} \chi(\wt{y(a)})= \chi( w_\alpha^{-1} \cdot \wt{y(a)} \cdot w_\alpha)=  \chi(\wt{y(a)}) \cdot \chi( \wt{h}_\alpha(a^{-\angb{y}{\alpha}})  )= \chi(\wt{y(a)}).$$
This gives the desired equality ${}^{w_\alpha} \chi = \chi$ and the contradiction.

The second assertion is just \cite[page 418, Proposition 3]{Rod4}, the argument of which relies only on the regularity of $\chi$ and thus also applies here.
\end{proof}

\begin{rmk}
It was pointed out in \cite[Page 441, \S II.1.2]{MW1} that for linear classical groups, $\Phi(\chi)$ is a subset of a set of simple roots for the original root system of $G$. However, this fails for general linear reductive group. A counterexample is given for certain principal series of the exceptional group ${\rm G}_2$ in \cite[Page 441]{MW1}.
\end{rmk}

For a regular character $\chi$, define
$$\Phi(\chi):=\set{\alpha\in \Phi:  \chi(\wt{h}_\alpha(\varpi^{n_\alpha})  )=q^{-1} }   \subset \Phi.$$
Denote $V:=X\otimes_\Z \R$. Let $\mca{C}^+ \subset V$ be the positive Weyl chamber associated with $\Delta$. We also write $\mca{C}^-:=w_G(\mca{C}^+)$, the Weyl chamber ``opposite" $\mca{C}^+$. Denote by 
$$\msc{C}(X\otimes \R; \chi)$$
 the set of connected components of
\begin{equation} \label{uGam}
V-\bigcup_{\alpha \in \Phi(\chi)} \Ker(\alpha^\vee).
\end{equation}
Let $\msc{P}(\Phi(\chi))$ be the power set of $\Phi(\chi)$. We have a bijection between the two sets
\begin{equation*}
\msc{P}(\Phi(\chi))  \longleftrightarrow \msc{C}(X\otimes \R; \chi)
\end{equation*}
 given by
$$S \to \Gamma_S:=\set{ x\in X\otimes \R:  \angb{\alpha^\vee}{x} <0 \text{ if and only if } \alpha \in S},$$
We also denote the converse correspondence by
$$S_\Gamma \leftarrow \Gamma.$$

In particular, we write
$$\Gamma^+:= \Gamma_\emptyset = \bigcap_{\alpha\in \Phi(\chi)} (\alpha^\vee)^{-1}(\R_{>0}), \quad \Gamma^-:= \Gamma_{\Phi(\chi)}= \bigcap_{\alpha\in \Phi(\chi)} (\alpha^\vee)^{-1}(\R_{<0}) $$
It is easy to see that 
$$\Gamma=\bigsqcup_{w \in W_\Gamma} w(\mca{C}^+),$$
where
\begin{equation} \label{W-G}
W_\Gamma:= \set{w\in W: \Phi(\chi)^\vee \cap w(\Phi^\vee_-)  = S_\Gamma^\vee}  \subset W.
\end{equation}
Moreover,
$$\Gamma_\text{op}:= \bigsqcup_{w \in W_\Gamma \cdot w_G} w(\mca{C}^+)$$
is also a connected component of \eqref{uGam} with $W_{\Gamma_\text{op}}= W_\Gamma \cdot w_G$.
 
 If $\Phi(\chi) \subseteq \Delta$, then $\mca{C}^+ \subset \Gamma^+$ and $\mca{C}^+ \subset \Gamma^-$. More specially, if $\Phi(\chi)= \Delta$, then $\mca{C}^+ =\Gamma^+$; also in this case $\mca{C}^- =\Gamma^-$.


\begin{lm} \label{L:1}
Retain notations as above. One has $W_{\Gamma^+}=\set{w\in W: w^{-1}(\Phi(\chi)^\vee) \subseteq \Phi_+^\vee}$; similarly, $W_{\Gamma^-}=\set{w\in W: w^{-1}(\Phi(\chi)^\vee) \subseteq \Phi_-^\vee}$.
\end{lm}
\begin{proof}
This follows from \eqref{W-G} and the fact that $S_{\Gamma^+} =\emptyset$ and $S_{\Gamma^-} = \Phi(\chi)$.
\end{proof}

\begin{prop} \label{J-ker}
Let $\chi$ be a regular unramified genuine character. Let $w_1, w_2 \in W$.
\begin{enumerate}
\item[(i)] One has
$$\dim \Hom_{\wt{G}} (I({}^{w_1^{-1} } \chi),  I( {}^{w_2^{-1}} \chi))=1$$
with $T(w_2^{-1}w_1, {}^{w_1^{-1}}\chi)$ being a basis of the space on the left.
\item[(ii)] If $T$ is any basis of  $\Hom_{\wt{G}} (I({}^{w_1^{-1} } \chi),  I( {}^{w_2^{-1}} \chi))$ (and thus necessarily a scalar-multiple of $T(w_2^{-1}w_1, {}^{w_1^{-1}}\chi)$), then
\begin{equation} \label{JM}
\Ker(T)_U = \bigoplus_{w \in W^T}  \delta_B^{1/2} \cdot i( {}^{w^{-1}} \chi),
\end{equation}
where $W^T \subseteq W$ is the set of all $w \in W$ satisfying that there exists $\alpha^\vee \in \Phi(\chi)$ such that $\Ker(\alpha^\vee)$ separates $w_1(\mca{C}^+)$ and $w_2(\mca{C}^+)$, and that $w_1(\mca{C}^+)$ and $w(\mca{C}^+)$ lie at the same side of $\Ker(\alpha^\vee)$.
\end{enumerate}
\end{prop}

\begin{proof} 
For (i), the fact that $\dim \Hom_{\wt{G}} (I({}^{w_1^{-1} } \chi),  I( {}^{w_2^{-1}} \chi))=1$ for regular $\chi$ is well-known, see for example \cite[\S 2.9]{BZ2}, \cite[Theorem 4]{Mc1} and also \cite[\S 5]{Sav1}. Since $\chi$ is regular, ${}^{w_1^{-1}}\chi$ is regular as well, and thus 
by Lemma \ref{L:reg}, the intertwining operator $T(w_2^{-1}w_1, {}^{w_1^{-1}}\chi)$ is a well-defined element in $\Hom_{\wt{G}} (I({}^{w_1^{-1} } \chi),  I( {}^{w_2^{-1}} \chi))$; that is, it has no pole at ${}^{w_1^{-1}}\chi$ from its meromorphic continuation (see \cite[page 66-67]{KP} and also \cite[\S 13.7]{Mc1}). On the other hand, by choosing a certain $f\in I({}^{w_1^{-1} } \chi) $ with small support, we can show that 
$$T(w_2^{-1}w_1, {}^{w_1^{-1}}\chi)(f)\ne 0,$$
see the proof of \cite[Theorem I.2.9]{KP} for details. Therefore, $T(w_2^{-1}w_1, {}^{w_1^{-1}}\chi)$ is a basis for the one-dimensional space $\Hom_{\wt{G}} (I({}^{w_1^{-1} } \chi),  I( {}^{w_2^{-1}} \chi))$.

For (ii), we note that in the linear algebraic case $\wt{G}=G$, it is just \cite[Proposition 2]{Rod4}, which is deduced from \cite[Lemma 1]{Rod4} asserting on the reducibility of rank-one principal series. The proof of \cite[Lemma 1]{Rod4} relies on the uniqueness of Whittaker functionals for representation of $G$. Such uniqueness does not hold for covering groups. However, on the other hand, since we are dealing with unramified regular principal series representations of $\wt{G}$, we can circumvent the problem and deduce the reducibility point by applying directly Casselman's criterion as follows.

Let $I(\chi)$ be an unramfied principal series for a rank-one group $\wt{G}$, where $\chi$ is a regular character of $Z(\wt{T})$. Let $w:=w_\alpha$ be the unique nontrivial Weyl element associated to $\alpha\in \Phi$.  The Plancherel measure $\mu(w, \chi)$, as a rational function in $\chi$, is such that
$$T(w^{-1}, {}^w \chi) \circ T(w, \chi) = \mu(w, \chi)^{-1} \cdot \text{id}  \in \text{End}(I(\chi)).$$ 
More explicitly,
$$\mu(w, \chi)^{-1}= c(w^{-1}, {}^w \chi) \cdot c(w, \chi),$$
where $c(w, \chi)$ is the Gindikin-Karpelevich coefficient in \eqref{GK}. Since $\chi$ is regular, Casselman's criterion (see \cite[Theorem 6.6.2]{CasB}) implies that $\mu(w, \chi)$ has a pole at $\chi$ if and only if $I(\chi)$ is reducible. It therefore follows that $I(\chi)$ is reducible if and only if 
$$\chi(\wt{h}_\alpha(\varpi^{n_\alpha}))=q^{-1} \text{ or } q.$$
That is, 
$\Phi(\chi)=\set{\alpha}$ or $\set{-\alpha}$. In this case, it is clear that the Jacquet module of a subquotient of $I(\chi)$ is the expected one; that is, \eqref{JM} holds for rank-one groups.

Based on the above, rest of the argument in the proof of \cite[Lemma 1, Proposition 2]{Rod4} can be carried over for covering groups to conclude the proof.
\end{proof}
Retain the notations in Proposition \ref{J-ker}, we have
\begin{equation*}
W^T =
\set{
\begin{array}{cc}
w \in W: \\
 \text{ there exists } \alpha \in \Phi(\chi) \text{ such that } \\
   \angb{\mca{C}^+}{w_1^{-1}(\alpha^\vee)}\cdot \angb{\mca{C}^+}{w_2^{-1}(\alpha^\vee)} <0,\\
  \text{ and }  \angb{\mca{C}^+}{w_1^{-1}(\alpha^\vee)}\cdot \angb{\mca{C}^+}{w^{-1}(\alpha^\vee)} >0  
\end{array}
}.
\end{equation*}
It follows that 
$$
\Ima(T)_U = \bigoplus_{w \in W- W^T}  \delta_B^{1/2} \cdot i({}^{w^{-1}} \chi),
$$
where
\begin{equation} \label{J-im}
 W - W^T=
\set{
\begin{array}{cc}
w \in W: \\
 \text{ if } \alpha \in \Phi(\chi) \text{ and }  \angb{\mca{C}^+}{w_1^{-1}(\alpha^\vee)}\cdot \angb{\mca{C}^+}{w_2^{-1}(\alpha^\vee)} <0,\\
  \text{ then }  \angb{\mca{C}^+}{w^{-1}(\alpha^\vee)}\cdot \angb{\mca{C}^+}{w_2^{-1}(\alpha^\vee)} >0  
\end{array}
}.
\end{equation}
Namely, $W-W^T \subseteq W$ is the set of all $w$ such that $w_2(\mca{C}^+)$ and $w(\mca{C}^+)$ lie at the same side of $\Ker(\alpha^\vee)$ with $\alpha\in \Phi(\chi)$, whenever $\Ker(\alpha^\vee)$ separates $w_1(\mca{C}^+)$ and $w_2(\mca{C}^+)$.

\begin{thm} \label{T:Rodier}
Let $\wt{G}$ be a Brylinski-Deligne covering group and $\chi$ a regular unramified genuine character of $Z(\wt{T})$. Then the representation $I(\chi)$ has a Jordan-Holder series ${\rm JH}(I(\chi))$, the subquotients of which have multiplicity one. Moreover, there is a bijection
$$\msc{C}(X\otimes \R; \chi) \to {\rm JH}(I(\chi)),  \text{ denoted by }  \Gamma \to \pi_\Gamma,$$
satisfying the following properties:
\begin{enumerate}
\item[(i)] The representation $\pi_\Gamma$ is characterized by
$$(\pi_\Gamma)_U = \bigoplus_{w \in W_\Gamma  }  \delta_B^{1/2} \cdot i({}^{w^{-1}} \chi).$$
Moreover, $w \in W_\Gamma$  if and only if $\pi_\Gamma$ is isomorphic to the unique irreducible subrepresentation of $I({}^{w^{-1}} \chi)$; in fact, $\pi_\Gamma$ is the image of $T(w^{-1} w_1, {}^{w_1^{-1}} \chi)$ for every $w_1 \in W_{\Gamma_{\rm op}}$.
%
\item[(ii)] The representation $\pi_\Gamma$ is square integrable modulo the center of $\wt{G}$ if and only if $\val{\Phi(\chi)}=\val{\Delta}$ and $\Gamma=\Gamma^+$.
\item[(iii)] The representation $\pi_\Gamma$ is tempered if and only if $\Gamma=\Gamma^+$ and the restriction of $\chi$ to 
$$\varphi^{-1}\left( \bigcap_{\alpha\in \Phi(\chi)} \Ker(\alpha)   \right)$$
is unitary, where $\varphi: \wt{Y}_{Q,n} \to Y_{Q,n}$ is the quotient in \eqref{Ext1}.
\item[(iv)] The representation $\pi_{\Gamma^-}$ is the unique irreducible unramified subquotient of $I(\chi)$.
\end{enumerate}
\end{thm}
\begin{proof}
Part (i) is just the main theorem of \cite[page 418]{Rod4}. The same as for Proposition \ref{J-ker}, the proof there carries over to covering groups.

Part (ii) and (iii) follow from the argument for \cite[Proposition 5, Proposition 6]{Rod4} and the Langlands classification theorem for covering groups proved by Ban and Janzten \cite{BJ1, BJ2}. More precisely, the Casselman's criterion (\cite[Theorem 4.4.6]{CasB}) for square integrability and temperedness of $\pi_\Gamma$ is extended to central covering groups in \cite[Theorem 3.4]{BJ1}. Coupled with this, the argument  in \cite{Rod4} for linear algebraic group applies verbatim in the setting of covering groups.

For (iv), we pick and fix one $w\in W$ such that $w(\mca{C}^+) \subseteq \Gamma^-$. By (i), it suffices to show the non-vanishing of the Gindikin-Karpelevich coefficient:
$$c(w_G, {}^{w_G w^{-1}}  \chi)\ne 0.$$
We have 
$$\begin{aligned}
& c(w_G, {}^{w_G w^{-1}}  \chi) \\
= & \prod_{\alpha \in \Phi_+} \frac{1-q^{-1} ({}^{w_G w^{-1}}\chi)(\wt{h}_\alpha(\varpi^{n_\alpha}) )    }{ 1- ({}^{w_G w^{-1}} \chi)(\wt{h}_\alpha(\varpi^{n_\alpha}) )   }  \\
= & \prod_{\alpha>0} \frac{1-q^{-1} \chi(w\cdot \wt{h}_\alpha(\varpi^{n_\alpha})\cdot w^{-1})^{-1}   }{ 1- \chi(w \cdot \wt{h}_\alpha(\varpi^{n_\alpha}) \cdot w^{-1})^{-1}  }.
\end{aligned} $$

For every $\alpha^\vee>0$, denote temporarily $\beta^\vee:=w(\alpha^\vee)$. Noting that $w\in  W_{\Gamma^-}$, Lemma \ref{L:1}  implies that 
\begin{equation} \label{E:1}
\Phi(\chi)^\vee \cap w(\Phi^\vee_+) =\emptyset.
\end{equation}
This gives
$$\chi(w \cdot \wt{h}_\alpha(\varpi^{n_\alpha}) \cdot w^{-1})= \chi( \wt{h}_\beta(\varpi^{n_\beta}) ) \ne q^{-1},$$
where the last non-equality follows from \eqref{E:1}. Thus, $c(w_G, {}^{w_G w^{-1}}  \chi) \ne 0$. This completes the proof.
\end{proof}

\subsection{L-parameter}  \label{L-para}
Recall that $WD_F = W_F \times \SL_2(\C)$ is the Weil-Deligne group.  Denote by $\text{Rep}(WD_F)$ the set of isomorphism classes of continuous representations of $WD_F$ which take the form 
$$\pi_o \boxtimes \tau: WD_F \to GL(V),$$
where $\tau: \SL_2(\C) \to GL(V)$ is an algebraic homomorphism and $\pi_o$ a continuous representation of $W_F$. Denote by $\val{\cdot}$ the homomorphism
$$\val{\cdot}: WD_F \onto W_F \to F^\times \to \C,$$
where the middle map $W_F \to F^\times$ is the Artin reciprocity map sending a geometric Frobenius to the uniformizer $\varpi$, and the first map is the apparent quotient. On the other hand, we also consider the group (see \cite{Tat1, Roh}) $W_F \ltimes \C$
with the group law given by
$$g z g^{-1} = \val{g}\cdot z \text{ for all } z\in \C, g\in W_F.$$
A representation of $W_F\ltimes \C$ is given by a pair $(\pi, N)$ where $N\in \text{Lie}(GL(V))$ and $\pi$ is a representation of $W_F$ in $V$ such that
$$\pi(g)\cdot N  \cdot \pi(g)^{-1} = \val{g} N \text{ for all } g\in W_F.$$
Denote by $\text{Rep}(W_F\ltimes \C)$ the set of all isomorphism classes of such representation $(\pi, N)$ of $W_F\ltimes \C$.

For every $n\in \N$, denote by $\text{sp}(n) \in \text{Rep}(W_F\ltimes \C)$ the standard indecomposible $n$-dimensional representation given as  in \cite[\S 4.1]{Tat1}.  Let $\text{sym}(n): \SL_2(\C) \to GL(V)$ be the $n$-dimensional representation on the space $V$ of homogeneous polynomials of degree $n-1$ in two variables $x$ and $y$. Then there is a correspondence (see \cite[\S 6]{Roh})
$$\begin{tikzcd}
 \text{Rep}(W_F\ltimes \C)   \ar[r, leftrightarrow]  &  \text{Rep}(WD_F)
\end{tikzcd} $$
determined by
$$\pi \otimes \text{sp}(n)   \leftrightarrow  (\pi \otimes \val{\cdot}^{\frac{n-1}{2}}) \boxtimes \text{sym}(n).$$
More explicitly, if $(\pi, N) \in  \text{Rep}(W_F\ltimes \C) $ corresponds to $\pi' \in \text{Rep}(WD_F)$, then 
$$N=\pi'(e_{-\alpha}(1)) \text{ for the unique simple root }  \alpha \in \Delta(\SL_2)$$
 and 
$$\pi(g)= \pi' \left(g,  \begin{bmatrix} \val{g}^{-1/2} & 0 \\ 0 & \val{g}^{1/2}  \end{bmatrix}   \right).$$
Note that $\pi$ is not the restriction of $\pi'$ to $W_F$. On the other hand, assume $\pi'=\pi_o\boxtimes \tau$. Then $\tau: \SL_2(\C) \to GL(V)$ is the unique representation such that $\tau(e_{-\alpha}(1)) = N$, and
\begin{equation} \label{cor2}
\pi_o(g)= \pi(g) \cdot \tau\left( \begin{bmatrix} \val{g}^{1/2} & 0 \\  0 &  \val{g}^{-1/2}  \end{bmatrix}  \right).
\end{equation}

In any case, one can define the $L$-function $L(s, \pi')$ with $\pi' \in {\rm Rep}(WD_F)$ corresponding to $(\pi, N)$, which is given by
$$L(s, \pi_o\boxtimes \tau):= \det\left( 1-q^{-s} \text{Frob} |_{V_N^I}  \right)^{-1}.$$
Here $\text{Frob}$ denotes a geometric Frobenius, $I\subset W_F$ the inertia group, and $V_N^I=\text{Ker}(N) \cap V^I$.

\vskip 5pt

We would like to associate an $L$-parameter and $L$-function to every $\pi_{\Gamma}$, where an $L$-parameter is just a splitting $\lrho_\Gamma$ of the $L$-group extension:
$$\begin{tikzcd}
\wt{G}_{Q,n}^\vee \ar[r, hook]  &  {}^L \wt{G}  \ar[r, two heads] &  WD_F.
\end{tikzcd}$$
First, by the local Langlands correspondence for covering torus (cf. \cite[\S 10]{We6} or \cite[\S 8]{GG}) and \eqref{L-comp}, to each $I(\chi)$ we have an associated parameter
$$\begin{tikzcd}
\lrho_\chi:  W_F \ar[r]  & {}^L \wt{T}  \ar[r, hook]   & {}^L\wt{G}.
\end{tikzcd}$$
For $\pi_\Gamma$ we denote (see \cite[\S 5.2]{Rod6})
$$N:=\sum_{ \substack{\alpha\in \Phi(\chi) \\ \alpha^\vee(\Gamma) >0 }} E_{\alpha_{Q,n}^\vee},$$
where $E_{\alpha_{Q,n}^\vee}$ is the pinned basis in the Lie algebra $\text{Lie}(\wt{G}_{Q,n}^\vee)$ associated to the root $\alpha_{Q,n}^\vee$ of $\wt{G}_{Q,n}^\vee$. 

\begin{lm}
The map $\lrho_\Gamma: W_F \ltimes \C \to {}^L\wt{G}$ given by
$$\lrho_\Gamma(a, z)= \lrho_\chi(a) \cdot (zN)$$
is a well-defined homomorphism.
\end{lm}
\begin{proof}
It suffices to show that $\lrho_\Gamma$ respects the group law on $W_F\ltimes \C$, i.e., 
$$\lrho_\chi(a) \cdot N \cdot \lrho_\chi(a)^{-1} = \val{a} \cdot N \text{ for all } a\in W_F.$$
Note that $\lrho_\chi$ factors through $F^\times$. Thus, by abuse of notation, we assume that $a\in F^\times$. For all $\alpha\in \Phi(\chi)$, one has
$$\begin{aligned}
& \lrho_\chi(a) \cdot E_{\alpha_{Q,n}^\vee} \cdot \lrho_\chi(a)^{-1} \\
= &   \chi( \wt{h}_\alpha(a^{n_\alpha}) ) \cdot E_{\alpha_{Q,n}^\vee} \text{ by  \cite[Theorem 7.8]{Ga1} } \\
= & \val{a} \cdot E_{\alpha_{Q,n}^\vee} \text{ since  } \alpha\in \Phi(\chi).
\end{aligned}$$
It follows that $\lrho(a) \cdot N \cdot \lrho(a)^{-1} = \val{a}\cdot N$, and this concludes the proof.
\end{proof}

By the description in \eqref{cor2}, which also applies to the ${}^L\wt{G}$-valued parameter $\lrho_\Gamma$, one obtains (by abuse of notation) an $L$-parameter
$$\begin{tikzcd}
\lrho_\Gamma:   WD_F  \ar[r]  & {}^L\wt{G}
\end{tikzcd}$$
in ${\rm Rep}(WD_F)$. The association $\pi_\Gamma \mapsto \lrho_\Gamma$ satisfies the desiderata for the local Langlands correspondence (see \cite[\S 10]{Bor}). In fact, for linear algebraic $G$, this association follows as a special case of the work of Kazhdan-Lusztig \cite{KL2}, where a local Langlands correspondence is established for Iwahori-fixed representations.

The following for linear algebraic groups is stated in \cite[Proposition2]{Rod6}, and it holds for covering groups by Theorem \ref{T:Rodier}.

\begin{prop}
The representation $\pi_\Gamma$ is square-integrable modular $Z(\wt{G})$ if and only if the image of $\lrho_\Gamma$ is not contained in a proper Levi subgroup of ${}^L\wt{G}$. Moreover, $\pi_\Gamma$ is tempered if and only if $\lrho_\Gamma(W_F)$ is bounded.
\end{prop}

For a representation $r: {}^L\wt{G} \to GL(V)$, one has the local Artin $L$-function $L(s, r\circ \lrho_\Gamma)$ as above.  Assume $G$ is a semisimple group. Then, from \eqref{L-iso}, there exists a representation
$$\begin{tikzcd}
r: {}^L \wt{G} \ar[r, "{\chi_\psi}"] & \wt{G}_{Q,n}^\vee \ar[r, hook]  & GL(V),
\end{tikzcd} $$
where $\wt{G}_{Q,n}^\vee \into GL(V)$ is an embedding of the semisimple group $\wt{G}_{Q,n}^\vee$ in some $GL(V)$. It remains a problem of studying such $L$-function $L(s, r\circ \lrho_\Gamma)$ by either a theory of zeta integral or the Langlands-Shahidi method. See \cite{Rod6} for the work on linear algebraic $G$ and \cite{CFGK1, CFK1, Kap01} for the yet developing theory in the covering setting.

\section{Whittaker space of the irreducible constituents} \label{S:WhGa}
\subsection{Whittaker functionals} Let $\psi: F\to \C^\times$ be an additive character of conductor $O_F$. Let
$$\psi_U: U\to \C^\times$$
be the character on $U$ such that its restriction to every $U_\alpha, \alpha\in \Delta$ is given by $\psi \circ e_\alpha^{-1}$. We write $\psi$ for $\psi_U$ for simplicity.

\begin{dfn}
For a genuine representation $(\pi, V_\pi)$ of $\wt{G}$, a linear functional $\ell: V_\pi \to \C$ is called a $\psi$-Whittaker functional if $\ell(\pi(u) v)= \psi(u) \cdot v$ for all $u\in U$ and $v\in V_\pi$.
Write $\Wh(\pi)$ for the space of $\psi$-Whittaker functionals for $\pi$.
\end{dfn}

The space $\Wh(I(\chi))$ for an unramified genuine principal series could be described as follows.
\begin{enumerate}
\item[$\bullet$] First, let $\Ftn(i(\chi))$ be the vector space of  functions $\cc$ on $\wt{T}$  satisfying
$$\cc(\wt{t} \cdot \wt{z}) =  \cc(\wt{t}) \cdot \chi(\wt{z}), \quad \wt{t} \in \wt{T} \text{ and } \wt{z} \in \wt{A}.$$
The support of $\cc \in \Ftn(i(\chi))$ is a disjoint union of cosets in $\wt{T}/\wt{A}$. We have
$$\dim \Ftn(i(\chi))=\val{\msc{X}_{Q,n}},$$
since $\wt{T}/\wt{A} \simeq Y/Y_{Q,n} = \msc{X}_{Q,n}$. 
\item[$\bullet$] Second, let $\set{\gamma_i}\subseteq \wt{T}$ be a set of representatives of $\wt{T}/\wt{A}$. Consider $\cc_{\gamma_i} \in \Ftn(i(\chi))$ which has support $\gamma_i \cdot \wt{A}$ and $\cc_{\gamma_i}(\gamma_i)=1$. It gives rise to a linear functional $\lambda_{\gamma_i}^{\chi} \in i(\chi)^\vee$ such that 
$$\lambda_{\gamma_i}^{\chi}(f_{\gamma_j})=\delta_{ij},$$
 where $f_{\gamma_j}\in i(\chi)$ is the unique element such that $\text{supp}(f_{\gamma_j})=\wt{A}\cdot \gamma_j^{-1}$ and $f_{\gamma_j}(\gamma_j^{-1})=1$. That is, $f_{\gamma_j}=i(\chi)(\gamma_j)\phi_0$, where $\phi_0\in i(\chi)$ is the normalized unramified vector of $i(\chi)$ such that $\phi_0(1_{\wt{T}})=1$.  Denote by $i(\chi)^\vee$ the complex dual space of functionals of $i(\chi)$. Then one has
an isomorphism of vector spaces
$$\Ftn(i(\chi)) \simeq i(\chi)^\vee$$ given explicitly by
$$ \cc   \mapsto   \lambda_\cc^{\chi}:= \sum_{\gamma_i \in \wt{T}/\wt{A}} \cc(\gamma_i) \lambda_{\gamma_i}^{\chi}.
$$
The isomorphism does not depend on the choice of representatives for $\wt{T}/\wt{A}$.
\item[$\bullet$] Third, there is an isomorphism between $i(\chi)^\vee$ and the space $\Wh(I(\chi))$ (see \cite[\S 6]{Mc2}), given by $\lambda \mapsto W_\lambda$ with
$$W_\lambda:  I(\chi) \to \C, \quad f \mapsto \lambda \left( \int_{U} f(\wt{w}_G^{-1}u) \psi(u)^{-1} \mu(u) \right),$$
where $f\in I(\chi)$ is an $i(\chi)$-valued function on $\wt{G}$. Here $\wt{w}_G= \wt{w}_{\alpha_1} \wt{w}_{\alpha_2} ... \wt{w}_{\alpha_k}\in K$ is the representative of $w_G$ chosen in \S \ref{Sec:SF}.
\end{enumerate}
Thus, we have a composite of natural isomorphisms of vector spaces of dimension $\val{\msc{X}_{Q,n}}$:
$$\Ftn(i(\chi))\simeq i(\chi)^\vee \simeq \Wh(I(\chi)).$$
For any $\cc\in \Ftn(i(\chi))$, by abuse of notation, we will write $\lambda_\cc^{\chi} \in \Wh(I(\chi))$ for the resulting $\psi$-Whittaker functional of $I(\chi)$ from the above isomorphism.

\subsection{Scattering matrix}
The operator $T(w,\chi): I(\chi) \to I({}^w \chi)$ induces a homomorphism of vector spaces
$$T(w, \chi)^*: \Wh(I({}^w \chi)) \to \Wh(I(\chi)) $$
given by 
$$\angb{\lambda_\cc^{{}^w \chi} }{-} \mapsto \angb{\lambda_\cc^{{}^w \chi} }{T(w,\chi)(-)}$$
for any $\cc \in \Ftn(i({}^w \chi))$. Let $\set{\lambda_{\gamma}^{^w \chi}}_{\gamma \in \wt{T}/\wt{A}}$ be a basis for  $\Wh(I({}^w \chi))$, and $\set{ \lambda_{\gamma'}^{\chi} }_{\gamma \in \wt{T}/\wt{A}}$ a basis for $\Wh(I(\chi))$. The map $T(w,\chi)^*$ is
then determined by the square matrix
$$[\tau(w, \chi, \gamma, \gamma')]_{\gamma, \gamma'\in \wt{T}/\wt{A}}$$
such that
\begin{equation} \label{tau-Mat}
T(w, \chi)^*(\lambda_{\gamma}^{^{w} \chi}) = \sum_{\gamma'\in \wt{T}/\wt{A}} \tau(w, \chi, \gamma, \gamma') \cdot \lambda_{\gamma'}^{\chi}.
\end{equation}
The matrix $[\tau(w, \chi, \gamma, \gamma')]$ is a so-called scattering matrix associated to $T(w, \chi)^*$. It satisfies some immediate properties:
\begin{enumerate}
\item[$\bullet$] For $w\in W$ and $\wt{z}, \wt{z}'\in \wt{A}$, the identity 
\begin{equation} \label{SLCM1}
\tau(w, \chi, \gamma \cdot \wt{z}, \gamma' \cdot \wt{z}')=({}^{w}\chi)^{-1}(\wt{z}) \cdot \tau(w, \chi, \gamma, \gamma') \cdot \chi(\wt{z}')
\end{equation}
holds.
\item[$\bullet$] For $w_1, w_2 \in W$ such that $l(w_2w_1)=l(w_2) + l(w_1)$, one has
\begin{equation} \label{SLCM2}
\tau( w_2 w_1, \chi,  \gamma, \gamma')=\sum_{\gamma''\in \wt{T}/\wt{A}} \tau(w_2, {}^{w_1}\chi, \gamma, \gamma'') \cdot \tau(w_1, \chi, \gamma'', \gamma'),
\end{equation}
which is referred to as the cocycle relation.
\end{enumerate}

In view of the cocycle relation in (\ref{SLCM2}), the understanding  of $\tau(w, \chi, \gamma, \gamma')$ in principle is reduced to the case where $w=w_\alpha$ for some $\alpha\in \Delta$. For this purpose, we first introduce the Gauss sum.

Let $du$ be the self-dual Haar measure of $F$ such that $du(O)=1$; thus, $du(O^\times)=1-q^{-1}$. The Gauss sum is defined by
$$G_{\psi}(a, b)=\int_{O^\times} (u, \varpi)_n^a \cdot \psi(\varpi^b u) du \text{ for } a, b\in \Z.$$
Denote
$$\mathbf{g}_{\psi}(k):=G_{\psi}(k, -1),$$
where $k\in \Z$ is any integer. We write henceforth
$$\vep:= (-1, \varpi)_n \in \C^\times.$$ 
It is known that
\begin{equation} \label{F:gauss}
\mbf{g}_{\psi}(k)=
\begin{cases}
\vep^k \cdot \overline{\mbf{g}_{\psi}(-k)} & \text{ for any  } k\in \Z, \\
-q^{-1} & \text{ if } n| k,\\
\mbf{g}_{\psi}(k) & \text{ with  }  \val{\mbf{g}_{\psi}(k)}=q^{-1/2}  \text{ if } n \nmid k.\\
\end{cases}
\end{equation}
Here $\overline{z}$ denotes the complex conjugation of a complex number $z$.

It is shown in \cite{KP, Mc2} (with some refinement from \cite{Ga2}) that $\tau(w_\alpha, \chi, \gamma, \gamma')$ is determined as follows.

\begin{thm} \label{T:tau}
Suppose that $\gamma=\s_{y_1}$ and $\gamma'=\s_y$.  First, we can write $\tau(w_\alpha, \chi, \gamma, \gamma')=\tau^1(w_\alpha, \chi, \gamma, \gamma') + \tau^2(w_\alpha, \chi, \gamma, \gamma')$ with the following properties:
\begin{enumerate} 
\item[$\bullet$] 
$\tau^i(w_\alpha, \chi, \gamma \cdot \wt{z}, \gamma' \cdot \wt{z}')=({}^{w_\alpha} \chi)^{-1}(\wt{z}) \cdot \tau^i( w_\alpha, \chi, \gamma, \gamma') \cdot \chi(\wt{z}'), \quad\text{ for all } \wt{z}, \wt{z}'\in \wt{A} $;
\item[$\bullet$] 
$\tau^1(w_\alpha, \chi, \gamma, \gamma')=0$  unless  $y_1 \equiv y \mod Y_{Q,n}$;
\item[$\bullet$] $\tau^2(w_\alpha, \chi, \gamma, \gamma')=0$   unless $y_1 \equiv w_\alpha[y] \mod Y_{Q,n}$.
\end{enumerate}
Second, 
\begin{enumerate}
\item[$\bullet$] if $y_1= y$, then 
$$\tau^1(w_\alpha, \chi, \gamma, \gamma')=(1-q^{-1}) \frac{\chi (\wt{h}_\alpha(\varpi^{n_\alpha}))^{k_{y,\alpha}}}{1-\chi (\wt{h}_\alpha(\varpi^{n_\alpha}))}, \text{ where } k_{y,\alpha}=\ceil{\frac{\angb{y}{\alpha}}{n_\alpha}};$$
\item[$\bullet$] if $y_1=w_\alpha[y]$, then
$$\tau^2(w_\alpha, \chi, \gamma, \gamma') = \vep^{ \angb{y_\rho}{\alpha} \cdot D(y, \alpha^\vee) } \cdot \g(\angb{y_\rho}{\alpha}Q(\alpha^\vee)).$$
\end{enumerate}
\end{thm}

Let $\chi$ be a regular unramified genuine character and $\pi_\Gamma$ a constituent of $I(\chi)$. Suppose $\pi=\Ima(T(w,\chi))$. Then
$$\dim \Wh (\pi) = \text{rank of the matrix } [\tau(w, \chi, \gamma, \gamma')]_{\gamma, \gamma'\in \wt{T}/\wt{A}},$$
which is a well-defined complex-valued square matrix by Lemma \ref{L:reg}. However, unless $w=w_\alpha$, the computation of the rank of the matrix is not straightforward and involves non-trivial combinatorial problem arising from the cocycle relation. 

By abuse of notation, we may denote by 
$$[\tau(w, \chi, \s_y, \s_z)]_{y, z\in \msc{X}_{Q,n}}$$
the above matrix, assuming that we have chosen a fixed set of representatives in $Y$ of $\msc{X}_{Q,n}$. Since we are essentially interested in the rank of $[\tau(w, \chi, \s_y, \s_z)]_{y, z\in \msc{X}_{Q,n}}$, which is independent of the choice of cosets representative for $\msc{X}_{Q,n}$ by \eqref{SLCM1}; this ambiguity does not detract from all our argument.

\begin{rmk}
We call $[\tau(w, \chi, \s_y, \s_z)]_{y, z\in \msc{X}_{Q,n}}$ a scattering matrix following \cite{GSS2} (instead of local coefficients matrix), as the matrix is not the ``correct" generalization or analogue of Shahidi's local coefficients for linear algebraic groups. For instance, the characteristic polynomial of the matrix $[\tau(w, \chi, \s_y, \s_z)]_{y, z\in \msc{X}_{Q,n}}$ actually depends on the choice of representatives of $\msc{X}_{Q,n}$ for $\chi$ in general position. Such subtleties are explained in details in \cite{GSS2}. However, since we are concerned only about the rank of the matrix, the usage of the scattering matrix suffices and actually is preferred, for the reason of utilizing the twisted Weyl action which describes the entries of $[\tau(w, \chi, \s_y, \s_z)]_{y, z\in \msc{X}_{Q,n}}$ as elucidated by Theorem \ref{T:tau}.
\end{rmk}

For $w\in W$, let $\Delta_w \subset \Delta$ be such that
$$w=\prod_{\alpha \in \Delta_w} w_\alpha$$
in a minimum decomposition. By \cite[page 12, Proposition 7]{Bou}, $\Delta_w$ is independent of the minimal decomposition for $w$.
Let $W(\Delta_w) \subset W$ be the parabolic subgroup of $W$ generated by $\Delta_w$. In particular, $W(\Delta_{w_G})=W$ and $W(\Delta_\text{id})=\set{ \text{id} }$.

We denote by $\mca{O}_{\msc{X}_{Q,n}}^w$ the set of $W(\Delta_w)$-orbits in $\msc{X}_{Q,n}$. Denote by $\mca{O}_y^w \in \mca{O}_{\msc{X}_{Q,n}}^w$ any such orbit of $y \in \msc{X}_{Q,n}$. In particular, $\mca{O}_y^{w_G}=\mca{O}_y$ represents the usual $W$-orbit of $y$. Note that our notation here is consistent with that in \eqref{dc-O1}.

\subsection{$\mca{O}_y$-relative Whittaker subspace}
For any $w\in W$, we have a decomposition
$$\msc{X}_{Q,n} = \bigsqcup_{i \in I_w}  \mca{O}^w_{y_i},$$
where $\val{I_w}$ is just the number of $W(\Delta_w)$-orbits in $\msc{X}_{Q,n}$. In particular,
$$\msc{X}_{Q,n}=\bigsqcup_{i\in I_{w_G}} \mca{O}_{y_i}.$$

Now for any $\chi$ and $\mca{O}_y \subset \msc{X}_{Q,n}$, let
$$\Wh (I(\chi))_{\mca{O}_y}= \text{Span}\set{\lambda^{\chi}_{\s_z}: z\in \mca{O}_y   } \subset \Wh(I(\chi))$$
be the ``$\mca{O}_y$-subspace" of the Whittaker space of $I(\chi)$.
It is well-defined and independent of the representatives for $\mca{O}_y$, and
$$\dim \Wh (I(\chi))_{\mca{O}_y}= \val{ \mca{O}_y }.$$
Moreover, one  has a decomposition
$$\Wh (I(\chi)) = \bigoplus_{\mca{O}_y \in \mca{O}_{\msc{X}_{Q,n}}}  \Wh (I(\chi))_{\mca{O}_y}.$$

\begin{prop} \label{P:decom}
Let $\chi$ be a regular unramified character of $Z(\wt{T})$. For every $\mca{O}_y \subseteq \msc{X}_{Q,n}$ and $w\in W$, the restriction of $T(w, \chi)^*: \Wh(I({}^w \chi)) \to \Wh(I(\chi)) $  to $\Wh(I({}^w \chi))_{\mca{O}_y}$  gives a well-defined homomorphism
$$T(w, \chi)^*_{\mca{O}_y}: \Wh(I({}^w \chi))_{\mca{O}_y} \to \Wh(I(\chi))_{\mca{O}_y}.$$
Moreover,
$$T(w, \chi)^*= \bigoplus_{\mca{O}_y \in \mca{O}_{\msc{X}_{Q,n}}} T(w, \chi)^*_{\mca{O}_y},$$
where the sum is taken over all $W$-orbits in $\msc{X}_{Q,n}$.
\end{prop}
\begin{proof}
The well-definedness of $T(w, \chi)^*_{\mca{O}_y}$ follows from the cocycle condition in \eqref{SLCM2} and the property of $\tau(w, \chi, \s_y, \s_z)$  in Theorem \ref{T:tau}, which shows that $\tau(w, \chi, \s_y, \s_z)=0$ unless $y$ and $z$ lie in the same $W(\Delta_w)$-orbit. In particular,  $\tau(w, \chi, \s_y, \s_z)=0$ unless $y$ and  $z$ lie in the same $W$-orbit. In fact, the operator $T(w, \chi)^*_{\mca{O}_y}$ is just represented by the square-matrix $[\tau( w, \chi, \s_y, \s_z)]_{y, z \in \mca{O}_y  }$. The decomposition is clear.
\end{proof}

\begin{cor} \label{C:rk-inv}
Let $\chi$ be a regular unramified character. Let  $\Gamma, \Gamma' \in \msc{C}(X\otimes \R; \chi)$ be two components. For every orbit $\mca{O}_y \subseteq \msc{X}_{Q,n}$, the rank of $T(w_2^{-1} w_1, {}^{w_1^{-1}}\chi)_{\mca{O}_y}^*, w_2\in W_{\Gamma}, w_1\in W_{\Gamma'}$ is independent of the choice of $w_1, w_2$.
\end{cor}
\begin{proof}
We prove the independence on $w_2$, while that on $w_1$ can be argued in a similar way. Let $w_2' \in W_{\Gamma}$ be another element. It follows from Proposition \ref{J-ker} that the basis 
$$T:=T(w_2'^{-1} w_2, {}^{w_2^{-1}}\chi)   \in \Hom(I({}^{w_2^{-1}}\chi), I({}^{w_2'^{-1}}\chi))$$
is an isomorphism. Consider the decomposition of $T^*$ from Proposition \ref{P:decom}:
$$T^*= \bigoplus_{\mca{O}_y \in \mca{O}_{\msc{X}_{Q,n}}}  T_{\mca{O}_y}^* .$$
Then 
$$T_{\mca{O}_y}^*:  \Wh(I({}^{w_2'^{-1}}\chi)))_{\mca{O}_y} \to \Wh(I({}^{w_2^{-1}}\chi)))_{\mca{O}_y}.$$
is also an isomorphism. On the other hand,
we have the equality
$$T(w_2'^{-1} w_1, {}^{w_1^{-1}}\chi)^* =  T(w_2^{-1} w_1, {}^{w_1^{-1}}\chi)^* \circ T^*$$
of operators. It follows from Proposition \ref{P:decom} again that
$$T(w_2'^{-1} w_1, {}^{w_1^{-1}}\chi)^*_{\mca{O}_y} =  T(w_2^{-1} w_1, {}^{w_1^{-1}}\chi)^*_{\mca{O}_y} \circ T^*_{\mca{O}_y}.$$
Therefore, 
$$\text{rank}(T(w_2'^{-1} w_1, {}^{w_1^{-1}}\chi)^*_{\mca{O}_y})= \text{rank}(T(w_2^{-1} w_1, {}^{w_1^{-1}}\chi)^*_{\mca{O}_y}),$$
as desired.
\end{proof}

For $\Gamma, \Gamma' \in \msc{C}(X\otimes \R; \chi)$, we denote by
$$\pi_{\Gamma', \Gamma} \subseteq I({}^{w_2^{-1}}\chi)$$
the image of the intertwining operator $T(w_2^{-1} w_1, {}^{w_1^{-1}}\chi)$ with $w_1\in W_{\Gamma'}, w_2 \in W_{\Gamma}$. The isomorphism class of $\pi_{\Gamma', \Gamma}$ is independent of the choice of $w_1, w_2$. In this notation, we have
 $$\pi_{\Gamma} = \pi_{\Gamma_\text{op}, \Gamma}.$$

\begin{dfn}
For $\mca{O}_y\subset \msc{X}_{Q,n}$, the $\mca{O}_y$-dimension of the Whittaker space of $\pi_{\Gamma', \Gamma}$ is 
$$\dim \Wh(\pi_{\Gamma', \Gamma})_{\mca{O}_y}:=\text{rank}(T(w_2^{-1} w_1, {}^{w_1^{-1}}\chi)_{\mca{O}_y}^*),$$
where $w_1 \in W_{\Gamma'}$ and $w_2 \in W_\Gamma$.
\end{dfn}
It follows from Corollary \ref{C:rk-inv} that $\dim \Wh(\pi_{\Gamma', \Gamma})_{\mca{O}_y}$ is well-defined, independent of the choice of $w_1$ and $w_2$. Proposition \ref{P:decom} gives
\begin{equation} \label{Wh-sum}
\dim \Wh(\pi_{\Gamma',\Gamma})=\sum_{\mca{O}_y \in \mca{O}_{\msc{X}_{Q,n}}}  \dim \Wh(\pi_{\Gamma', \Gamma})_{\mca{O}_y}.
\end{equation}

\begin{conj} \label{C:O-ex}
Suppose $\pi_{\Gamma', \Gamma} = \bigoplus_{i\in I} \pi_{\Gamma_i} \in \msc{R}(\wt{G})$ in the Grothendieck group $\msc{R}(\wt{G})$ of ${\rm Irr}_\epsilon(\wt{G})$. Then for every orbit $\mca{O}_y \subseteq \msc{X}_{Q,n}$,
$$\dim \Wh(\pi_{\Gamma', \Gamma})_{\mca{O}_y}  =\sum_{i\in I} \dim \Wh(\pi_{\Gamma_i})_{\mca{O}_y}.$$
\end{conj}

If $W$ acts transitively on $\msc{X}_{Q,n}$, i.e., there is only one $W$-orbit (this occurs rarely and only if $n$ is sufficiently small compared to the size of $W$), then by the exactness of the function $\dim \Wh(-): \msc{R}(\wt{G}) \to \Z$,  Conjecture \ref{C:O-ex} holds.

\subsection{A coarse formula for $\dim \Wh(\pi_\Gamma)_{\mca{O}_y}$} \label{SS:coarseF}
We assume $\Phi(\chi) \subseteq \Delta$ in this subsection. For every $S \in \msc{P}(\Phi(\chi))$, we define
$$\Gamma_S^\natural:= \bigcap_{\alpha\in S} \alpha^\vee(\R_{<0})  = \set{x \in X\otimes \R:  \ \angb{x}{\alpha^\vee}< 0 \text{ for every } \alpha \in S}.$$
Clearly $\Gamma_S \subseteq \Gamma_S^\natural$.  For $S \subset \Phi(\chi)$, denote by $P_S  = M_S U_S \subseteq G$ the associated parabolic subgroup. Consider the representation $\pi_{M_S, \Gamma_S} \in \JH( I_{\wt{B}_M}^{\wt{M}_S} (\chi) )$, which is the unique Langlands quotient of $I_{\wt{B}_M}^{\wt{M}_S}  (\chi)$; that is, the image of the intertwining map
$$\hat{T}(w_S,\chi): I_{\wt{B}_M}^{\wt{M}} (\chi) \to I_{\wt{B}_M}^{\wt{M}} ({}^{w_S}\chi)$$ 
between principal series of $\wt{M}$. Here $\chi$ is an exceptional character for $\wt{M}_S$ in the sense of \cite{KP} and thus $\pi_{M_S, \Gamma_S}$ is the theta representation of $M_S$.

We denote
$$\pi_{\Gamma_S^\natural}  :=  \Ind_{\wt{P}_S}^{\wt{G}} \ \pi_{M_S, \Gamma_S} \otimes \mbm{1} $$
and let $W_{\Gamma_S^\natural} \subset W$ be such that 
$$(\pi_{\Gamma_S^\natural})_U = \bigoplus_{w \in W_{\Gamma_S^\natural}} \delta_B^{1/2} \cdot i({}^{w^{-1}} \chi).$$
It follows from induction by stages and thus the following commutative diagram
$$\begin{tikzcd}
I(\chi)  \ar[rr, "{T(w_S,\chi)}"] \ar[d, equal] & & I({}^{w_S}\chi) \ar[d, equal] \\
I_{\wt{P}}^{\wt{G}} \big( I_{\wt{B}_M}^{\wt{M}} (\chi) \big) \ar[rr, "{ I_{\wt{P}}^{\wt{G}} (\hat{T}(w_S,\chi))   }"] & & I_{\wt{P}}^{\wt{G}} \big( I_{\wt{B}_M}^{\wt{M}} ({}^{w_S}\chi) \big)
\end{tikzcd}$$
that  $\pi_{\Gamma_S^\natural} = {\rm Im}(T(w_S, \chi))$.

\begin{lm} \label{L:key}
For every $S \subseteq \Phi(\chi)$, one has  $\mbm{1}_{\Gamma_{S'}^\natural} =  \sum_{S': \  S \subseteq S' \subseteq \Phi(\chi) }  \mbm{1}_{\Gamma_{S'}} $ and thus
\begin{equation*} \label{E:alt-s}
\mbm{1}_{\Gamma_S} = \sum_{S': \  S \subseteq S' \subseteq \Phi(\chi) }  (-1)^{\val{S'-S}} \cdot \mbm{1}_{\Gamma_{S'}^\natural}.
\end{equation*}
Moreover, 
$$\pi_{\Gamma_S^\natural} \simeq \bigoplus_{\Gamma \subseteq \Gamma_S^\natural}  \pi_\Gamma \in \msc{R}(\wt{G}).$$
\end{lm}
\begin{proof}
The first assertion follows from the inclusion-exclusion principle. For the second assertion, it suffices to check for the Jacquet modules on the two sides. The Jacquet module of $\bigoplus_{\Gamma \subseteq \Gamma_S^\natural}  \pi_\Gamma$ is clearly indexed by 
\begin{equation} \label{ind-tem}
\bigcup_{\Gamma \subseteq \Gamma_S^\natural }  W_\Gamma = \set{w\in W: \Phi(\chi)^\vee \cap w(\Phi_-^\vee) \supseteq S^\vee}.
\end{equation}
On the other hand, since $\pi_{\Gamma_S^\natural}$ is the image of the intertwining operator 
$$T(w_S, \chi): I(\chi) \to I({}^{w_S} \chi),$$
where $w_S \in W_S$ is the longest element in the Weyl group of $(M_S, T)$, by applying $w_1= {\rm id}, w_2 = w_S$ and $\Phi(\chi) \subset \Delta$ in \eqref{J-im}, it follows that the Jacquet module of $\pi_{\Gamma_S^\natural}$ is indexed by 
$$\set{w\in W:  \Phi(\chi)^\vee \cap w_S(\Phi_-^\vee) \subseteq w(\Phi_-^\vee)},$$
which can be checked easily to be equal to \eqref{ind-tem}. This completes the proof.
\end{proof}


Immediately from Lemma \ref{L:key} we have:
\begin{prop} \label{P:c-dim}
For every $S\subseteq \Phi(\chi)$, one has
$$\dim \Wh(\pi_{\Gamma_S}) = \sum_{S': \  S \subseteq S' \subseteq \Phi(\chi) }  (-1)^{\val{S'-S}} \cdot \dim \Wh(\pi_{\Gamma_{S'}^\natural}).$$
Further more, if Conjecture \ref{C:O-ex} holds, then for every orbit $\mca{O}_y \subset \msc{X}_{Q,n}$,
$$\dim \Wh(\pi_{\Gamma_S})_{\mca{O}_y} = \sum_{S': \  S \subseteq S' \subseteq \Phi(\chi) }  (-1)^{\val{S'-S}} \cdot \dim \Wh(\pi_{\Gamma_{S'}^\natural})_{\mca{O}_y}.$$
\end{prop}

We would like to show that the second assertion in Proposition \ref{P:c-dim} holds unconditionally, if $\val{\Phi(\chi)} \le 2$. We write $\pi_S:=\pi_{\Gamma_S}$ for every $S\subseteq \Phi(\chi)$. To proceed, consider general $\Phi(\chi) \subset \Delta$ and $S_o\subset \Phi(\chi)$ with $S= S_o \sqcup \set{\alpha} \subseteq \Phi(\chi)$. Denote by $w_{S_o}, w_S$ the longest Weyl element in $W(S_o)$ and $W(S)$ respectively. Considering the chain of intertwining operators

$$\begin{tikzcd}
I({}^{w_\alpha} \chi)  \ar[r, "T_1"]   & I(\chi)  \ar[r, "T_2"]  & I({}^{w_{S_o}} \chi)    \ar[r, "T_3"]  & I({}^{w_S} \chi),
\end{tikzcd}$$
one has
$$\pi_{S_o}^\natural = {\rm Im}(T_2)  \text{ and }    \pi_{S}^\natural = {\rm Im}(T_3\circ T_2).$$
We also denote
$$\pi_{S_o, S}^\natural:= {\rm Im}(T_2 \circ T_1).$$

\begin{lm} \label{L:deco}
Inside $\msc{R}(\wt{G})$, one has
$$\pi_{S_o, S}^\natural = \bigoplus_{ \substack{S': \\ S_o \subseteq S' \subseteq \Phi(\chi) \backslash \set{\alpha} } }  \pi_{S'}, \text{ and therefore } \pi_{S_o}^\natural = \pi_{S_o, S}^\natural \oplus \pi_S^\natural.$$
\end{lm}
\begin{proof}
It suffices to check that the Jacquet module of $\pi_{S_o, S}^\natural$ is indexed by the set
$$\set{w: S_o \subseteq \Phi(\chi)^\vee \cap w(\Phi_-^\vee)  \subseteq \Phi(\chi)^\vee \backslash \set{\alpha^\vee}},$$
which follows easily from \eqref{J-im}.
\end{proof}

\begin{prop} \label{P:decW}
For every orbit $\mca{O}_y \subseteq \msc{X}_{Q,n}$, one has
$$\dim \Wh(\pi_{S_o}^\natural)_{\mca{O}_y} = \dim \Wh(\pi_{S_o, S}^\natural)_{\mca{O}_y} +  \dim \Wh(\pi_S^\natural)_{\mca{O}_y}.$$
\end{prop}
\begin{proof}
In view of \eqref{Wh-sum} and the exactness of the function $\pi \mapsto \dim \Wh(\pi)$, it suffices to show that for every orbit $\mca{O}_y$, one has the inequality
\begin{equation} \label{E:ine}
\dim \Wh(\pi_{S_o}^\natural)_{\mca{O}_y} \ge \dim \Wh(\pi_{S_o, S}^\natural)_{\mca{O}_y} +  \dim \Wh(\pi_S^\natural)_{\mca{O}_y}.
\end{equation}
Keeping notations as above, we have
$$\begin{tikzcd}
\Wh(I({}^{w_S} \chi) )  \ar[r, "T_3^*"] & \Wh(I({}^{w_{S_o}} \chi))  \ar[r, "T_2^*"]  & \Wh(I(\chi))  \ar[r, "T_1^*"] &  \Wh(I({}^{w_\alpha} \chi)),
\end{tikzcd}$$
where by definition
$$\dim \Wh(\pi_{S_o}^\natural)_{\mca{O}_y} = {\rm rank}((T_2^*)_{\mca{O}_y}),  \quad \dim \Wh(\pi_S^\natural)_{\mca{O}_y}= {\rm rank}( (T_2^* \circ T_3^*)_{\mca{O}_y}  ),$$
and
$$\dim \Wh(\pi_{S_o, S}^\natural)_{\mca{O}_y} = {\rm rank}( (T_1^* \circ T_2^*)_{\mca{O}_y}  ).$$
Since $T_3 \circ T_2 \circ T_1 =0$, we have $T_1^* \circ T_2^*  \circ T_3^* = 0$. Hence,
$$\begin{aligned}
&{\rm rank}((T_2^*)_{\mca{O}_y}) -  {\rm rank}( (T_1^* \circ T_2^*)_{\mca{O}_y}  ) \\
= & \dim {\rm Im} (T_2^*)_{\mca{O}_y} \cap  {\rm Ker}( (T_1^*)_{\mca{O}_y})   \\
\ge &  \dim {\rm Im} (T_2^*)_{\mca{O}_y} \cap  {\rm Im} (T_2^* \circ T_3^*)_{\mca{O}_y} \\
= & {\rm rank}( (T_2^* \circ T_3^*)_{\mca{O}_y}  ).
\end{aligned} $$
That is, \eqref{E:ine} holds as claimed.
\end{proof}

\begin{cor} \label{C:exac=2}
Assume $\Phi(\chi) \subseteq \Delta$ with $\val{\Phi(\chi)} \le 2$. Then for every $S \subseteq \Phi(\chi)$ and every orbit $\mca{O}_y \subseteq \msc{X}_{Q,n}$, one has
\begin{equation} \label{E:eqd=2}
\dim \Wh(\pi_S^\natural)_{\mca{O}_y} = \sum_{ \substack{S': \\  S \subseteq S' \subseteq \Phi(\chi)}  }  \dim \Wh(\pi_{S'})_{\mca{O}_y}.
\end{equation}
\end{cor}
\begin{proof}
If $S= \Phi(\chi)$, then \eqref{E:eqd=2} holds vacuously; in particular, there is nothing to check if $\Phi(\chi) =\emptyset$. If $\Phi(\chi) =\set{\alpha}$, then applying $S_o =\emptyset$ and $S=\set{\alpha}$ in Proposition \ref{P:decW} gives \eqref{E:eqd=2}. 

Now assume $\Phi(\chi) =\set{\alpha, \beta} \subseteq \Delta$. Then equality \eqref{E:eqd=2} for $S=\set{\alpha}$ or $\set{\beta}$ follows from Proposition \ref{P:decW}, as in this case $\pi_{S, \Phi(\chi)}^\natural = \pi_S$. Now we consider the case $S=\emptyset$. Applying Proposition \ref{P:decW} gives
$$\begin{aligned} 
\dim \Wh(I(\chi))_{\mca{O}_y} & = \dim \Wh(\pi_{\emptyset, \set{\alpha}}^\natural)_{\mca{O}_y} +  \dim \Wh(\pi_{ \set{\alpha} }^\natural)_{\mca{O}_y} \\
 & =  \dim \Wh(\pi_{\emptyset, \set{\alpha}}^\natural)_{\mca{O}_y} +  \dim \Wh(\pi_{ \set{\alpha} })_{\mca{O}_y}  + \dim \Wh(\pi_{\Phi(\chi)})_{\mca{O}_y},
\end{aligned}$$
where the last equality follows from the above case when $S=\set{\alpha}$. We see that 
$$ \pi_{\emptyset, \set{\alpha}}^\natural \simeq \pi_\emptyset \oplus \pi_{\set{\beta} } \in \msc{R}(\wt{G}).$$
By a similar argument as in the proof of Proposition \ref{P:decW}, one can show that 
$$\dim \Wh(\pi_{\emptyset, \set{\alpha}}^\natural )_{\mca{O}_y} = \dim \Wh (\pi_\emptyset )_{\mca{O}_y}  + \dim \Wh(\pi_{\set{\beta} })_{\mca{O}_y}.$$
This proves equality \eqref{E:eqd=2} for $S=\emptyset$. The proof is thus completed.
\end{proof}

\section{Kazhdan-Lusztig representations and the main conjecture} \label{S:KLcon}
The goal of this section is to give a conjectural formula for  $\dim \Wh (\pi_\Gamma)_{\mca{O}_y}$. For this purpose, we recall briefly the theory of Kazhdan-Lusztig representation of a Weyl group, which will be a key input in the subsequent discussion. The reader may consult the original paper \cite{KL1} and other work for more detailed exposition (see for example \cite{Lus4, Lus5, Shi, Deo, Hum, BB}).

\subsection{Right cells and Kazhdan-Lusztig representations}
Let $A:=\Z[q,q^{-1}]$ be the ring of Laurent polynomials over $\Z$ with indeterminate $q$. Let $(W,S)$ be the Weyl group generated by 
$$S=\set{w_\alpha: \alpha \in \Delta}.$$
We may also write $s$ for a generic element in $S$. Denote by $w < w'$ the Bruhat-Chevalley order on elements of $W$; that is, $w < w'$ if there is a reduced expression of $w'$ such that $w$ arises from a subexpression of it (which is then necessarily reduced).

The Hecke algebra $\msc{H}_W$ is the free $A$-module generated by $T_w, w\in W$ with relations given by:
\begin{enumerate}
\item[$\bullet$] $T_sT_w = T_{s w}$ if $l(s w) > l(w)$;
\item[$\bullet$] $(T_s +1)(T_s- q)=0$.
\end{enumerate}

Define an involution 
$$\iota: \msc{H}_W \to \msc{H}_W$$
as follows. First, the involution of $\iota$ on $A$ is given by sending $q$ to $q^{-1}$. Second, define $\iota(T_w):= (T_{w^{-1}})^{-1}$. It is shown in \cite{KL1} that $\iota$ gives rise to a well-defined involution on the whole algebra $\msc{H}_W$. 

We write
$$\varepsilon_w:=(-1)^{l(w)}, \quad q_w:=q^{l(w)}.$$

\begin{thm}[{\cite[Theorem 1.1]{KL1}}]
For every $w\in W$, there exists a unique element $C_w \in \msc{H}_W$ satisfying the following two properties:
\begin{enumerate}
\item[$\bullet$] $\iota(C_w) = C_w$;
\item[$\bullet$] $$C_w= \varepsilon_w q_w^{1/2} \cdot \sum_{x:\ x\le w}  \varepsilon_x q_x^{-1} \iota(P_{x,w}) T_x,$$
where $P_{w, w}=1$ and if $x< w$, then $P_{x,w}(q) \in \Z[q]$ is a polynomial of  degree $\le \frac{1}{2} (l(w)- l(x) -1)$.
\end{enumerate}
The set $\set{C_w}_{w\in W}$ forms an $A$-basis for $\msc{H}_W$.
\end{thm}
The Kazhdan-Lusztig polynomial $P_{x,w}$ above has attracted much attention for extensive research, partly due to its deep connection with some applications. For example, it is conjectured by Kazhdan-Lusztig that $P_{x,w}(1)$ is equal to the multiplicity in decomposing a certain Verma module, which is established in Brylinski-Kashiwara \cite{BK}, and independently by Beilinson-Bernstein \cite{BeiBer} with a sketched proof. For another example, we note that it was first shown by Tits (see \cite[page 59, Exercise 27]{Bou}) that one has $\C[W] \simeq \msc{H}_W \otimes_{A, f} \C$ whenever $f: A \to \C$ is a specialization algebra homomorphism such that $\msc{H}_W\otimes_{A,f} \C$ is semisimple, see also \cite[Proposition 10.11.2]{Car}. The method of Tits is to analyse the invariants associated to the algebras on the two sides. However, by using the polynomials $P_{x,w}$, Lusztig \cite{Lus1} gave an explicit construction of an isomorphism $\Q(q^{1/2})[W] \simeq \msc{H}_W \otimes_A \Q(q^{1/2})$. The proof by Lusztig adapts a more uniform approach, compared to the case by case analysis in the earlier work by Benson and Curtis \cite{BC}.

\vskip 5pt

We write $x \prec  w$ if
\begin{enumerate}
\item[$\bullet$] $x < w$, and 
\item[$\bullet$] $\text{deg}(P_{x, w}) = \frac{1}{2}(l(w) - l(x) -1)$. 
\end{enumerate}
If $x \prec w$, then we denote
$$\mu(x, w):=\text{coefficient of the highest power of $q$ in } P_{x, w}.$$
It is important to see how $T_s$ acts (on the right) on the new basis $\set{C_w: w\in W}$ of  $\msc{H}_W$. This is given as follows:

\begin{prop}[{\cite[(2.3.a)-(2.3.d)]{KL1}}] \label{P:action}
Let $s\in S$ and $w\in W$.
\begin{enumerate}
\item[$\bullet$] If $w s < w$, then $C_w T_s = - C_w$.
\item[$\bullet$] If $w s > w$, then 
$$C_w T_s= q C_w + q^{1/2} C_{w s} + q^{1/2} \sum_{ \substack{ z \in W \\ z\prec w \\ zs < z } } \mu(z, w) \cdot C_z.$$
\end{enumerate}
\end{prop}

For $w\in W$, define the left and right descent set of $w$ to be
$$\mfr{L}(w):=\set{s\in S: sw < w}, \quad  \mfr{R}(w):= \set{s\in S: w s < w}.$$
For any $w, x\in W$, we write
$$x\le_\mfr{R} w,$$
 if there is a sequence $x_0, x_1, ..., x_r$ in $W$ with $x_0=x$ and $x_r= w$ such that for every $0\le i < r$ the following hold:
\begin{enumerate}
\item[$\bullet$]  either $x_i \prec x_{i+1}$ or $x_{i+1} \prec x_i$; and
\item[$\bullet$]  $\mfr{R}(x_i) \nsubseteq \mfr{R}(x_{i+1})$.
\end{enumerate}
We write $x \sim_\mfr{R} w$ if both $x \le_\mfr{R} w$ and $w \le_\mfr{R} x$ hold. The resulting equivalence classes from $\sim_\mfr{R}$ are called right cells of $W$. We have a right cell decomposition
$$W=\bigsqcup_{j\in J} \mfr{C}_j.$$
We also write $x\sim_{\mfr{L}} w$ if $x^{-1} \sim_\mfr{R} w^{-1}$; this equivalence relation gives the notion of left cells in $W$. Combining the two equivalences, one denotes
$$x\sim_\mfr{LR} w$$
if either $x \sim_\mfr{L} w$ or $x \sim_\mfr{R} w$. The resulting equivalence classes  are called two-sided cells.

\begin{lm}[{\cite[Proposition 2.4]{KL1}}] \label{L:desc}
If $x\le_\mfr{R} y$, then $\mfr{L}(x) \supseteq \mfr{L}(y)$. Therefore, if $x \sim_\mfr{R} y$, then $\mfr{L}(x) = \mfr{L}(y)$.
\end{lm}

Let $\mfr{C} \subset W$ be a right cell. Let 
$\mca{I}_\mfr{C}$ be the $A$-span of all $C_w, w\in \mfr{C}$ together with all $C_x$ where $x\le_\mfr{R} w$ for some $w\in \mfr{C}$. Let $\mca{I}'_\mfr{C} \subset \mca{I}_\mfr{C}$ be the span of those $C_x$ for which $x\le_\mfr{R} w$ with $w\in \mfr{C}$ but $x\notin \mfr{C}$. Define
$$V_\mfr{C}:= \mca{I}_\mfr{C}/ \mca{I}'_\mfr{C}.$$
Then $V_\mfr{C}$ affords a (not necessarily irreducible) representation of $\msc{H}_W$, which we denote by $\sigma_\mfr{C}$. 
When specialised to the case $q=1$, as we will assume from now on, it gives a representation 
$$(\sigma_\mfr{C}, V_\mfr{C})$$
of the Weyl-group $W$, which is called the Kazhdan-Lusztig representation associated to the right cell $\mfr{C}$.
Clearly, 
$$\dim \sigma_\mfr{C}= \val{\mfr{C}}.$$
Let $\chi_{\sigma_\mfr{C}}$ be its character. We have a decomposition of the right regular representation $\C[W]$ of $W$ into the Kazhdan-Lusztig representations:
$$\C[W] = \bigoplus_{j\in J} \sigma_{\mfr{C}_j}.$$

\begin{lm} \label{L:conv}
Let $w_\alpha, \alpha\in \Delta$ be a simple reflection and $\mfr{C}$ a right cell. Then
$$\chi_{\sigma_{\mfr{C}}}(w_\alpha)= 2\val{\mfr{C}_{w_\alpha}}  - \val{\mfr{C}},$$
where $\mfr{C}_{w_\alpha}:=\set{ w \in \mfr{C}: w w_\alpha > w  }$.
\end{lm}
\begin{proof}
This follows immediately from Proposition \ref{P:action}.
\end{proof}

\begin{eg} \label{E:pm}
There are always two special right cells $\mfr{C}^+=\set{\text{id}}$ and $\mfr{C}^-=\set{ w_G}$ (see \cite{Lus3, Lus4}). In fact, $\mfr{C}^+$ and $\mfr{C}^-$ are both left cells and two-sided cells. For $\mfr{C}=\set{\text{id}}$, we have $\sigma_{\mfr{C}}=\mathbbm{1}$ by Lemma \ref{L:conv}. On the other hand, for $\mfr{C}= \set{w_G }$, one has $\sigma_\mfr{C}=\varepsilon_W$, the sign representation of $W$ where $\varepsilon_W(w)= (-1)^{l(w)}$.
\end{eg}

%
\begin{eg} \label{SL3}
For $\mbf{G}=\SL_3$, let $\alpha_1$ and $\alpha_2$ be two simple roots. Write $w_i:=w_{\alpha_i}$ for $i=1, 2$. Let $\mathbbm{1}, \varepsilon_W, \sigma_0$ be all the irreducible representations of $W\simeq S_3$, where $\sigma_0$ is of dimension two. The character table of $W$ is given in Table 1.

\begin{table}[!htbp]  \label{T1}
\caption{Character table for $\text{Irr}(W)$ for $\SL_3$}
\vskip 5pt
\begin{tabular}{|c|c|c|c|c|c|c|}
\hline
 & id  &  $w_1$ & $w_2$  & $w_1 w_2$ &  $ w_2 w_1 $  & $w_G$ \\
\hline
$\mathbbm{1}$ & 1 & 1  & 1 & 1  & 1 &  1 \\ 
\hline 
$ \varepsilon_W $ & $1$  & $-1$  & $-1$ & 1 &  1 & $-1$ \\ 
\hline
$ \chi_{\sigma_0} $ & 2  & 0  & 0 & $-1$ &  $-1$ & 0 \\ 
\hline
\end{tabular}
\end{table}

\noindent It can be inferred from Lemma \ref{L:desc} and Example \ref{E:pm} that there are four right cells in $W$:
$$\mfr{C}^+,\ \mfr{C}^-, \ \mfr{C}_1:=\set{w_1, w_1 w_2 } ,  \  \mfr{C}_2:=\set{w_2, w_2 w_1 } .  $$
Moreover, $\sigma_{\mfr{C}_1} \simeq \sigma_{\mfr{C}_2} \simeq \sigma_0$. 
\end{eg}

\begin{eg} \label{Sp4}
Let $\mbf{G}=\Sp_4$. Let $\alpha_1$ be the short simple root and $\alpha_2$ the long simple root. Again, we write $w_i:=w_{\alpha_i}$. The Weyl group $W$ is the Dihedral group of order 8. Let 
$$\mathbbm{1}, \varepsilon_W, \chi', \chi'', \sigma_0$$
 be all the irreducible representations of $W$, where the first four are one-dimensional characters and $\sigma_0$ is of dimension two. The character table is given in Table 2.
 There are four right cells of $W$:
$$\mfr{C}^+,\ \mfr{C}^-, \  \mfr{C}_1:=\set{w_1, w_1 w_2, w_1 w_2 w_1 },  \  \mfr{C}_2:=\set{w_2, w_2 w_1, w_2 w_1 w_2 }.  $$
Moreover, 
$$\sigma_{\mfr{C}_1}= \sigma_0 \oplus \chi', \quad  \sigma_{\mfr{C}_2}= \sigma_0 \oplus \chi''.$$
 
\begin{table}[!htbp]  \label{T2}
\caption{Character table for $\text{Irr}(W)$ of $\Sp_4$}
\vskip 5pt
\begin{tabular}{|c|c|c|c|c|c|c|c|c|}
\hline
 & id  &  $w_1$ & $w_2$  & $w_1 w_2$ &  $ w_2 w_1 $  & $w_1 w_2 w_1$  &  $w_2 w_1 w_2$ & $w_G$ \\
\hline
$ \mathbbm{1}$ & 1 & 1  & 1 & 1  & 1 &  1 & 1 & 1 \\ 
\hline 
$ \varepsilon_W $ & $1$  & $-1$  & $-1$ & 1 &  1 & $-1$ & $-1$ & 1 \\ 
\hline
$ \chi' $ & 1  & $-1$  & 1 & $-1$ &  $-1$ & 1 & $-1$ & 1 \\ 
\hline
$ \chi'' $ & 1  & 1  &  $-1$ & $-1$ &  $-1$ & $-1$ & 1 & 1 \\ 
\hline
$ \chi_{\sigma_0} $ & 2  & 0  & 0 & 0 &  0 & 0 & 0 & $-2$ \\ 
\hline
\end{tabular}
\end{table}
\end{eg}

\subsection{A conjectural formula} Again, let $\chi$ be a regular unramified genuine character of $Z(\wt{T})$. Let 
$$\Gamma= \bigcup_{w\in W_\Gamma} w(\mca{C}^+)$$
be a connected component of $V - \bigcup_{\alpha \in \Phi(\chi)} \Ker (\alpha^\vee)$.

\begin{lm}
If $\Phi(\chi) \subset \Delta$, then $W_\Gamma= \bigsqcup_i \mfr{C}_i$; that is, $W_\Gamma$ is a disjoint union of right cells.
\end{lm}
\begin{proof} It suffices to show that if $\mfr{C}$ is a right cell, then $\bigsqcup_{w\in \mfr{C}} w(\mca{C}^+)$ does not cross the walls 
$$\set{\text{Ker}(\alpha^\vee): \alpha\in \Phi(\chi) \subset \Delta}.$$
Assume that the contrary holds, then there exists $\alpha^\vee \in \Delta^\vee$ and $w\in W$ such that 
$l(w_\alpha w)=l(w) + 1$, and moreover $\set{w, w_\alpha w} \subset \mfr{C}$. However, Lemma \ref{L:desc} shows that $w_\alpha \in  \mfr{L}(w_\alpha w) = \mfr{L}(w)$, which gives a contradiction on the lengths of $w$ and $w_\alpha w$. This completes the proof.
\end{proof}

For any connected component $\Gamma$ with $W_\Gamma=\bigsqcup_i \mfr{C}_i$, we denote
$$\sigma_\Gamma:= \bigoplus_i \sigma_{\mfr{C}_i}$$
and call $\sigma_\Gamma$ the Kazhdan-Lusztig representation associated to $\Gamma$. Let $\chi_{\sigma_\Gamma}$ be the character of $\sigma_\Gamma$.

Recall the representation  (see \S \ref{SS:twW})
$$\sigma_\msc{X}: W \to \text{Perm}( \msc{X}_{Q,n} ),$$
which arises from the twisted Weyl action $w[y]=w(y-\rho) + \rho$.  Let $\chi_{\sigma_\msc{X}}$ be the character of $\sigma_\msc{X}$. 
To proceed, we first discuss about the restriction of $\sigma_\msc{X}$ to parabolic subgroups of $W$ in the general setting.

 Let $w \in W$ such that $w(\mca{C}^+) \subset \Gamma$. Let $w_\Gamma \in W_{\Gamma_\text{op}}$ be such that $\pi_\Gamma={\rm Im}(T(w_\Gamma, {}^{w_\Gamma^{-1} w^{-1}  } \chi  ))$, where 
 $$T(w_\Gamma, {}^{ w_\Gamma^{-1} w^{-1}  } \chi  ) \in \Hom( I({}^{w_\Gamma^{-1} w^{-1}  } \chi), I({}^{w^{-1}} \chi  )  )$$
  is a basis. Here $w_\Gamma$ may not be unique. Denote by
$$\sigma_\msc{X}^{w_\Gamma}: W(\Delta_{w_\Gamma}) \into W \to \text{Perm}( \msc{X}_{Q,n} )$$
the restriction of $\sigma_\msc{X}$ to $W(\Delta_{w_\Gamma})$.
From the decomposition 
$$\msc{X}_{Q,n} = \bigsqcup_{i \in I_{w_\Gamma}}  \mca{O}^{w_\Gamma}_{y_i}$$ 
into $W(\Delta_{w_\Gamma}) $-orbits, we obtain the decomposition
$$\sigma_\msc{X}^{w_\Gamma}=\bigoplus_{i\in I_{w_\Gamma}} \sigma_\msc{X}^{w_\Gamma, y_i},$$
where 
$$\sigma_\msc{X}^{w_\Gamma, y_i}: W(\Delta_{w_\Gamma}) \to \text{Perm} ( \mca{O}_{y_i}^{w_\Gamma} ) $$
is the permutation representation of $W(\Delta_{w_\Gamma})$ on $\mca{O}_{y_i}^{w_\Gamma} $.
Again, $\mca{O}_{y}^{w_G}=\mca{O}_{y}$ for any $y\in \msc{X}_{Q,n}$, and in this case  we will simply write 
$$\sigma_{\msc{X}}^{y}:=\sigma_\msc{X}^{w_G, y},$$
which is a $\val{\mca{O}_y}$-dimensional representation of $W$. One has
\begin{equation} \label{E:dec-X}
\sigma_{ \msc{X} } = \bigoplus_{\mca{O}_y \in \mca{O}_{ \msc{X}_{Q,n} }} \sigma_\msc{X}^y.
\end{equation}

\begin{conj} \label{C:S}
Let $\wt{G}$ be a Brylinski-Deligne covering group. Let $\chi$ be a regular unramified genuine character of $Z(\wt{T})$.  Assume $\Phi(\chi) \subseteq \Delta$. Let $\pi_\Gamma$ be an irreducible constituent of $I(\chi)$. Then for every persistent orbit $\mca{O}_y \subseteq \msc{X}_{Q,n}$, one has 
\begin{equation} \label{E:C-S}
\dim \Wh(\pi_\Gamma)_{\mca{O}_y}  = \angb{ \sigma_\msc{X}^y }{ \sigma_\Gamma}_W.
\end{equation}
In particular, if $\wt{G}$ is persistent, then \eqref{E:C-S} holds for every orbit $\mca{O}_y \subseteq \msc{X}_{Q,n}$ and therefore
\begin{equation} \label{E:C-S1}
\dim \Wh(\pi_\Gamma)  = \angb{ \sigma_\msc{X}}{ \sigma_\Gamma}_W.
\end{equation}
\end{conj}

Note that equality \eqref{E:C-S1} is a consequence of  \eqref{E:C-S} by \eqref{Wh-sum}.


\begin{rmk} \label{R:1}
First, \eqref{E:C-S1}
 is compatible with the fact that 
$$\val{\msc{X}_{Q,n}}=\dim \Wh(I(\chi)).$$
Indeed, since $\dim \Wh(-)$ is additive on exact sequences, we have
$$\dim \Wh(I(\chi))= \sum_{\text{all } \Gamma} \dim \Wh(\pi_\Gamma).$$ 
On the other hand, one has
$$ \sum_{\text{all } \Gamma} \angb{\sigma_{\msc{X}}}{ \sigma_\Gamma } = \sum_{\text{all } \mfr{C}} \angb{\sigma_{\msc{X}}}{ \sigma_\mfr{C} } =\angb{ \sigma_\msc{X} }{ \C[W] } = \val{\msc{X}_{Q,n}},$$ 
which gives the claimed compatibility. 

Second, the conjectured equality \eqref{E:C-S} might be refined as the left hand side also equals the rank of $T(w_\Gamma, {}^{w_\Gamma^{-1} w^{-1}} \chi)^*_{\mca{O}_y^{w_\Gamma}}$ (represented by a square matrix of size $\val{ \mca{O}_y^{w_\Gamma} }$), where $w, w_\Gamma$ are chosen as above. However, in this case, we are not aware of a natural replacement for the right hand side in \eqref{E:C-S}.
\end{rmk}

\begin{rmk}
If $\wt{G}=G$, then it is shown in \cite[\S II.1]{MW1} that to each $\pi_\Gamma$ there are naturally associated nilpotent orbits of $G$ which constitute the wave-front set of $\pi_\Gamma$, i.e., the maximal orbits in the Harish-Chandra character expansion of $\pi_\Gamma$. These orbits are contained in a unique nilpotent orbit $\mca{O}_\Gamma \subset \mbf{G}(\wt{F})$, which is the Richardson orbit (see \cite[\S 5.2]{Car}) associated to the parabolic subgroup $P_{\Phi(\chi) - S}$, see \cite[Proposition II.1.3]{MW1}. However, for covering groups, this identification no longer holds and it is not clear what role $\mca{O}_\Gamma$ plays in describing the character expansion of $\pi_\Gamma$.
\end{rmk}

Consider the involution on $W$ given by
$$W \to W, \quad w \mapsto w_G w w_G.$$

\begin{prop} \label{P:sym}
Let $\Gamma_1$ and $\Gamma_2$ be two connected components of $V$ such that $W_{\Gamma_2}= w_G \cdot W_{\Gamma_1} \cdot w_G$. Assume that Conjecture \ref{C:S} holds. Then for every persistent orbit $\mca{O}_y\subseteq \msc{X}_{Q,n}$, one has
$$\dim \Wh(\pi_{\Gamma_1})_{\mca{O}_y} = \dim \Wh(\pi_{\Gamma_2})_{\mca{O}_y};$$
in particular, $\dim \Wh(\pi_{\Gamma_1}) = \dim \Wh(\pi_{\Gamma_2})$ if $\wt{G}$ is persistent.
\end{prop}
\begin{proof}
Let 
$$W_{\Gamma_1}=\bigsqcup_i \mfr{C}_i,$$
where $\mfr{C}_i$ is a right cell. By assumption, we get
$$W_{\Gamma_1}=\bigsqcup_i w_G \mfr{C}_i w_G,$$
where $\mfr{C}_i':=w_G \mfr{C}_i w_G$ is also a right cell and we have $\sigma_{\mfr{C}'_i} \simeq \sigma_{\mfr{C}_i}$ (see \cite[Proposition 6.3.5]{BB}). It follows that 
$$\sigma_{\Gamma_1} \simeq \sigma_{\Gamma_2}.$$
Thus  Conjecture \ref{C:S} implies that
$$\dim \Wh(\pi_{\Gamma_1})_{\mca{O}_y} = \dim \Wh(\pi_{\Gamma_2})_{\mca{O}_y}$$
for every persistent $\mca{O}_y \subset \msc{X}_{Q,n}$. This concludes the proof.
\end{proof}

The involution $w \mapsto w_G w w_G$ is the identity map whenever $w_G = -1 \in W$, which holds if and only if all the exponents of $W$ are odd (see \cite[Page 127, Corollary 3]{Bou}). However, this involution is in general not the identity map, for example when $W$ is of type $A_r$ with even $r$. Thus Proposition \ref{P:sym} applies to covers $\wt{\SL}_{r+1}^{(n)}$ with $r$ even in a nontrivial way. We will illustrate in section \S \ref{S:3eg} the case $\wt{\SL}_3^{(n)}$ as an example.
\vskip 5pt


\subsection{The special case $\mca{O}_y=\set{y}$}
As an (important) example, we show:
\begin{thm} \label{T:dis}
Let $\mca{O}_y = \set{y}$ be a singleton orbit in $\msc{X}_{Q,n}$. Assume that $\mca{O}_y$ is persistent. Then Conjecture \ref{C:S} holds, i.e., $\dim \Wh(\pi_\Gamma)_{\mca{O}_y} = \angb{\sigma_\msc{X}^y}{\sigma_\Gamma}$ for every $\pi_\Gamma \in \JH(I(\chi))$.
\end{thm}
\begin{proof}
We retain the assumption on $\chi$ and notations in Conjecture \ref{C:S}. Let $y\in (\msc{X}_{Q,n})^{W}$, i.e., $\mca{O}_y=\set{y}$. Then $\sigma_\msc{X}^y=\mbm{1}_W$. 
Note that we have
\begin{equation} \label{E:rhs}
\angb{\sigma_\msc{X}^y}{\sigma_\Gamma}=
\begin{cases}
1 & \text{ if $\mca{C}^+ \subset W_\Gamma$, i.e., $\Gamma=\Gamma^+$};\\
0 & \text{ otherwise}.
\end{cases}
\end{equation}
On the other hand, 
$\dim \Wh(\pi_\Gamma)_{\mca{O}_y}$ equals to the rank of $T(w_G, {}^{w_G^{-1} w^{-1}} \chi)_{\mca{O}_y}^*, w\in W_\Gamma$, which is scalar-valued.

Denote temporarily $\chi':= {}^{w_G^{-1} w^{-1}} \chi$. Let $w_G=w_k w_{k-1} ... w_1$ be a minimal decomposition with $w_i=w_{\alpha_i}$ for some $\alpha_i\in \Delta$. Since $\mca{O}_y$ is assumed to be persistent, it follows from \cite[Theorem 5.13, Remark 5.14]{GSS2} that
$$T(w_G, \chi')_{\mca{O}_y}^* = J(\chi', y, w_G) \cdot \prod_{i=1}^k \gamma(w_i, {}^{w_{i-1} ... w_1} \chi')^{-1},$$
where for any genuine character $\chi$
$$\gamma(w_\alpha, \chi)^{-1}=\frac{1-q^{-1} \cdot \chi(\wt{h}_{\alpha}(\varpi^{n_\alpha}))^{-1}}{1- \chi(\wt{h}_{\alpha}(\varpi^{n_\alpha})) }$$
and $J(\chi, y, w_G) \in \C^\times$ is a nonzero number. It now follows from Lemma \ref{L:reg} that
$$\begin{aligned}
T(w_G, \chi')_{\mca{O}_y}^*  & =_{\C^\times} \prod_{\alpha\in \Phi^+}  (1-q^{-1} \cdot \chi' (\wt{h}_\alpha(\varpi^{n_\alpha}))^{-1}) \\
& =_{\C^\times} \prod_{\alpha\in \Phi^+}  (1-q^{-1} \cdot ({}^{w^{-1}}\chi) (\wt{h}_\alpha(\varpi^{n_\alpha}))) \\
& =_{\C^\times}  \prod_{\alpha\in \Phi^+}  (1-q^{-1} \cdot \chi (w\cdot \wt{h}_\alpha(\varpi^{n_\alpha}) \cdot w^{-1})),
\end{aligned}$$
where $=_{\C^\times}$ means equality up to a nonzero complex number. Thus, we have the following equivalence:
\begin{equation} \label{E:2side}
 \begin{aligned}
& T(w_G, \chi')_{\mca{O}_y}^* = 0 \\
\Longleftrightarrow & \text{ there exists $\alpha\in \Phi_+$ such that } \chi (w\cdot \wt{h}_\alpha(\varpi^{n_\alpha}) \cdot w^{-1})=q \\
\Longleftrightarrow & \text{ there exists $\alpha\in \Phi_+$ such that } -w(\alpha^\vee) \in \Phi(\chi)^\vee \\
\Longleftrightarrow &\   \Phi(\chi)^\vee \cap w(\Phi_-^\vee) \ne \emptyset \\
\Longleftrightarrow &\  w \notin \W_{\Gamma^+} \text{ by Lemma \ref{L:1} }.
\end{aligned}
\end{equation}
That is, $\dim \Wh(\pi_\Gamma)_{\mca{O}_y} =0$ if and only if $\Gamma \ne \Gamma^+$. In view of \eqref{E:rhs}, this completes the proof.
\end{proof}

Recall that (cf. Example \ref{ss-sat}) in the case $n=1$, a linear algebraic group $G=\wt{G}$ is always saturated and therefore persistent (see Lemma \ref{Satu}). Theorem \ref{T:dis} thus recovers \cite[Proposition 4]{Rod4} for $G$, which asserts that $\pi_\Gamma$ is generic if and only if $\Gamma =\Gamma^+$.

\section{Induction, $\dim \Wh(\pi_{\Gamma^-})_{\mca{O}_y}$ and $\dim \Wh (\pi_\Gamma)$} \label{S:ThSt}

\subsection{The case $\pi_{\mca{C}^-}$} \label{SS:Theta}
In this subsection, we assume $\Phi(\chi) =\Delta$, i.e., 
$$\chi(\wt{h}_\alpha(\varpi^{n_\alpha}))= q^{-1} \text{ for all } \alpha\in \Delta.$$
We consider the two representations $\pi_\Gamma$ when $\Gamma = \mca{C}^+$  or $\Gamma=\mca{C}^-$ (and thus $W_\Gamma=\mfr{C}^+$ or $\mfr{C}^-$ respectively). The representation $\pi_{\mca{C}^+}$ is the unique irreducible subrepresentation of $I(\chi)$, and is the covering analogue of the Steinberg representation. Meanwhile, $\pi_{\mca{C}^-}$ is the unique irreducible Langlands quotient of $I(\chi)$, the so-called theta representation $\Theta(\wt{G}, \chi)$ (in the notation of \cite{Ga2}) associated to the exceptional character $\chi$, see also \cite{KP, Cai1, Les}.

Recall from Example \ref{E:pm}:
$$\sigma_{\mca{C}^+}= \mbm{1}_W, \quad \sigma_{\mca{C}^-}=\varepsilon_W.$$

\begin{lm} \label{L:obs1}
For every persistent orbit $\mca{O}_y\subseteq \msc{X}_{Q,n}$, one has
\begin{equation} \label{ro}
 \angb{ \sigma_\msc{X}^y  }{ \mbm{1}_W } =1,
  \text{ and }
 \angb{ \sigma_\msc{X}^y  }{ \varepsilon_W } = 
 \begin{cases}
 1 & \text{ if }  \mca{O}_y \text{ is free },\\
 0 & \text{ otherwise.}
 \end{cases}
\end{equation}
In particular, if $\wt{G}$ is persistent, then
$$\angb{ \sigma_\msc{X}  }{ \mbm{1}_W }=\val{ \mca{O}_{\msc{X}_{Q,n}} } \text{ and } \angb{ \sigma_\msc{X}  }{ \varepsilon_W }=\val{ \OF_{\msc{X}_{Q,n}} },$$
where $\OF_{ \msc{X}_{Q,n} }$ is the set of free $W$-orbits in $\msc{X}_{Q,n}$. 
\end{lm}
\begin{proof}
It suffices to prove \eqref{ro}. The equality $\angb{ \sigma_\msc{X}^y  }{ \mbm{1}_W } =1$ is clear and in fact holds for every orbit $\mca{O}_y \subset \msc{X}_{Q,n}$, not necessarily persistent. We prove the second equality in \eqref{ro}.

For every $y\in \msc{X}_{Q,n}$, since $W$ acts transitively on $\mca{O}_y$, we have $\sigma_\msc{X}^y \simeq \text{Ind}_{W_y}^W \mbm{1}_{W_y}$, where $W_y:={\rm Stab}_W(y; \msc{X}_{Q,n})$ is the stabilizer of $y$. 
It now follows from Frobenius reciprocity 
 $$\angb{ \sigma_\msc{X}^y  }{ \mbm{1}_W }_W= \angb{ \mbm{1}_{W_y}  }{ \varepsilon_{W}|_{W_y} }_{W_y}.$$
Since $W_y = {\rm Stab}_W(y; \msc{X}_{Q,n}^{sc})$ by the assumption that $\mca{O}_y$ is persistent,  if $W_y\ne \set{1}$ (i.e. $\mca{O}_y$ is not free), then by \cite[Lemma 3.12]{Ga2}, we may assume $y$ is  such that $w_\alpha \in W_y$ for some $\alpha\in \Delta$. Thus, 
 $$\angb{ \sigma_\msc{X}^y  }{ \mbm{1}_W }_W = 0.$$
 On the other hand, if $\mca{O}_y$ is free, then $W_y=\set{1}$ and clearly $\angb{ \sigma_\msc{X}^y  }{ \mbm{1}_W }_W=1$. This completes the proof of \eqref{ro} and thus the lemma.
\end{proof}

\begin{prop} \label{T:C-}
Let $\chi$ be a regular unramified character of $Z(\wt{T})$ with $\Phi(\chi) =\Delta$. Then for every persistent orbit $\mca{O}_y$, one has
\begin{equation} \label{E:T-thetaO}
\dim \Wh (\pi_{\mca{C}^-})_{\mca{O}_y} = \angb{ \sigma_\msc{X}^y }{ \varepsilon_W }.
\end{equation}
In particular, if $\wt{G}$ is persistent, then
\begin{equation} \label{E:T-theta}
\dim \Wh (\pi_{\mca{C}^-}) = \angb{ \sigma_\msc{X}}{ \varepsilon_W },
\end{equation}
which is equal to the number of free $W$-orbits in $\msc{X}_{Q,n}$.
\end{prop}
\begin{proof}
We first note that the equality \eqref{E:T-theta} is just a reinterpretation of \cite[Theorem 3.15]{Ga2}. Indeed, it is shown in loc. cit. that 
$$\dim \Wh (\pi_{\mca{C}^-})= \val{ \OF_{\msc{X}_{Q,n}}  }.$$
Thus \eqref{E:T-theta} follows from Lemma \ref{L:obs1}.


Regarding \eqref{E:T-thetaO} , we note that if $\pi_{\mca{C}^-}=T(w_\Gamma, {}^{w_\Gamma^{-1} w^{-1}} \chi)$, then necessarily $w=w_G$ and $w_\Gamma=w_G$. Thus, it suffices to show
$$\text{rank of } T(w_G, \chi)^*_{\mca{O}_y}= \angb{\sigma_\msc{X}^y }{\varepsilon_W}$$
for every persistent $\mca{O}_y \subset \msc{X}_{Q,n}$. The proof of this is also implicit in \cite{Ga2}, and is basically the ``$\mca{O}_y$-relative" version of the argument there. We give a sketch as follows. 

The discussion in \cite[page 345-354]{Ga2} (including from Proposition 3.4 to Theorem 3.14) could be carried in the ``$\mca{O}_y$-relative" setting arising from considering  any orbit $\mca{O}_y \subset \msc{X}_{Q,n}$. For instance, since $\pi_{\mca{C}^-}=\text{Im}(T(w_G, \chi))$, for every orbit $\mca{O}_y$, we have
$$\Wh(\pi_{\mca{C}^-})_{\mca{O}_y}=\text{Im}( T(w_G, \chi)^*_{\mca{O}_y} ).$$
It follows from Proposition \ref{P:decom}, for example, that \cite[Proposition 3.4]{Ga2} has a ``$\mca{O}_y$-relative" version:
\begin{equation} \label{O-Cap}
\begin{aligned}
& \Wh (\pi_{\mca{C}^-})_{\mca{O}_y} \\
= & \set{\lambda \in \Wh(I(\chi))_{\mca{O}_y}:  \lambda|_{\Ker (T(w_G, \chi) )  } = 0} \\
= &  \set{\lambda \in \Wh(I(\chi))_{\mca{O}_y}:  T(w_\alpha, {}^{w_\alpha}\chi)^*(\lambda) = 0 \text{ for all } \alpha\in \Delta}  \\
= &  \bigcap_{\alpha\in \Delta} \text{Ker} \left( T(w_\alpha, {}^{w_\alpha}\chi)^*_{\mca{O}_y}:  \Wh(I(\chi))_{\mca{O}_y}  \to \Wh( I({}^{w_\alpha} \chi)_{\mca{O}_y}   )         \right).
\end{aligned}
\end{equation}
Following this, argument in \cite{Ga2} can be easily adapted and we have:
\begin{enumerate}
\item[$\bullet$] an analogue of \cite[(10)]{Ga2} shows that $\dim \Wh(\pi_{\mca{C}^-})_{\mca{O}_y} \le 1$;
\item[$\bullet$] if $\mca{O}_y$ is free, then $\dim \Wh(\pi_{\mca{C}^-})_{\mca{O}_y} = 1$;
\item[$\bullet$] if $\mca{O}_y$ is persistent but not free, then the proof of \cite[Proposition 3.13]{Ga2} shows that $\dim \Wh(\pi_{\mca{C}^-})_{\mca{O}_y} =0$.
\end{enumerate}
That is, we have
\begin{equation*}
\dim \Wh (\pi_{\mca{C}^-})_{\mca{O}_y} =
\begin{cases}
1 & \text{ if $\mca{O}_y$ is free};\\
0 & \text{ otherwise}.
\end{cases}
\end{equation*}
However, this is equivalent to the equality $\dim \Wh (\pi_{\mca{C}^-})_{\mca{O}_y} = \angb{\sigma_\msc{X}^y}{ \varepsilon_W}$ in view of Lemma \ref{L:obs1}. This completes the proof for \eqref{E:T-thetaO}.
\end{proof}

\subsection{Induction and $\dim \Wh (\pi_{\Gamma})_{\mca{O}_y}$} \label{S:Ind}
In this subsection, we will show by a reduction that \eqref{E:C-S} in Conjecture \ref{C:S} is a consequence of Conjecture \ref{C:O-ex}. The same method proves \eqref{E:C-S1} for every persistent covering group unconditionally. The key is that to prove \eqref{E:C-S} (modulo Conjecture \ref{C:O-ex} in this case) or \eqref{E:C-S1}, it suffices to prove the equality for every standard Levi subgroup; this latter result has been discussed in the preceding subsection. This reduction arises from a method of induction on the two sides of \eqref{E:C-S}: what pertains to the left hand side of \eqref{E:C-S} is a (simpler) instance of Rodier's heridity theorem on Whittaker functionals, while the right hand side concerns the induction of Kazhdan-Lusztig representations on right cells.
\vskip 10pt

Let $S \subset \Delta$ and $W(S)$ be the parabolic  subgroup of $W$. Let $w_S \in W(S)$ be the longest element. 
Let $R_S \subset W$ be the set of representatives of minimal length for the right cosets $W(S)\backslash W$, i.e., 
$$R_S=\set{w\in W: \ l(w_\alpha \cdot w) > l(w) \text{ for all } \alpha \in S}.$$

\begin{lm} \label{L:key1}
Let $S$ be such that $S\subseteq \Phi(\chi) \subseteq \Delta$. One has $W_{\Gamma_S^\natural} = w_S \cdot R_S$.
\end{lm}
\begin{proof}
Let $\alpha\in \Delta$. It follows from \cite[page 170, Corollary 2]{Bou} that $l(w_\alpha w)> l(w)$ if and only if 
$w^{-1}(\alpha^\vee) > 0$. Thus, 
$$R_S =\set{w:  w^{-1}(\alpha^\vee)>0 \text{ for every } \alpha \in S}.$$
On the other hand, $W_{\Gamma_S^\natural} =\set{w: \Phi(\chi)^\vee \cap w(\Phi_-^\vee) \supseteq S^\vee}$. One can now check easily that the claim holds.
\end{proof}

Before stating the main theorem, we recall the following result of Barbasch and Vogan.

\begin{prop}[{\cite[Proposition 3.15]{BV2}, see also \cite[Theorem 2]{Roi}}]  \label{P:ind}
Let $S \subseteq \Delta$ and $W(S) \subset W$ be the associated parabolic subgroup. Let $\mfr{C}_S \subset W(S)$ be a right cell of $W(S)$, and $\sigma_{\mfr{C}_S}$ the associated Kazhdan-Lusztig representation of $W(S)$. Then $\mfr{C}_{S} \cdot R_S \subset W$ is a union of right cells in $W$. Let $\sigma_{\mfr{C}_{S} \cdot R_S}$ be its associated Kazhdan-Lusztig representation of $W$. Then
$$\sigma_{\mfr{C}_S \cdot R_S} = \Ind_{W(S)}^W \sigma_{\mfr{C}_S}.$$
\end{prop}
There is a counterpart regarding the restriction of Kazhdan-Lusztig cells and representations, for which the reader may refer to \cite[Proposition 3.11]{BV2} or \cite[Proposition 1]{Roi}.

For every $S \subseteq \Phi(\chi) \subseteq \Delta $, we denote
$$\sigma_{\Gamma_S^\natural}:= \bigoplus_{S': S \subseteq S' \subseteq \Phi(\chi)}  \sigma_{\Gamma_{S'}},$$
which is called the Kazhdan-Lusztig representation associated with $W_{\Gamma_S^\natural}$. 
\begin{cor} \label{C:IE-KL}
For every $S \subseteq \Phi(\chi) \subseteq \Delta $, we have
$$ \sigma_{\Gamma_S} = \sum_{S': S\subseteq S'  \subseteq \Phi(\chi)}  (-1)^{\val{S'-S}}  \cdot  \sigma_{\Gamma_{S'}^\natural} \text{ with }  \sigma_{\Gamma_{S'}^\natural}={\rm Ind}_{W(S')}^W \varepsilon_{W(S')}.$$
\end{cor}
\begin{proof}
The first equality follows from Lemma \ref{L:key}. The second equality follows from Lemma \ref{L:key1}, Proposition \ref{P:ind}, and the fact that $\set{w_S}$ corresponds to the Kazhdan-Lusztig representation $\varepsilon_{W(S)}$ of $W(S)$ for every $S\subseteq \Phi(\chi)$.
\end{proof}

\begin{thm} \label{T:G-pm}
Assume $\Phi(\chi) \subseteq \Delta$. Then for every $S\subseteq \Phi(\chi)$ and every persistent orbit $\mca{O}_y$, one has
$$\dim \Wh(\pi_{\Gamma_S^\natural})_{\mca{O}_y} =  \angb{ \sigma_\msc{X}^y  }{ \sigma_{\Gamma_S^\natural} }_W.$$
As a consequence, we have:
\begin{enumerate}
\item[(i)]  Conjecture \ref{C:S} holds for $\pi_{\Gamma^-}$;
\item[(ii)] if Conjecture \ref{C:O-ex} holds, then Conjecture \ref{C:S} holds for every $\pi_{\Gamma}$, i.e., for every persistent orbit $\mca{O}_y \subseteq \msc{X}_{Q,n}$,
 $$\dim \Wh(\pi_\Gamma)_{\mca{O}_y} = \angb{ \sigma_\msc{X}^y  }{ \sigma_{\Gamma} }_W;$$
\item[(iii)] if $\wt{G}$ is persistent, then for every $\pi_\Gamma \in \JH(I(\chi))$,
$$\dim \Wh(\pi_\Gamma) = \angb{\sigma_\msc{X}}{\sigma_\Gamma}_W.$$
\end{enumerate}
\end{thm}
\begin{proof}
We consider induction on the two sides of 
$$\dim \Wh(\pi_{\Gamma_S^\natural})_{\mca{O}_y} =  \angb{ \sigma_\msc{X}^y  }{ \sigma_{\Gamma_S^\natural} }_W.$$
Let $\mbf{P} = \mbf{M} \mbf{U} \subset \mbf{G}$ be the parabolic subgroup associated to $S \subseteq \Phi(\chi)$. Recall from \S \ref{SS:coarseF} the following commutative diagram
$$\begin{tikzcd}
I(\chi)  \ar[rr, "{T(w_S,\chi)}"] \ar[d, equal] & & I({}^{w_S}\chi) \ar[d, equal] \\
I_{\wt{P}}^{\wt{G}} \big( I_{\wt{B}_M}^{\wt{M}} (\chi) \big) \ar[rr, "{ I_{\wt{P}}^{\wt{G}} (\hat{T}(w_S,\chi))   }"] & & I_{\wt{P}}^{\wt{G}} \big( I_{\wt{B}_M}^{\wt{M}} ({}^{w_S}\chi) \big).
\end{tikzcd}$$
The image of $\hat{T}(w_S,\chi)$ is just $\pi_{\Gamma^-_M}$, the representation of $\wt{M}$ associated to the component $\Gamma_M^-=w_S(\mca{C}^+_M)$; that is, $\pi_{\Gamma_M^-}$ is just the theta representation of $\wt{M}$. By definition, $\pi_{\Gamma_S^\natural} = I_{\wt{P}}^{\wt{G}}\ \pi_{\Gamma_M^-}$, and the above diagram gives that  $\pi_{\Gamma_S^\natural} = \text{Im}(T(w_S, \chi))$.

Let $\mca{O}_y \subseteq \msc{X}_{Q,n}$ be a persistent $W$-orbit. By definition, one has
$$\dim \Wh(\pi_{\Gamma_S^\natural})_{\mca{O}_y}= \text{ rank of } T(w_S, \chi)_{\mca{O}_y}^*.$$
On the other hand, decompose 
$$\mca{O}_y=\bigsqcup_{i\in I} \mca{O}_{y_i}^{w_S}$$
 into disjoint $W(S)$-orbits, we have
$$\hat{T}(w_S, \chi)_{\mca{O}_y}^*:= \bigoplus_{i\in I} \hat{T}(w_S, \chi)_{\mca{O}_{y_i}^{w_S}}^*.$$
Since $T(w_S, \chi)_{\mca{O}_y}^*$ and $\hat{T}(w_S, \chi)_{\mca{O}_y}^*$ can be represented by the same matrix, they have equal rank. Therefore,
\begin{equation} \label{E:chain}
\dim \Wh(\pi_{\Gamma_S^\natural})_{\mca{O}_y}= \sum_{i\in I} \text{rank}\big( \hat{T}(w_S, \chi)_{\mca{O}_{y_i}^{w_S}}^* \big)= \sum_{i\in I} \dim \Wh(\pi_{\Gamma_M^-})_{\mca{O}_{y_i}^{w_S}}.
\end{equation}
We have from Lemma \ref{L:key1} that $W_{\Gamma_S^\natural}= W_{\Gamma_{M}^-} \cdot R_{S}$, where $W_{\Gamma_{M}^-}=\set{w_S}$ and $\sigma_{\Gamma_{M}^-}=\varepsilon_{W(S)}$. Note that every $W(S)$-orbit $\mca{O}_{y_i}^{w_S}$ is persistent by Lemma \ref{pers-para}.  It then follows that for every persistent $\mca{O}_y \subseteq  \msc{X}_{Q,n}$:
\begin{equation*} \label{E:recip}
\begin{aligned}
& \dim \Wh(\pi_{\Gamma_S^\natural})_{\mca{O}_y} \\
= & \sum_{i\in I} \dim \Wh(\pi_{\Gamma_M^-})_{\mca{O}_{y_i}^{w_S}}  \text{ by  \eqref{E:chain}} \\
= & \sum_{i\in I} \angb{\sigma_{\msc{X}}^{w_S,y_i} }{ \sigma_{\Gamma_M^-}  }_{W(S)} \text{ by Theorem \ref{T:C-} } \\
= & \angb{\sigma_{\msc{X}}^y|_{W(S)} }{ \sigma_{\Gamma_M^-}  }_{W(S)} \\
= &  \angb{\sigma_\msc{X}^y }{ \text{Ind}_{W(S)}^W \sigma_{\Gamma_M^-}  }_W \text{ by Frobenius reciprocity }\\
= & \angb{\sigma_\msc{X}^y  }{ \sigma_{\Gamma_S^\natural} }_W  \text{ by Proposition \ref{P:ind} }. 
 \end{aligned} 
 \end{equation*}
Specialized to the case $S=\Phi(\chi)$, one has $\pi_{\Gamma_S^\natural} = \pi_{\Gamma^-}$ and this gives (i).

For (ii), we have from Corollary \ref{C:IE-KL}
$$\sigma_{\Gamma_S} =  \sum_{S': S \subseteq S' \subseteq \Phi(\chi)} (-1)^{\val{S'-S}} \cdot \sigma_{\Gamma_{S'}^\natural}  \in \msc{R}(W) .$$
Since $\dim \Wh(\pi_{\Gamma_S^\natural})_{\mca{O}_y} = \angb{\sigma_\msc{X}^y}{ \sigma_{\Gamma_S^\natural} }_W$ for every $S\subset \Phi(\chi)$, the assertion (ii) follows from the second part of Proposition \ref{P:c-dim}. Similarly, (iii) follows from the first part of Proposition \ref{P:c-dim}.
\end{proof}

\begin{cor} \label{C:le2}
Assume $\Phi(\chi) \subset \Delta$ with $\val{\Phi(\chi)} \le 2$. Then Conjecture \ref{C:S} holds, i.e., for every persistent orbit $\mca{O}_y$ and every $\pi_\Gamma \in \JH(I(\chi))$, one has
$$\dim \Wh (\pi_\Gamma)_{\mca{O}_y} = \angb{ \sigma_{\msc{X}}^y }{ \sigma_\Gamma  }.$$
\end{cor}
\begin{proof}
In the proof of Theorem \ref{T:G-pm} (ii), we used Proposition \ref{P:c-dim}. However, if $\val{\Phi(\chi)} \le 2$, then one can apply Corollary \ref{C:exac=2} instead to deduce the above equality unconditionally without assuming Conjecture \ref{C:O-ex}.
\end{proof}

The above result shows that for $\wt{G}$ of semisimple rank at most two, Conjecture \ref{C:S} holds for every regular unramified $\chi$ such that $\Phi(\chi) \subset \Delta$.

\begin{cor} \label{C:lb}
Assume $\Phi(\chi) \subseteq \Delta$. 
\begin{enumerate}
\item[(i)] If Conjecture \ref{C:O-ex} holds, then for every persistent orbit $\mca{O}_y$ one has
$$\dim \Wh(\pi_\Gamma)_{\mca{O}_y} \ge \dim \Wh(\pi_{\Gamma^-})_{\mca{O}_y}$$
for all $\pi_\Gamma \in \JH(I(\chi))$. 
\item[(ii)] If $\wt{G}$ is persistent, then we always have 
$$\dim \Wh(\pi_\Gamma) \ge \dim \Wh(\pi_{\Gamma^-})$$
for every $\pi_\Gamma \in \JH(I(\chi))$.
\end{enumerate}
\end{cor}
\begin{proof}
For every $S\subseteq \Phi(\chi)$, we write $W_S$ for $W_{\Gamma_S}=\set{w\in W:  \Phi(\chi)^\vee \cap w(\Phi_-^\vee) = S^\vee}$, and similarly we have 
$$W(\Phi(\chi))_S=\set{w\in W(\Phi(\chi)):  \Phi(\chi)^\vee \cap w(\Phi_{o,-}^\vee) = S^\vee},$$
where $\Phi_{o,-}$ denotes the set of negative roots generated by $\Phi(\chi)$.
We first show
\begin{equation}  \label{E:indC}
 W_S = W(\Phi(\chi))_S \cdot R_{\Phi(\chi)}. 
 \end{equation}
Recall that $w\in R_{\Phi(\chi)}$ if and only if $w^{-1}(\alpha^\vee) > 0$ for every $\alpha \in \Phi(\chi)$. 
It is easy to check $W_S \supseteq W(\Phi(\chi))_S \cdot R_{\Phi(\chi)}$ for every $S \in \msc{P}(\Phi(\chi))$. As the $W_S$'s with $S \in \msc{P}(\Phi(\chi))$ form a partition of $W$, and similarly for $W(\Phi(\chi))_S$'s and $W(\Phi(\chi))$, it then follows from the equality $W = W(\Phi(\chi)) \cdot R_{\Phi(\chi)}$ that \eqref{E:indC} holds for every $S$.

Now Proposition \ref{P:ind} gives that 
$$\sigma_{\Gamma_S} = \Ind_{W(\Phi(\chi))}^W\ \sigma_{\Gamma_{o,S}},$$
where $\sigma_{\Gamma_{o,S}}$ is the Kazhdan-Lusztig representation of $W(\Phi(\chi))$ associated to $W(\Phi(\chi))_S$. By Theorem \ref{T:G-pm} (ii), assuming Conjecture \ref{C:O-ex}, we have
$$\dim \Wh(\pi_{\Gamma_S})_{\mca{O}_y} = \angb{ \sigma_{\msc{X}}^y }{ \sigma_{\Gamma_S} }_W = \angb{ \sigma_{\msc{X}}^y }{ \sigma_{\Gamma_{o,S}} }_{W(\Phi(\chi))}.$$
Decompose $\mca{O}_y$ as $W(\Phi(\chi))$-orbits:  $\mca{O}_y = \sqcup_{i\in I}  \mca{O}_{y_i}$. (In the previous notation, $\mca{O}_{y_i}$ is $\mca{O}_{y_i}^{w_{\Phi(\chi)}}$.) We see that 
\begin{equation} \label{E:tem-d}
\dim \Wh(\pi_{\Gamma_S})_{\mca{O}_y} =  \sum_{i\in I} \angb{ \sigma_{\mca{O}_{y_i}} }{ \sigma_{\Gamma_{o,S}}  }_{W(\Phi(\chi))},
\end{equation}
where $\sigma_{\mca{O}_{y_i}}$ denotes the permutation representation of $W(\Phi(\chi))$ on $\mca{O}_{y_i}$.

We have $\Gamma_{o,S} = \Gamma_o^-$ if and only if $S=\Phi(\chi)$. In view of Lemma \ref{L:obs1}, we see that
$$\angb{ \sigma_{\mca{O}_{y_i}} }{  \sigma_{\Gamma_{o,S}}  }_{W(\Phi(\chi))}  \ge  \angb{ \sigma_{\mca{O}_{y_i}} }{  \sigma_{\Gamma_o^-}  }_{W(\Phi(\chi))}$$
for every $S$. Coupled with \eqref{E:tem-d}, this completes the proof for (i), while that for (ii) is similar.
\end{proof}

In fact, we believe the following holds.
\begin{conj} \label{C:theta-mini}
Let $\wt{G}$ be an $n$-fold persistent covering group. Let $\Theta(\wt{G}, \chi) = \pi_{\Gamma^-}$ be the theta representation associated to an unramified character $\chi$ with $\Phi(\chi) =\Delta$.  Then:
\begin{enumerate}
\item[(i)] $\Theta(\wt{G}, \chi)$ is the least generic representation among all irreducible genuine representations with the same Bernstein support; that is,
$$\dim \Wh(\pi) \ge \dim \Wh(\Theta(\wt{G}, \chi))$$
for every irreducible constituent $\pi$ of an arbitrary unramified genuine principal series of $\wt{G}$.
\item[(ii)] Let $\wt{\GL}_r^{(n)}$ be a Brylinski-Deligne cover of $\GL_r$. If $\dim \Wh(\Theta(\wt{\GL}_r^{(n)}, \chi)) >0$, then for every irreducible genuine representation $\pi$ of $\wt{\GL}_r^{(n)}$, one has $\dim {\rm Wh}_{\psi'}(\pi) > 0$ for some non-degenerate $\psi'$.
\end{enumerate}
\end{conj}

As a remark on Conjecture \ref{C:theta-mini} (i), it is possible that $\Theta(\wt{G}, \chi)$ is generic, but there exists non-generic supercuspidal representation $\pi$ of $\wt{G}$. For example, the consideration of the $\theta_{10}$-phenomenon (cf. \cite{Sri1, HPS, BlSt}) in the covering setting gives such an example. Moreover, there exist generic depth-zero supercuspidal representation $\pi$ such that $\dim \Wh(\Theta(\wt{G}, \chi)) \ge \dim \Wh(\pi) =1$. See \cite[Theorem 3]{Blo} or \cite[Corollary 3.18, Corollary 5.2]{GW}. 

On the other hand, Conjecture \ref{C:theta-mini} (ii) seems to be folkloric, and our belief partly originates from the fact that supercuspidal representations of $\wt{\GL}_r^{(n)}$ are always generic (see \cite[\S 5-\S 6]{GK} and \cite[Theorem I.5.2]{KP}).

We also have the following asymptotic behaviour of $\dim \Wh(\pi_\Gamma)$.
\begin{prop}  \label{P:asymF}
Let $\mbf{G}$ be a simply-connected group and $Q$ a fixed Weyl-invariant quadratic form on its cocharacter lattice such that $Q(\alpha^\vee) =1$ for every short simple coroot.
Let $\wt{G}^{(n)}$ be a saturated cover of $G$ associated with $Q$. If we increase $n$ such that $\wt{G}^{(n)}$ stays in the saturated class, then one has
$$ \dim \Wh(\pi_\Gamma) \sim   \frac{\val{ \msc{X}_{Q,n}  }}{ \val{W} }  \cdot \dim \sigma_\Gamma,$$
where $A(n) \sim B(n)$ means $\lim_{n\to \infty} A(n)/B(n) =1$.
\end{prop}
\begin{proof}
As $G$ is simply-connected, the covering $\wt{G}^{(n)}$ is saturated if and only if $Y_{Q,n} = Y_{Q,n}^{sc}$, see Example \ref{ss-sat}. We first remark that it is shown in \cite[Proposition 2.13]{GSS2} that $\wt{G}_{Q,n}^\vee$ is periodic with respect to $n$ (with $\mbf{G}$ and $Q$ fixed). In particular, there are infinitely many $n$ such that $\wt{G}^{(n)}$ is saturated.

Since $Y_{Q,n}^{sc}$ is a root lattice, there exists $n_o \le {\rm max}\set{n_\alpha: \alpha \in \Delta}$ such that 
$$Y_{Q,n}^{sc} = n_o \cdot Y^\dag$$
where $Y^\dag \subseteq Y$ is a root lattice of the same type as $Y$ of index $[Y: Y^\dag]$ bounded above by 3.

We claim that for $\wt{G}^{(n)}$ stays in the saturated class, one has
\begin{equation} \label{E:lim1}
\lim_{n\to \infty}  \frac{ \val{\OF_{Q,n}} }{ \val{\mca{O}_{Q,n}} } =1.
\end{equation}
One has a bijection $Y/Y_{Q,n} \simeq (n_o^{-1}Y) / Y^\dag$ given by $y \mapsto y/n_o$. Denote by $C$ an alcove of $Y^\dag \otimes \R$ with respect to the affine Weyl group $Y^\dag \rtimes W$. Then $\mca{O}_y$ is not a free-orbit if and only if $y/n_o$ lies on the boundary $\partial \overline{C}$ of $C$. It is easy to see that
$$\lim_{n \to \infty} \frac{ \val{ (\partial \overline{C}) \cap (n_o^{-1} Y)}  }{ \val{ \overline{C} \cap (n_o^{-1} Y) }  } = 0,$$
as $\partial\overline{C}$ is of codimension one in $\overline{C}$. This gives \eqref{E:lim1}.

We have
$$\angb{\sigma_\Gamma}{ \sigma_\msc{X}} = \sum_{\mca{O}_y \in \mca{O}_{Q,n}} \angb{\sigma_\Gamma}{ \sigma_\msc{X}^y } \le \sum_{\mca{O}_y \in \mca{O}_{Q,n}} \angb{\sigma_\Gamma}{ \C[W]} = \val{ \mca{O}_{Q,n} } \cdot   \dim \sigma_\Gamma,$$
and
$$\angb{\sigma_\Gamma}{ \sigma_\msc{X}} = \sum_{\mca{O}_y \in \mca{O}_{Q,n}} \angb{\sigma_\Gamma}{ \sigma_\msc{X}^y } \ge \val{ \OF_{Q,n} } \cdot   \dim \sigma_\Gamma.$$
It is also clear
$$\val{ \OF_{Q,n} } \le \frac{ \val{\msc{X}_{Q,n}} }{ \val{W} }  \le  \val{\mca{O}_{Q,n}}.$$
Now it follows from \eqref{E:lim1} and Theorem \ref{T:G-pm} (iii) that
$$\frac{ \dim \Wh(\pi_\Gamma) }{\dim \sigma_\Gamma} \sim   \frac{\val{ \msc{X}_{Q,n}  }}{ \val{W} }   \sim \val{ \mca{O}_{Q,n} }  \sim \val{ \OF_{Q,n} },$$
as $n$ increases while $\wt{G}^{(n)}$ stays saturated.
\end{proof}

\begin{rmk} \label{R:ub}
In order to remove any potential confusion, we note that $\chi$ and $\Gamma$ are defined only after $\wt{G}^{(n)}$ is chosen, in particular after $n$ is given. Thus more strictly, we should write $\pi_\Gamma$ as $\pi_{\wt{G}^{(n)}, \chi_{G,n}, \Gamma}$, where $\chi_{G, n}$ now denotes the character $\chi$ for $\wt{G}^{(n)}$. What we consider in Proposition 
\ref{P:asymF}  concerns the family of $\pi_{\wt{G}^{(n)}, \chi_{G,n}, \Gamma}$ such that $\wt{G}^{(n)}$ is saturated, $\Phi_0:= \Phi(\chi_{G,n})$ is a subset of $\Delta$ independent of $n$, and thus $\Gamma \subset \Phi_0$ is a common subset of every $\Phi(\chi_{G,n})$. It is therefore validated to consider the asymptotic behavior of $\dim \Wh(\pi_{\wt{G}^{(n)}, \chi_{G, n}, \Gamma})$ as $n$ increases.

It is expected that Proposition \ref{P:asymF} also holds for persistent coverings of a general reductive group. In any case, it follows from loc. cit. that, in contrast with Corollary \ref{C:lb}, $\dim \Wh(\pi_{\Gamma^+})$ does not provide the upper bound for $\dim \Wh(\pi_\Gamma), \pi_\Gamma \in \JH(I(\chi))$, though for linear algebraic groups this is the case as $\pi_{\Gamma^+}$ is the only generic constituent of $I(\chi)$. Indeed, Proposition \ref{P:asymF} implies that for $n$ large, one has
$$\frac{ \dim \Wh(\pi_{\Gamma_1}) }{ \dim \Wh(\pi_{\Gamma_2})  }  \sim  \frac{\dim \sigma_{\Gamma_1}}{ \dim \sigma_{\Gamma_2} }.$$
For more concrete examples, see \S \ref{S:3eg}.
\end{rmk}

\begin{rmk} \label{Gin-ref}
For persistent $\wt{G}$ and $\chi$ a regular unramified genuine character of $Z(\wt{T})$, Theorem \ref{T:G-pm} and its proof entail a refinement (and also the truth) of Ginzburg's conjecture \cite{Gin4}, which states that the unique unramified  subquotient of $I(\chi)$ is non-generic if and only if it is a subrepresentation of the parabolic induction from a non-generic theta representations on the Levi subgroup. For instance, the equality  \eqref{E:chain}
$$ \dim \Wh(\pi_{\Gamma^-})_{\mca{O}_y} = \sum_{i\in I} \dim \Wh(\pi_{\Gamma_M^-})_{\mca{O}_{y_i}^{w_S}},
$$
when applied to $S=\Phi(\chi)$, is such a refinement.
\end{rmk}

\begin{eg} \label{E:sing}
Suppose $\Phi(\chi)= \set{\alpha} \subset \Delta$. Then we have
$$\begin{tikzcd}
\pi_+  \ar[r, hook]  & I(\chi)   \ar[r, two heads] & \pi_-,
\end{tikzcd}$$
where $\pi_+:=\pi_{\Gamma^+}$ and $\pi_-:=\pi_{\Gamma^-}$. Denote $W_\alpha=\set{\text{id}, w_\alpha} \subset W$. We also write $\sigma_+:=\sigma_{\Gamma^+}$ and $\sigma_-:=\sigma_{\Gamma^-}$ for the Kazhdan-Lusztig representation . Then $\sigma_+= \Ind_{W_\alpha}^W \mathbbm{1}$ and $\sigma_- =\Ind_{W_\alpha}^W \varepsilon_{W_\alpha}$, by Proposition \ref{P:ind}. If $\wt{G}$ is persistent, then
$$\dim \Wh(\pi_{\Gamma^+}) =\angb{\sigma_\msc{X}}{ \sigma_+ } \text{ and }  \dim \Wh(\pi_{\Gamma^-})= \angb{\sigma_\msc{X}}{ \sigma_- }.$$
This is clearly compatible with Remark \ref{R:1}, since $\sigma_+ \oplus \sigma_- = \C[W]$. Let 
$$a_\alpha:=\dim \Wh(\pi_{\Gamma^-}) \text{ and } b_\alpha:= \dim \Wh(\pi_{\Gamma^+}) - \dim \Wh(\pi_{\Gamma^-})= \val{ (\msc{X}_{Q,n})^{W_\alpha}}.$$
Then in \cite[\S 4.7]{GSS2}, we have given an interpretation of the number $a_\alpha$ (resp. $b_\alpha$) as exponent of the Plancherel measure (resp. gamma or metaplectic gamma factor) that appears in the determinant of the Shahidi local coefficients matrix associated to $\mca{T}(w_\alpha, i(\chi))$ for $\chi$ in general position.
\end{eg}

\section{A dual argument and $\dim \Wh( \pi_{\Gamma^+})_{\mca{O}_y}$} \label{S:Gamma+}
First we show that Theorem \ref{T:G-pm} implies the following:
\begin{prop} \label{P:St-1}
Assume $\wt{G}$ is persistent and $\Phi(\chi) \subset \Delta$.
One has 
\begin{equation} \label{E:tem-r1}
\sigma_{\Gamma^+}= \Ind_{W(\Phi(\chi))}^W \mbm{1}_{W(\Phi(\chi))} ,
\end{equation}
and therefore 
$$\dim \Wh(\pi_{\Gamma^+}) = \angb{\sigma_\msc{X}}{  \sigma_{\Gamma^+}  }_W = \angb{\sigma_\msc{X} }{ \mbm{1} }_{W(\Phi(\chi))},$$
which is the number of $W(\Phi(\chi))$-orbits in $\msc{X}_{Q,n}$.
\end{prop}
\begin{proof}
Similar to Lemma \ref{L:key1}, it is not hard to see that $W_{\Gamma^+} = R_{\Phi(\chi)}$. Thus \eqref{E:tem-r1} follows from Proposition \ref{P:ind}. The rest is clear in view of  Theorem \ref{T:G-pm} (iii).
\end{proof}

In particular, we see that $\pi_{\Gamma^+}$ is always generic. As mentioned, if $\Phi(\chi) =\Delta$, then $\pi_{\Gamma^+}$ is the covering analogue of the Steinberg representation. In any case, Proposition \ref{P:St-1} could also be viewed as a generalization of Rodier's result \cite[Proposition 4]{Rod4} for linear algebraic groups, compared to Theorem \ref{T:dis}.

\vskip 10pt

The goal of this section is to prove the analogue of Theorem \ref{T:G-pm} (i) for $\pi_{\Gamma^+}$. That is, we prove that Conjecture \ref{C:S} holds for $\dim \Wh(\pi_{\Gamma^+})_{\mca{O}_y}$ for every persistent $W$-orbit $\mca{O}_y \subseteq \msc{X}_{Q,n}$. This refines Proposition \ref{P:St-1}.  The proof in fact goes parallel with that for $\pi_{\Gamma^-}$, and  follows from a ``dual argument". However, as at one point (see Proposition \ref{P:wd-}) there is a crucial contrast between the two cases of $\pi_{\Gamma^+}$ and $\pi_{\Gamma^-}$, which is of independent interest, we present more details of the proof below.

\subsection{When $\Phi(\chi)=\Delta$}
In this subsection, we assume $\Phi(\chi) = \Delta$ and denote
$$\fwchi:={}^{w_G} \chi.$$
Clearly $\Phi(\fwchi) = -\Delta$. 
The analogue of the following is proved for $\pi_{\mca{C}^-}$ (i.e., the theta representation) in \cite{KP, Ga2}. We show that it holds for $\pi_{\mca{C}^+}$ as well.

\begin{prop} \label{red1}
Let $\wt{G}$ be a general covering group and $\chi$ a regular character with $\Phi(\chi)=\Delta$. Then:
\begin{enumerate}
\item[(i)] $\pi_{ \mca{C}^+ }$ is the image of 
$$T(w_G, \fwchi): I(\fwchi) \to I({}^{w_G}(\fwchi));$$
in fact, it is the unique irreducible quotient of $I(\fwchi)$.
\item[(ii)] $\Ker\left(T(w_G, \fwchi): I(\fwchi) \to I({}^{w_G}(\fwchi)) \right)$ is generated by 
$$\set{\Ima\left(T(w_\alpha, {}^{w_\alpha} (\fwchi) ): I({}^{w_\alpha} (\fwchi))  \to I(\fwchi) \right): \ \alpha\in \Delta}.$$
Consequently,
\begin{equation} \label{E:Wh-St}
\Wh( \pi_{ \mca{C}^+ }  )_{\mca{O}_y} \simeq \bigcap_{\alpha\in \Delta} \Ker\left(  T(w_\alpha, {}^{w_\alpha}(\fwchi))^*_{\mca{O}_y}: \Wh(I(\fwchi))_{\mca{O}_y} \to \Wh(I(^{w_\alpha} (\fwchi)))_{\mca{O}_y} \right)
\end{equation}
for every orbit $\mca{O}_y \subseteq \msc{X}_{Q,n}$.
\end{enumerate}
\end{prop}
\begin{proof}
The proof here follows the same argument as in \cite[Theorem I.2.9]{KP}. However, since the details of the proof of statement (d) in loc. cit. are omitted, we give the argument for completeness.

For (i), the fact that $\pi_{\mca{C}^+}$ is the image of the intertwining map $T(w_G, \chi^\flat): I(\fwchi) \to I({}^{w_G} (\fwchi))$ follows from Theorem \ref{T:Rodier} (i). We show the uniqueness. Suppose that 
$$I(\fwchi) \onto \pi,$$
where $(\pi, V_\pi)$ is an irreducible quotient of $I(\fwchi)$. Let $w \in W$ be such that $\delta_B^{1/2}\cdot i({}^w (\fwchi)) \subset \pi_U$. Then 
$$\pi \into I({}^w (\fwchi)).$$
The composite $I(\chi^\flat) \onto \pi \into I({}^w (\fwchi))$ gives a nonzero map in $\Hom(I(\chi^\flat), I({}^w (\chi^\flat)))$, which by Proposition \ref{J-ker} (i) must be a (nonzero) multiple of the intertwining map $T(w, \chi^\flat): I(\fwchi) \to I({}^w (\fwchi))$. Thus, $\pi=\Ima(T(w, \chi^\flat))$. As 
$$T(w_G, \chi^\flat)= T(w_G\cdot w^{-1}, {}^w \chi^\flat) \circ T(w, \chi^\flat),$$
we see that
$$\pi \onto \pi_{\mca{C}^+}.$$
Since $\pi$ is irreducible, one has $\pi \simeq \pi_{\mca{C}^+}$ and $w=w_G$.

To prove (ii), for readability, we denote by $T_{w, \chi}$ (or even just $T_w$ if no confusion arises) the map $T(w, \chi)$. Denote  by
$$\left\langle \Ima(T_{w_\alpha, {}^{w_\alpha}(\fwchi)}): \alpha\in \Delta  \right\rangle \subset I(\fwchi)$$
the submodule of $I(\fwchi)$ which is generated by $\Ima(T_{w_\alpha, {}^{w_\alpha}(\fwchi)})$ for all $\alpha\in \Delta$. Note that the composite
\begin{equation} 
\begin{tikzcd} 
I({}^{w_\alpha} (\fwchi))  \ar[r, "T_{w_\alpha}"]  & I(\fwchi)  \ar[r, "T_{w_G}"]  & I( {}^{w_G}(\fwchi)) 
\end{tikzcd}
\end{equation}
is 0, since $T_{w_G} \circ T_{w_\alpha} = T_{w_G \cdot w_\alpha^{-1}} \circ T_{w_\alpha} \circ T_{w_\alpha}$ with $T_{w_\alpha} \circ T_{w_\alpha}=0$. Here the equality $T_{w_\alpha} \circ T_{w_\alpha}=0$ follows from Proposition \ref{J-ker} (i) and the fact that $T_{w_\alpha} \circ T_{w_\alpha}(f) = 0$ for the normalized unramified vector $f$. We thus have
\begin{equation} \label{one-incl}
\Ima(T_{w_\alpha, {}^{w_\alpha}(\fwchi)}) \subset \Ker(T_{w_G})
\end{equation}
for all $\alpha \in \Delta$.

Now to show the other inclusion $\left\langle \Ima(T_{w_\alpha, {}^{w_\alpha}(\fwchi)}): \alpha \in \Delta \right\rangle \supset \Ker(T_{w_G})$, we first show the following inclusion:
\begin{equation} \label{claim1}
\left\langle \Ima(T_{w, {}^{w^{-1}}(\fwchi)}): w\ne \text{id} \in W \right\rangle \supset \Ker(T_{w_G})
\end{equation}
To prove \eqref{claim1}, it suffices to show 
\begin{equation} \label{claim2}
\bigcup_{w\ne \text{id} \in W} \Ima(T_{w, {}^{w^{-1}}(\fwchi)})_U \supset \Ker(T_{w_G})_U,
\end{equation}
which is equivalent to
\begin{equation} \label{claim3}
\bigcap_{w\ne \text{id} \in W} \left( I(\fwchi)_U - \Ima(T_{w, {}^{w^{-1}}(\fwchi)})_U \right) \subset \set{ \delta_B^{1/2} \cdot i({}^{w_G} (\fwchi))},
\end{equation}
since 
$$\Ker(T_{w_G})_U= \bigcup_{w \ne w_G\in W} \delta_B^{1/2} \cdot i({}^w (\fwchi)).$$
For the purpose of proving \eqref{claim3}, suppose on the contrary that there exists $\delta_B^{1/2}\cdot i({}^{w'} (\fwchi))$ with $w' \ne w_G$ lying in the left hand side of \eqref{claim3}. Then $\delta_B^{1/2} \cdot i({}^{w'} (\fwchi))$ does not appear in $\Ima(T_{w, {}^{w^{-1}}(\fwchi)})_U$ for every $w\ne {\rm id} \in W$. Thus, 
$$\delta_B^{1/2} \cdot i({}^{w'} (\fwchi)) \subset \Ker(T_{w, {}^{w^{-1}}(\fwchi)})_U$$ 
for every $w \ne {\rm id}$, as $\Ker(T_{w, {}^{w^{-1}}(\fwchi)})_U$ and $\Ima(T_{w, {}^{w^{-1}}(\fwchi)})_U$ have empty intersection. This gives us a contradiction for  $w= w'^{-1} \cdot w_G$ as follows:
\begin{enumerate}
\item[$\bullet$] The composite
$$\begin{tikzcd}
I({}^{w_G w'} (\fwchi))  \ar[r, "T_{w'^{-1} w_G}"]  & I(\fwchi) \ar[r, "T_{w'}"] & I({}^{w'} (\fwchi))
\end{tikzcd}$$
is the map $T_{w_G}$. This shows that $\delta_B^{1/2} \cdot i({}^{w'} (\fwchi)) \subset \Ker(T_{w'^{-1} w_G})_U \subset \Ker(T_{w_G})_U$; however $\delta_B^{1/2}\cdot i({}^{w'} (\fwchi)) \subset \Ima(T_{w_G})_U$ as well. This is a contradiction.
\end{enumerate}
Therefore, \eqref{claim1} is proved. 

Note that for $w=w_\alpha w_1 \ne {\rm id}$ with $\alpha\in \Delta$, we have
$$\Ima(T_{w_\alpha, {}^{w_\alpha}(\fwchi)}) \supset \Ima(T_{w, {}^{w^{-1}}(\fwchi)  }).$$
It thus follows from \eqref{claim1} that
\begin{equation} \label{claim4}
\left\langle \Ima(T_{w_\alpha, {}^{w_\alpha}(\fwchi)}): \alpha\in \Delta \right\rangle \supset \Ker(T_{w_G}).
\end{equation}
This inclusion coupled with \eqref{one-incl} proves the first statement in (ii).  The equality \eqref{E:Wh-St} follows from the same argument as in \eqref{O-Cap}.
\end{proof}

\begin{cor} \label{C:iff-2}
Let $\mca{O}_z \subseteq \msc{X}_{Q,n}$ be a $W$-orbit. Let $\lambda_\cc^{\fwchi} \in \Wh(I(\fwchi))_{\mca{O}_z}$ be the $\psi$-Whittaker functional of $I(\fwchi)$ associated to some $\cc \in \Ftn(i(\fwchi))$. Then, $\lambda_\cc^{\fwchi}$ lies in $\Wh(\pi_{\mca{C}^+})_{\mca{O}_z}$ if and only if for every simple root $\alpha\in \Delta$ one has
\begin{equation} \label{C:crucial}
\cc\big( \s_{w_\alpha[y]}) \big) = - q^{-k_{y,\alpha} } \cdot \vep^{ \angb{y_\rho}{\alpha} \cdot D(y, \alpha^\vee) } \cdot \g(\angb{y_\rho}{\alpha}Q(\alpha^\vee))^{-1} \cdot \cc(\s_y) \text{ for every } y \in \mca{O}_z.
\end{equation}
\end{cor}
\begin{proof}
Consider  $\lambda_{\gamma}^{\fwchi} \in \Wh(I(\fwchi))$ and $\alpha\in \Delta$, we have from \eqref{tau-Mat}
$$T(w_\alpha, {}^{w_\alpha}\chi)^*(\lambda_{\gamma}^{\fwchi}) = \sum_{\gamma'} \tau(w_\alpha, {}^{w_\alpha} (\fwchi), \gamma, \gamma') \cdot \lambda_{\gamma'}^{{}^{w_\alpha} (\fwchi)}.$$
Let $\cc\in \Ftn(i(\fwchi))$ and let $\lambda_\cc^{\fwchi} = \sum_{\gamma \in \wt{T}/\wt{A}} \cc(\gamma) \lambda_{\gamma}^{\fwchi} \in \Wh(I(\fwchi))$ be the associated functional. Then,
\begin{align*}
T(w_\alpha, {}^{w_\alpha}(\fwchi))^*(\lambda_\cc^{\fwchi}) & = \sum_{\gamma} \cc(\gamma) \left( \sum_{\gamma'} \tau(w_\alpha, {}^{w_\alpha}(\fwchi), \gamma, \gamma') \cdot \lambda_{\gamma'}^{{}^{w_\alpha}\fwchi} \right) \\
& = \sum_{\gamma'} \left( \sum_{\gamma} \cc(\gamma) \tau(w_\alpha,  {}^{w_\alpha}(\fwchi), \gamma, \gamma') \right) \lambda_{\gamma'}^{{}^{w_\alpha}\fwchi} .
\end{align*}

By \eqref{E:Wh-St}, a function $\cc \in \Ftn(i(\fwchi))$ gives rise to a functional in $\Wh(\pi_{\mca{C}^+})_{\mca{O}_z}$ (i.e. $\lambda_\cc^{\fwchi} \in \Wh(\pi_{\mca{C}^+})_{\mca{O}_z}$) if and only if
for every $\alpha\in \Delta$,
$$\sum_{y_1 \in \mca{O}_z} \cc(\s_{y_1}) \tau(w_\alpha,  {}^{w_\alpha}(\fwchi),  \s_{y_1}, \s_y)=0 \text{ for all } y\in \mca{O}_z.$$
It follows from Theorem \ref{T:tau} that the above equality is equivalent to that the equality
$$\cc (\s_y) \cdot \tau(w_\alpha, {}^{w_\alpha}(\fwchi),  \s_y, \s_y) +\cc( \s_{w_\alpha[y]}) \cdot \tau(w_\alpha, {}^{w_\alpha}(\fwchi),  \s_{w_\alpha[y]}, \s_y)=0$$
holds for every $\alpha\in \Delta$ and $y\in \mca{O}_z$. Again, Theorem \ref{T:tau} gives
\begin{align*}
&\cc(\s_{w_\alpha[y]})\\
=& -(1-q^{-1})\frac{(\fwchi(\wt{h}_\alpha(\varpi^{n_\alpha})))^{-k_{y,\alpha}}}{1-\fwchi(\wt{h}_\alpha(\varpi^{n_\alpha}))^{-1}} \cdot \vep^{ \angb{y_\rho}{\alpha} \cdot D(y, \alpha^\vee) } \cdot \g(\angb{y_\rho}{\alpha}Q(\alpha^\vee))^{-1} \cdot \cc(\s_y) \\
=& - q^{-k_{y,\alpha}} \cdot \vep^{ \angb{y_\rho}{\alpha} \cdot D(y, \alpha^\vee) } \cdot \g(\angb{y_\rho}{\alpha}Q(\alpha^\vee))^{-1} \cdot \cc(\s_y) \ \text{ since } \Phi(\fwchi)=-\Delta.
\end{align*}
This completes the proof.
\end{proof}

Now for any $y\in Y$ and $\alpha\in \Delta$, we define:
$$\mathbf{d}(w_\alpha, y):=- q^{-k_{y,\alpha} } \cdot \vep^{ \angb{y_\rho}{\alpha} \cdot D(y, \alpha^\vee) } \cdot \g(\angb{y_\rho}{\alpha}Q(\alpha^\vee))^{-1},$$
where
$$ k_{y,\alpha}=\ceil{\frac{\angb{y_\rho}{\alpha}+1}{n_\alpha}}.$$
Clearly $\mathbf{d}(w_\alpha, y) \ne 0$ for all $\alpha\in \Delta$ and $y\in Y$. For $w=w_k w_{k-1}...w_2 w_1\in W$ in a minimum expansion, define
$$\mathbf{d}(w, y):=\prod_{i=1}^k \mathbf{d}(w_i, w_{i-1}... w_1[y]).$$
The following result plays a crucial role, and it is in contrast with \cite[Proposition 3.10]{Ga2} for $\pi_{\mca{C}^-}$.

\begin{prop} \label{P:wd-}
Let $\mca{O}_z \subseteq \msc{X}_{Q,n}$ be a persistent orbit. Then for every $y\in \mca{O}_z$ and $w_1, w_2\in W$, one has $\mathbf{d}(w_1 w_2, y) = \mathbf{d}(w_1, w_2[y]) \cdot \mathbf{d}(w_2, y)$. In particular, $\mathbf{d}(w, y)$ is well-defined, independent of the choice of minimum expansion of $w$.
\end{prop}
\begin{proof}
The Weyl group $W$ has the presentation
$$W=\left\langle w_\alpha: (w_\alpha w_\beta)^{m_{\alpha \beta}} =1 \text{ for } \alpha, \beta\in \Delta\right\rangle.$$
Let $y\in \mca{O}_z$ be any element.  We first show that the equality 
\begin{equation} \label{2c}
\mathbf{d}(w_\alpha, w_\alpha[y]) \cdot \mathbf{d}(w_\alpha, y)=1
\end{equation}
holds for all  $\alpha\in \Delta$. There are two cases to consider.
\begin{enumerate}
\item[$\bullet$] First, if $w_\alpha \notin {\rm Stab}_W(y; \msc{X}_{Q,n})$, then $w_\alpha[y] - y \notin Y_{Q,n}^{sc}$. That is, $n_\alpha \nmid \angb{y_\rho}{\alpha}$ and $n_\alpha\nmid \angb{w_\alpha[y]_\rho}{\alpha}$. As in the proof for \cite[Lemma 3.9]{Ga2}, we have
$$k_{w_\alpha[y], \alpha} \cdot k_{y, \alpha}=1.$$
Thus,
$$\begin{aligned}
&  \mathbf{d}(w_\alpha, w_\alpha[y]) \cdot \mathbf{d}(w_\alpha, y) \\
= \ &  \frac{ q^{-1} \cdot \vep^{\angb{w_\alpha[y]_\rho}{\alpha} \cdot D(w_\alpha[y]_\rho, \alpha^\vee) + \angb{y_\rho}{\alpha} \cdot D(y,\alpha^\vee)   }  }{ \g( \angb{w_\alpha[y]_\rho }{\alpha} Q(\alpha^\vee)) \cdot \g( \angb{y_\rho }{\alpha} Q(\alpha^\vee)) } \\
 =\ & q^{-1} \cdot \vep^{\angb{w_\alpha[y]_\rho}{\alpha} \cdot D(w_\alpha[y]_\rho, \alpha^\vee) + \angb{y_\rho}{\alpha} \cdot D(y,\alpha^\vee)   } \cdot q \cdot \vep^{ \angb{y_\rho }{\alpha} Q(\alpha^\vee) } \\
 =\ & \vep^{ \angb{y_\rho}{\alpha}^2 \cdot Q(\alpha^\vee) } \cdot \vep^{  \angb{y_\rho }{\alpha} Q(\alpha^\vee) } \\
 =\ &1.
\end{aligned}$$
\item[$\bullet$] Second, if $w_\alpha \in {\rm Stab}_W(y; \msc{X}_{Q,n})$, then since $\mca{O}_y = \mca{O}_z$ is persistent we have
$w_\alpha[y] - y\in Y_{Q,n}^{sc}$. That is, $w_\alpha[y] - y = - \angb{y_\rho}{\alpha} \alpha^\vee \in Y_{Q,n}^{sc}$. Thus, we have $\angb{y}{\alpha}= kn_\alpha +1$ for some $k\in \Z$. In this case,
$$k_{w_\alpha[y], \alpha} \cdot k_{y, \alpha}=2$$
and also by \eqref{F:gauss},
$$\g( \angb{w_\alpha[y]_\rho }{\alpha} Q(\alpha^\vee))= \g( \angb{y_\rho }{\alpha} Q(\alpha^\vee))=-q^{-1}.$$
It follows easily that $\mathbf{d}(w_\alpha, w_\alpha[y]) \cdot \mathbf{d}(w_\alpha, y)=1$ in this case. 
\end{enumerate}
Therefore, we have shown that \eqref{2c} holds.

Now we show that the braid relation on the function $\mathbf{d}(w_\alpha, y)$ holds. That is, if $m_{\alpha \beta}=3$ for example, then the equality
\begin{equation} \label{rel}
\mathbf{d}(w_\alpha, w_\beta w_\alpha[y])\cdot \mathbf{d}(w_\beta, w_\alpha[y])\cdot \mathbf{d}(w_\alpha, y)=
\mathbf{d}(w_\beta, w_\alpha w_\beta[y])\cdot \mathbf{d}(w_\alpha, w_\beta[y])\cdot \mathbf{d}(w_\beta, y)
\end{equation}
holds. However, the same argument in \cite[page 351-352]{Ga2} shows that this is the case. Indeed, the checking in loc. cit. for \eqref{rel} and its analogues for the cases $m_{\alpha \beta}=4, 6$ is a formal verification, which does not rely on any condition on $y$. 

The above shows that $\mathbf{d}(w_1 w_2, y) = \mathbf{d}(w_1, w_2[y]) \cdot \mathbf{d}(w_2, y)$ for any $w_1, w_2\in W$, as desired.
\end{proof}

\begin{prop} \label{T:C+}
Let $\chi$ be a regular unramified character with $\Phi(\chi) =\Delta$. Then for every persistent orbit $\mca{O}_y \subseteq \msc{X}_{Q,n}$,
$$\dim \Wh (\pi_{\mca{C}^+})_{\mca{O}_y} = \angb{\sigma_\msc{X}^y  }{ \mbm{1}_W } = 1.$$
\end{prop}
\begin{proof}
Let $\mca{O}_y \subseteq \msc{X}_{Q,n}$ be a persistent $W$-orbit. We define a nonzero $\cc_{\mca{O}_y} \in \Ftn(i(\chi^\flat))$ with support $\mca{O}_y\cdot \wt{A}$ as follows. First, let $\cc_{\mca{O}_y}(\s_y)=1$, and for any $w \in W$, define
$$\cc_{\mca{O}_y}(\s_{w[y]}):= \mathbf{d}(w, y) \cdot \cc_{\mca{O}_y}(\s_y).$$
It is well-defined and independent of the minimum decomposition of $w$ by Proposition \ref{P:wd-}. Second, we extend $\cc_{\mca{O}_y}$ to the covering torus $\wt{T}$ by defining
$$\cc_{\mca{O}_y}(\s_{w[y]} \cdot \wt{z}) =  \cc_{\mca{O}_y}(\s_{w[y]}) \cdot \chi^\flat(\wt{z}) \text{ for } \wt{z} \in \wt{A}$$
and
$$\cc(\wt{t})=0 \text{ if } \wt{t} \notin \bigcup_{w\in W} \s_{w[y]}  \cdot \wt{A}.$$
If $\mca{O}_{y} = \mca{O}_{z} \subset \msc{X}_{Q,n}$, then Proposition \ref{P:wd-} gives
$$\C \cdot \cc_{\mca{O}_y} = \C\cdot  \cc_{\mca{O}_z} \in \Ftn(i(\chi^\flat)).$$
That is, every orbit $\mca{O}_y$ in $\msc{X}_{Q,n}$ contributes to a one-dimensional space of $\Wh(\pi_{\mca{C}^+})_{\mca{O}_y}$. On the other hand, Corollary \ref{C:iff-2} implies that every element in $\Wh(\pi_{\mca{C}^+})_{\mca{O}_y}$ arises from such a $\cc_{\mca{O}_y}$. Therefore, 
$$\dim \Wh(\pi_{\mca{C}^+})_{\mca{O}_y}= 1 = \angb{\sigma_\msc{X}^y}{ \mathbbm{1} }_W,$$
where the second equality follows from Lemma \ref{L:obs1}. 
\end{proof}

\subsection{The general case when $\Phi(\chi) \subseteq \Delta$}
The main result in this section is

\begin{thm} \label{T:St-O}
Let $\chi$ be a regular unramified character with $\Phi(\chi) \subseteq \Delta$. For every persistent $\mca{O}_y \subseteq \msc{X}_{Q,n}$, one has
$$\dim \Wh(\pi_{\Gamma^+})_{\mca{O}_y} = \angb{\sigma_\msc{X}^y }{ \sigma_{\Gamma^+}}_W;$$
that is, Conjecture \ref{C:S} holds for $\pi_{\Gamma^+}$.
\end{thm}
\begin{proof}
As the argument is almost the same as that for Theorem \ref{T:G-pm}, we will just sketch the key steps.
First, letting $w_l \in W(\Phi(\chi))$ be the longest element, one checks that $\pi_{\Gamma^+}$ is the image of intertwining operator 
$$T(w_l, \chi^\flat):  I( {}^{w_l} (\chi^\flat) ) \to I(\chi^\flat).$$
Then Proposition \ref{T:C+} coupled with the argument in Theorem \ref{T:G-pm} give that
$$\dim \Wh(\pi_{\Gamma^+})_{\mca{O}_y} = \angb{\sigma_\msc{X}^y}{ \Ind_{W(\Phi(\chi))}^W  \mbm{1}_{W(\Phi(\chi))}  }_{W}.$$
Second, analogous to Lemma \ref{L:key1}, one has $W_{\Gamma^+} = R_{\Phi(\chi)}$ and therefore Proposition \ref{P:ind} implies $\sigma_{\Gamma^+} = \Ind_{W(\Phi(\chi))}^W  \mbm{1}_{W(\Phi(\chi))} $. The result now follows.
\end{proof}

\begin{rmk}
It is possible to give an analysis of $\pi_{\Gamma^+}$ and $\pi_{\Gamma^-}$, by generalizing the argument for $\pi_{\mca{C}^\pm}$ (when $\Phi(\chi)=\Delta$). More precisely, one can show that the analogues of Proposition \ref{red1} (iii) and \cite[Proposition 3.4]{Ga2} hold, where the intersection is then taken over $\Phi(\chi)$ instead of $\Delta$. From such an approach,  Theorem \ref{T:G-pm} (i) and Theorem \ref{T:St-O} can also be deduced.
\end{rmk}



\section{Several explicit examples} \label{S:3eg}


In this section, we consider saturated covers $\wt{G}$ of $\SL_3, \Sp_4$ and the exceptional group ${\rm G}_2$. For such covers, we compute explicitly $\dim \Wh(\pi_\Gamma)$. In fact, we only consider ${\rm JH}(I(\chi))$ with $\Phi(\chi)=\Delta =\set{\alpha_1, \alpha_2}$, as the case $\Phi(\chi)=\set{\alpha_i} \subset \Delta$ follows from Example \ref{E:sing}. In this case, one has
$$\JH(I(\chi))= \set{ \pi_{\emptyset}, \pi_{\Delta}, \pi_{\set{\alpha_1}} \text{ and } \pi_{\set{\alpha_2}}  }.$$
For $\SL_3$ and $\Sp_4$, we follow the labelling in Example \ref{SL3} and Example \ref{Sp4} instead, and thus $\JH(I(\chi))= \set{ \pi_{\Gamma^+}, \pi_{\Gamma^-}, \pi_{\Gamma_1} \text{ and } \pi_{\Gamma_2}  }$, where $W_{\Gamma_1}= \mfr{C}_1$ and $W_{\Gamma_2}= \mfr{C}_2$.
\subsection{Covers of $\SL_3$}

Retain the notations in Example \ref{SL3}. In particular, $\alpha_1, \alpha_2$ are the two simple roots. Put $\alpha_3=\alpha_1 + \alpha_2$. We have $\alpha_3^\vee =\alpha_1^\vee + \alpha_2^\vee$. We fixed the quadratic form $Q$ on $Y$ such that
$$Q(\alpha_1^\vee)=Q(\alpha_2^\vee)=1.$$

\begin{lm}
The group $\wt{G}:=\wt{\SL}_3^{(n)}$ is saturated if and only if $3\nmid n$.
\end{lm}
\begin{proof}
It follows from the definition that saturation of $\wt{\SL}_3^{(n)}$ is equivalent to $Y_{Q,n}=Y_{Q,n}^{sc}$; equivalently, the dual group $\wt{G}^\vee$ is of adjoint type, namely $\text{PGL}_3$. The claim then follows from a simple combinatorial calculation with $Y_{Q,n}$, and we omit the details (cf. \cite[\S 2.7]{We6}).
\end{proof}

Thus, we assume $3\nmid n$, which then gives 
$$\msc{X}_{Q,n}\simeq (\Z/n\Z)\alpha_1^\vee \oplus (\Z/n\Z)\alpha_2^\vee.$$
The character of $\sigma_\msc{X}$ is given in Table 3.

\begin{table}[!htbp]  \label{T3}
\caption{Character of $\sigma_{\msc{X}}$ for saturated $\wt{\SL}_3^{(n)}$}
\vskip 5pt
\begin{tabular}{|c|c|c|c|c|c|c|}
\hline
 & id  &  $w_1$ & $w_2$  & $w_1 w_2$ &  $ w_2 w_1 $  & $w_G$ \\
\hline
$ \chi_{\sigma_\msc{X}} $ & $n^2$ & $n$  & $n$ & $1$  & $1$ &  $n$ \\ 
\hline
\end{tabular}
\end{table}
Following the notation in Example \ref{SL3}, we have
$$W_{\Gamma^\pm}=\mfr{C}^\pm, \quad W_{\Gamma_1}=\mfr{C}_1=\set{w_1, w_1 w_2}, \quad W_{\Gamma_2}=\mfr{C}_2=\set{w_2, w_2 w_1},$$
which gives $\sigma_{\Gamma^\pm}$ and $\sigma_{\Gamma_i}$. We obtain Table 4 from Theorem \ref{T:G-pm} (iii).
\begin{table}[!htbp]  \label{T4}
\caption{$\dim \Wh(-)$ for saturated $\wt{\SL}_3^{(n)}$}
\vskip 5pt
\begin{tabular}{|c|c|c|c|c|}
\hline
 & $\pi_{\Gamma^+}$  &  $\pi_{\Gamma_1} $ & $\pi_{\Gamma_2}$  & $\pi_{\Gamma^-}$  \\
\hline
$ \dim \Wh(-) $ & $ \frac{n^2+3n+2}{6} $ & $\frac{n^2-1}{3}$  & $ \frac{n^2-1}{3}$ & $\frac{n^2-3n + 2}{6}$ \\ 
\hline
\end{tabular}
\end{table}

We will illustrate further on the equality $\dim \Wh(\pi_{\Gamma})_{\mca{O}_y} = \angb{ \sigma_{\msc{X}}^y }{ \sigma_{\Gamma} }$ proven in Corollary \ref{C:le2}. Since $\dim \Wh (\pi_{\Gamma^\pm})_{\mca{O}_y}$ is well-understood from Theorem \ref{T:G-pm} (i) and Theorem \ref{T:St-O}, we will concentrate on $\pi_{\Gamma_i}, i=1, 2$.

Consider the intertwining operator
$$T(w_1 w_2, {}^{w_2}\chi):  I({}^{w_2} \chi) \to I({}^{w_1} \chi).$$
Since $w_1 \in W_{\Gamma_1}$ and $w_2\in W_{(\Gamma_1)_\text{op}}$,  the image of $T(w_1 w_2, {}^{w_2}\chi)$ is just $\pi_{\Gamma_1}$. Therefore, for $y\in \msc{X}_{Q,n}$, we have
$$\dim \Wh( \pi_{\Gamma_1} )_{\mca{O}_y} = \text{rank of } T(w_1 w_2, {}^{w_2}\chi)^*_{\mca{O}_y}.$$
For simplicity, we will only compute explicitly for the case $n=2$, and justify again that
\begin{equation} \label{SL-y}
\dim \Wh(\pi_{\Gamma_1})_{\mca{O}_y}= \angb{\sigma_{\msc{X}}^y}{\sigma_{\Gamma_1}}
\end{equation}
holds for any $y\in \msc{X}_{Q,2}$.

For $n=2$, one has $Y_{Q,2}=2Y^{sc}$. We take the following ordered representatives in $Y$ for $\msc{X}_{Q,2}$:
$$R:=\set{0, \alpha_1^\vee, \alpha_2^\vee, \alpha_3^\vee}.$$
We could partition $R$ into different orbits. That is, $\set{z_i } \subset R$ is said to be in an orbit if $\set{ f(\hat{z_i}) } \subset \msc{X}_{Q,n}$ forms an orbit  in $\msc{X}_{Q,n}$. By abuse of notation, we write $\mca{O}_{z_i} = \set{z_i}$ in this case. It follows that
$$R= \mca{O}_0 \sqcup \mca{O}_{\alpha_3^\vee}$$
where
$$\mca{O}_0= \set{0, \alpha_1^\vee, \alpha_2^\vee } \text{ and }  \mca{O}_{\alpha_3^\vee}= \set{\alpha_3^\vee}.$$
In particular, $\alpha_3^\vee \in (\msc{X}_{Q,2})^W$. It then follows from Theorem \ref{T:dis} that 
$$\dim \Wh(\pi_{\Gamma_1})_{\mca{O}_{\alpha^\vee_3}}= \angb{\sigma_{\msc{X}}^{\alpha_3^\vee}}{\sigma_{\Gamma_1}}=0.$$
Thus, we are left to consider $\mca{O}_0$.

We will compute explicitly the scattering matrix
$$\left[ \tau(w_1 w_2, {}^{w_2} \chi, \s_{y}, \s_{z})  \right]_{y, z\in \mca{O}_0 }  . $$
The cocycle relation \eqref{SLCM2} gives that
\begin{equation}
\tau(w_1 w_2, {}^{w_2} \chi, \s_{y}, \s_{z}) = \sum_{x\in R} \tau(w_1,  \chi, \s_{y}, \s_{x}) \cdot \tau( w_2, {}^{w_2} \chi, \s_{x}, \s_{z}).
\end{equation}
By using the explicit form of the rank-one scattering matrix in Theorem \ref{T:tau} and the fact that $\Phi({}^{w_2} \chi)=\set{-\alpha_2^\vee, \alpha_3^\vee}$, we obtain the matrix with respect to the ordered set $\mca{O}_0$:
$$ \left[ \tau(w_1 w_2, {}^{w_2} \chi, \s_{y}, \s_{z})  \right]_{y, z\in \mca{O}_0}
= 
\begin{bmatrix}
-q^{-1} &  -q^{-1}(1+q^{-1})  \xi \g(1) &  \xi \g(1)    \\
-q^{-1} \g(-1) & -q^{-2} (1+q^{-1}) & q^{-1}   \\
 0 & 0 & 0
\end{bmatrix},$$
the rank of which is clearly 1. On the other hand, since 
$$\sigma_{\msc{X}}^0 \oplus \mbm{1}_\W = \sigma_{\msc{X}},$$
we have $\angb{\sigma_{\msc{X}}^0}{ \sigma_{\Gamma_1} }=1$.  Therefore, the equality \eqref{SL-y} holds for $\mca{O}_0$ as expected. 

One can verify in a similar way \eqref{SL-y} for $\pi_{\Gamma_2}$ explicitly.  Note that $w_G W_{\Gamma_1} w_G= W_{\Gamma_2}$ in this case. Thus, the equalities 
$$\dim \Wh(\pi_{\Gamma_1})_{\mca{O}_{\alpha_3^\vee}}=\dim \Wh(\pi_{\Gamma_2})_{\mca{O}_{\alpha_3^\vee}}=0$$
and 
$$\dim \Wh(\pi_{\Gamma_1})_{\mca{O}_0}=\dim \Wh(\pi_{\Gamma_2})_{\mca{O}_0}=2,$$
which are clear from the above consideration, also follow from Proposition \ref{P:sym}.

\subsection{Covers of $\Sp_4$}
Let $\alpha_1^\vee$ be the long simple coroot and $\alpha_2^\vee$ the short simple coroot:

$$ \qquad 
\begin{picture}(5.7,0.2)(0,0)
\put(2.5,0){\circle{0.08}}
\put(3,0){\circle{0.08}}
\put(2.54,0.025){\line(1,0){0.42}}
\put(2.54,-0.025){\line(1,0){0.42}}
\put(2.7,-0.04){$<$}
\put(2.5,0.1){\footnotesize $\alpha_2^\vee$}
\put(3,0.1){\footnotesize $\alpha_1^\vee$}
\end{picture}
$$
Put 
$$\alpha_3^\vee:= \alpha_1^\vee + 2\alpha_2^\vee \text{  and } \alpha_4^\vee := \alpha_1^\vee + \alpha_2^\vee,$$
which gives 
$$\alpha_3= \alpha_1 + \alpha_2 \text{ and } \alpha_4= 2\alpha_1 + \alpha_2.$$
We fix the quadratic form $Q$ on $Y$ such that
$$Q(\alpha_2^\vee)=1,$$
which then implies $Q(\alpha_1^\vee)=2$.

The group $\wt{G}:=\wt{\Sp}_4^{(n)}$ is saturated if and only if $n$ is odd. Thus, we assume $2\nmid n$ for the rest of the subsection. We chose the basis $e_1:=\alpha_1^\vee + \alpha_2^\vee$ and $e_2:=\alpha_2^\vee$ of $Y$. For a saturated $\wt{G}$, we have
$$Y_{Q,n}=Y_{Q,n}^{sc}= n Y^{sc}= n\Z e_1 \oplus n\Z e_2.$$
The representation $\sigma_\msc{X}$ on $\msc{X}_{Q,n}\simeq (\Z/n\Z) \oplus (\Z/n\Z)$ has the character given in Table 5. 
\begin{table}[!htbp]  \label{T5}
\caption{Character of $\sigma_{\msc{X}}$ for saturated $\wt{\Sp}_4^{(n)}$}
\vskip 5pt
\begin{tabular}{|c|c|c|c|c|c|c|c|c|}
\hline
 & id  &  $w_1$ & $w_2$  & $w_1 w_2$ &  $ w_2 w_1 $  & $w_1 w_2 w_1$ &  $w_2 w_1 w_2$  &  $w_G$ \\
\hline
$ \chi_{\sigma_\msc{X} } $ & $n^2$ & $n$  & $n$ & $1$  & $1$ &  $n$ &  $n$  &  1 \\ 
\hline
\end{tabular}
\end{table}

Following the notations from Example \ref{Sp4}, we denote 
$$\Gamma_1:=\mfr{C}_1=\set{w_1, w_1 w_2, w_1 w_2 w_1 } \text{ and } \Gamma_2:= \mfr{C}_2=\set{w_2, w_2 w_1, w_2 w_1 w_2 }.$$  
Theorem \ref{T:G-pm} (iii) then gives Table 6.
\begin{table}[!htbp]  \label{T6}
\caption{$\dim \Wh(\pi_\Gamma)$ for saturated $\wt{\Sp}_4^{(n)}$}
\vskip 5pt
\begin{tabular}{|c|c|c|c|c|}
\hline
 & $\pi_{\Gamma^+}$  &  $\pi_{\Gamma_1} $ & $\pi_{\Gamma_2}$  & $\pi_{\Gamma^-}$  \\
\hline
$ \dim \Wh(-) $ & $ \frac{n^2+4n+3}{8} $ & $\frac{3(n^2-1)}{8}$  & $ \frac{3(n^2-1)}{8}$ & $\frac{n^2-4n + 3}{8}$ \\ 
\hline
\end{tabular}
\end{table}

Again, we illustrate further on the equality
\begin{equation} \label{Sp-y}
\dim \Wh(\pi_{\Gamma_i})_{\mca{O}_y} = \angb{\sigma_{\msc{X}}^y}{ \sigma_{\Gamma_i}  }
\end{equation}
for every $y\in \msc{X}_{Q,3}$ when $n=3$. The intertwining operator
$$T(w_1 w_2, {}^{w_2}\chi):  I({}^{w_2} \chi) \to I({}^{w_1} \chi)$$
has image exactly $\pi_{\Gamma_1}$.   Since $Y_{Q,3}=3Y^{sc}$, we take the following ordered representatives $R\subset Y$ for $\msc{X}_{Q,3}$:
$$R:=\set{0, \alpha_2^\vee, \alpha_1^\vee, 2\alpha_3^\vee, \alpha_4^\vee, \alpha_3^\vee, 2\alpha_4^\vee, 2\alpha_1^\vee, 2\alpha_2^\vee }.$$
The decomposition of $R$ into Weyl-orbits is as follows:
$$\mca{O}_0=\set{0, \alpha_2^\vee, \alpha_1^\vee, 2\alpha_3^\vee },  \quad \mca{O}_{\alpha_4^\vee}=\set{\alpha_4^\vee, \alpha_3^\vee, 2\alpha_4^\vee, 2\alpha_1^\vee}, \quad \mca{O}_{2\alpha_2^\vee}= \set{2\alpha_2^\vee}.$$

Since \eqref{Sp-y} holds for $\mca{O}_{2\alpha_2^\vee}$ by Theorem \ref{T:dis}, it suffices to consider the first two orbits. Note that we have $\xi=(-1, \varpi)_n=1$ in this case since $n$ is odd.  Similar to the computation for $\wt{\SL}_3^{(n)}$, one uses the cocycle relation \eqref{SLCM2}, Theorem \ref{T:tau} and the fact that $\Phi({}^{w_2} \chi)=\set{-\alpha_2^\vee, \alpha_3^\vee}$ to obtain
$$\begin{aligned}
& [T(w_1w_2, {}^{w_2} \chi)^*]_{\mca{O}_0} \\
= & 
\begin{bmatrix}
-q^{-1} &  \g(1) &  -q^{-1} \g(-1)(1+q^{-1})  & 0 \\
\g(-1) &     -1 &     0  &   -\g(1)(1+q^{-1}) \\
 -q^{-1}\g(1) &  \g(1)^2  &  -q^{-2} (1+q^{-1})    & 0 \\
 q^2 \g(-1)^2 &  -\g(-1) q^2    & 0 &   -q(1+ q^{-1})
\end{bmatrix},
\end{aligned}
$$
the rank of which is 2. Moreover,
$$\begin{aligned}
& [T(w_1w_2, {}^{w_2} \chi)^*]_{\mca{O}_{\alpha_4^\vee}  } \\
= & 
\begin{bmatrix}
0 &  0 &  0  & 0 \\
 \g(-1) &  -1  &  -q^{-1} \g(-1)    &   q^{-2} \\
 q^{-1}  &  -\g(1)     &   -q^{-2}  &   q^{-2} \g(1) \\
 0 &  0  &  0  &  0
\end{bmatrix},
\end{aligned}
$$
the rank of which is 1. On the other hand, it can be computed easily that
$$\angb{ \sigma_{\msc{X}}^0 }{ \sigma_{\Gamma_1}  } =2, \quad \angb{ \sigma_{\msc{X}}^{\alpha_4^\vee} }{ \sigma_{\Gamma_1}  } =1$$
Therefore, \eqref{Sp-y} is verified explicitly. An analogous computation for $\pi_{\Gamma_2}$ shows that \eqref{Sp-y} holds for every $y\in \msc{X}_{Q,3}$ as well.

\subsection{Covers of ${\rm G}_2$}
Consider ${\rm G}_2$ with Dynkin diagram for its simple coroots:

$$
\begin{picture}(5.7,0.2)(0,0)
\put(2.5,0){\circle{0.08}}
\put(3,0){\circle{0.08}}
\put(2.53,0.025){\line(1,0){0.44}}
\put(2.54,0){\line(1,0){0.42}}
\put(2.53,-0.025){\line(1,0){0.44}}
\put(2.7,-0.040){$<$}
\put(2.4,0.1){\footnotesize $\alpha_1^\vee$}
\put(3,0.1){\footnotesize $\alpha_2^\vee$}
\end{picture}
$$
\vskip 10pt

Let $Y=\langle \alpha_1^\vee, \alpha^\vee_2 \rangle$ be the cocharacter lattice of ${\rm G}_2$, where $\alpha_1^\vee$ is the short coroot. Let $Q$ be the Weyl-invariant quadratic on $Y$ such such $Q(\alpha_1^\vee)=1$ (thus $Q(\alpha_2^\vee)=3$). Then the bilinear form $B_Q$ is given by
$$
B_Q(\alpha_i^\vee, \alpha_j^\vee) =
\begin{cases}
2 & \text{if } i=j=1; \\
-3& \text{if } i=1, j=2; \\
6 & \text{if } i=j=2.
\end{cases}
$$
Covers of ${\rm G}_2$ are always saturated. A simple computation gives:
$$Y_{Q,n}=Y_{Q,n}^{\sct}=\Z (n_{\alpha_1}\alpha_1^\vee) \oplus \Z (n_{\alpha_2} \alpha_2^\vee),$$
where $n_{\alpha_2}=n/\text{gcd}(n, 3)$ and $n_{\alpha_1}=n$. Thus
$$\msc{X}_{Q,n} = (\Z/n_{\alpha_1}\Z)\alpha_1^\vee \oplus  (\Z/n_{\alpha_2} \Z)\alpha_2^\vee =
\begin{cases}
(\Z/n\Z) \alpha_1^\vee \oplus (\Z/n\Z)\alpha_2^\vee & \text{ if }  3\nmid n, \\
(\Z/n\Z) \alpha_1^\vee \oplus (\Z/m\Z)\alpha_2^\vee & \text{ if }  n = 3m.
\end{cases}
$$

The Weyl group $W=\langle w_{\alpha_1}, w_{\alpha_2} \rangle$ generated by $w_{\alpha_1}$ and $w_{\alpha_2}$ is the Dihedral group of order $12$. Again, write $w_i:= w_{\alpha_i}$ and $w_{j_1 j_2 .... j_r}:=w_{j_1} w_{j_2} .... w_{j_r}$. We have
$$w_1(\alpha_2^\vee) = 3\alpha_1^\vee + \alpha_2^\vee \text{ and }  w_2(\alpha_1^\vee) = \alpha_1^\vee + \alpha_2^\vee.$$
For any natural number $d$, we define
$$
\mfr{f}(d) =
\begin{cases}
1 & \text{ if $d$ is odd},\\
4 & \text{ if $d$ is even}. 
\end{cases}
$$
It is easy to compute $\chi_{\sigma_\msc{X}}$, the values are given in Table 7 and Table 8, for $3\nmid n$ and $n=3m$ respectively.

\begin{table}[!htbp]  \label{T7}
\caption{Character of $\sigma_{\msc{X}}$ for $\wt{\rm G}_2^{(n)}$ with $3\nmid n$}
\vskip 5pt
\begin{tabular}{|c|c|c|c|c|c|c|c|c|c|c|c|c|}
\hline
 & id  &  $w_1$ & $w_2$  & $w_{12}$ &  $w_{21}$  & $w_{121} $ &  $w_{212}$  &  $w_{1212}$ & $ w_{2121}$ & $ w_{12121}$ & $ w_{21212}$ & $ w_G$ \\
\hline
$ \chi_{\sigma_\msc{X} } $ & $n^2$ & $n$  & $n$ & $1$  & $1$ &  $n$ &  $n$  &  $1$ & $1$ & $n$ & $n$ &  $\mfr{f}(n)$ \\ 
\hline
\end{tabular}
\end{table}
\begin{table}[!htbp]  \label{T8}
\caption{Character of $\sigma_{\msc{X}}$ for $\wt{\rm G}_2^{(n)}$ with $n=3m$}
\vskip 5pt
\begin{tabular}{|c|c|c|c|c|c|c|c|c|c|c|c|c|}
\hline
 & id  &  $w_1$ & $w_2$  & $w_{12}$ &  $w_{21}$  & $w_{121} $ &  $w_{212}$  &  $w_{1212}$ & $ w_{2121}$ & $ w_{12121}$ & $ w_{21212}$ & $ w_G$ \\
\hline
$ \chi_{\sigma_\msc{X} } $ & $3m^2$ & $m$  & $3m$ & $1$  & $1$ &  $3m$ &  $m$  &  $3$ & $3$ & $m$ & $3m$ &  $\mfr{f}(m)$ \\ 
\hline
\end{tabular}
\end{table}

Again, we consider $\chi$ with $\Phi(\chi) =\Delta$, which gives
$$\JH(I(\chi)) = \set{\pi_\emptyset, \pi_\Delta, \pi_{\set{\alpha_1}}, \pi_{\set{\alpha_2}}   }.$$
Let $W_i =\set{\text{id}, w_i} \subset W$. By the proof of Theorem \ref{T:G-pm}, one has
$$\dim \Wh(\pi_{\set{\alpha_i}}) = \angb{ \sigma_{\msc{X}} }{ \varepsilon_{W_i} }_{W_i}  -  \angb{\sigma_{\msc{X}}}{ \varepsilon_W }_W.$$
This coupled with Proposition \ref{T:C-} and Proposition \ref{P:St-1} give that for $3\nmid n$ one has Table 9.
\begin{table}[!htbp]  \label{T9}
\caption{$\dim \Wh(\pi_\Gamma)$ for $\wt{\rm G}_2^{(n)}, 3\nmid n$}
\vskip 5pt
\begin{tabular}{|c|c|c|c|c|}
\hline
 & $\pi_\emptyset$  &  $\pi_{\set{\alpha_1}} $ & $\pi_{\set{\alpha_2}}$  & $\pi_\Delta$  \\
\hline
$ \dim \Wh(-) $ & $ \frac{n^2 + 6n + 4 + \mfr{f}(n)}{12} $ & $\frac{5n^2 - 4 - \mfr{f}(n)}{12}$  & $ \frac{5n^2 - 4 - \mfr{f}(n)}{12}$ & $\frac{n^2 -6n + 4 + \mfr{f}(n)}{12}$ \\ 
\hline
\end{tabular}
\end{table}

On the other hand, for $n=3m$ we have Table 10.
\begin{table}[!htbp]  \label{T10}
\caption{$\dim \Wh(\pi_\Gamma)$ for $\wt{\rm G}_2^{(n)}, n=3m$}
\vskip 5pt
\begin{tabular}{|c|c|c|c|c|}
\hline
 & $\pi_\emptyset$  &  $\pi_{\set{\alpha_1}} $ & $\pi_{\set{\alpha_2}}$  & $\pi_\Delta$  \\
\hline
$ \dim \Wh(-) $ & $ \frac{3m^2 + 12m + 8 + \mfr{f}(m)}{12} $ & $\frac{15m^2 + 6m - 8 - \mfr{f}(m)}{12}$  & $ \frac{15m^2 - 6m - 8 - \mfr{f}(m)}{12}$ & $\frac{3m^2 -12m + 8 + \mfr{f}(m)}{12}$ \\ 
\hline
\end{tabular}
\end{table}

In particular, we see that for $n=3m$, it is possible to have $\dim \Wh(\pi_{\set{\alpha_1}}) \ne \dim \Wh(\pi_{\set{\alpha_2}})$. This phenomenon does not occur for saturated covers of $\SL_3$ and $\Sp_4$.

\vskip 50pt


\end{document}